\documentclass[11pt]{amsart}
\usepackage{amsmath,amsfonts,amsthm,amscd,amssymb,graphicx}

\newtheorem{theorem}{Theorem}[section]

\newtheorem{lemma}[theorem]{Lemma}
\newtheorem{lem}[theorem]{Lemma}
\newtheorem{proposition}[theorem]{Proposition}
\newtheorem{prop}[theorem]{Proposition}

\newtheorem{cor}[theorem]{Corollary}
\newtheorem{example}[theorem]{Example}
\newtheorem{remark}[theorem]{Remark}
\newtheorem{definition}[theorem]{Definition}
\newtheorem{rem}[theorem]{Remark}
\newtheorem{defi}[theorem]{Definition}
\newtheorem{observation}[theorem]{Observation}

\def\eps{\varepsilon }
\def\D{\partial }
\newcommand{\na}{{\nabla}}
\renewcommand{\div}{{\rm div}}

\newcommand{\RR}{\mathbb{R}}
\newcommand{\cO}{\mathcal{O}}
\newcommand{\CC}{\mathbb{C}}

\newcommand{\cE}{\mathcal{E}}

\newcommand{\dt}{\frac{d}{dt}}

\newcommand{\CalB}{\mathcal{B}}
\newcommand{\CalF}{\mathcal{F}}

\newcommand{\CalO}{\mathcal{O}}

\newcommand{\CalS}{\mathcal{S}}

\newcommand{\BbbA}{{\mathbb A}}

\newcommand{\II}{{\mathbb I}}

\newcommand{\MM}{{\mathbb M}}

\newcommand{\const}{\text{\rm constant}}

\newcommand{\e}{{\epsilon}}

\newcommand{\Span}{{\rm span }  \,}
\newcommand{\Rank}{{\rm rank} \, }

\def\bb1{{1\!\!1}}

\def\CalO{\mathcal{O}}

\def\cP{\mathcal{P}}

\def\bU{{\bar{U}}}
\def\bW{{\bar{W}}}

\def\R{\Re e}
\def\I{\Im m}

\def\I{\Im m}

\def\cQ{\mathcal{Q}}
\def\cT{\mathcal{T}}
\def\cF{\mathcal{F}}
\def\cB{\mathcal{B}}
\def\cZ{\mathcal{Z}}

\def\diag{\mbox{diag}}

\newcommand{\iprod}[1]{\langle{#1}\rangle}
\newcommand{\wprod}[1]{\langle{#1}\rangle}

\begin{document}

\title[Stability of boundary layers]
{Long-time stability of large-amplitude noncharacteristic boundary
layers for hyperbolic--parabolic systems}
\author[T. Nguyen and K. Zumbrun]{Toan Nguyen and Kevin Zumbrun}

\date{Last Updated:  April 5, 2008}

\thanks{ This work was supported in part by the National Science Foundation award number DMS-0300487.}

\address{Department of Mathematics, Indiana University, Bloomington, IN 47402}
\email{nguyentt@indiana.edu}
\address{Department of Mathematics, Indiana University, Bloomington, IN 47402}
\email{kzumbrun@indiana.edu}

\begin{abstract}
Extending investigations of Yarahmadian and Zumbrun in
the strictly parabolic case,
we study time-asymptotic stability of arbitrary (possibly large)
amplitude noncharacteristic boundary layers of a class of
hyperbolic-parabolic systems including the Navier--Stokes equations of
compressible gas- and magnetohydrodynamics, establishing that
linear and nonlinear stability are both equivalent to an
Evans function, or generalized spectral stability, condition.
The latter is readily checkable numerically, and analytically
verifiable in certain favorable cases;
in particular, it has been shown by Costanzino, Humpherys,
Nguyen, and Zumbrun to hold for sufficiently large-amplitude
layers for isentropic ideal gas dynamics, with general adiabiatic
index $\gamma \ge 1$.
Together with these previous results, our results thus give
nonlinear stability of large-amplitude isentropic boundary layers,
the first such result for compressive (``shock-type'') layers
in other than the nearly-constant case.
The analysis, as in the strictly parabolic case,
proceeds by derivation of detailed pointwise Green function
bounds, with substantial new technical difficulties associated
with the more singular, hyperbolic behavior in the high-frequency/short
time regime.
\end{abstract}

\maketitle

\tableofcontents


\section{Introduction}
In this paper, we study the stability of boundary layers assuming
that the boundary layer solution is {\it noncharacteristic},
which means, roughly, that signals are transmitted into or
out of but not along the boundary.
In the context of gas dynamics or magnetohydrodynamics (MHD),
this corresponds to the situation of a porous boundary
with prescribed inflow or outflow conditions accomplished
by suction or blowing,
a scenario that has been suggested as a means to reduce
drag along an airfoil by stabilizing laminar flow;
see
Example \ref{aeroexam} below.

We consider a boundary layer, or stationary solution,
\begin{equation}\label{profile}
\tilde U=\bU(x), \quad \lim_{z\to +\infty} \bU(z)=U_+, \quad \bU(0)=\bar
U_0
\end{equation}
of a system of conservation laws on the quarter-plane
\begin{equation}\label{hyper-parabolic}
\tilde U_t +  F(\tilde U)_{x} = (B(\tilde U)\tilde U_{x})_{x}, \quad x,t>0,
\end{equation}
$\tilde U,F\in \mathbb{R}^n$, $B\in\mathbb{R}^{n \times n}$, with initial
data $\tilde U(x,0)=\tilde U_0(x)$ and Dirichlet type
boundary conditions specified in \eqref{inBC}, \eqref{outBC} below.
A fundamental question
connected to the physical motivations from aerodynamics
%
is whether or not such boundary layer
solutions are {\it stable} in the sense of PDE, i.e., whether or not
a sufficiently small perturbation of $\bU$ remains close to $\bU$,
or converges time-asymptotically to $\bU$, under the evolution of
\eqref{hyper-parabolic}.
That is the question we address here.

Our main result, in the general spirit of \cite{ZH,MaZ3,MaZ4,Z3,HZ,YZ}, is to
reduce the questions of linear and nonlinear stability to verification
of a simple and numerically well-posed {\it Evans function}, or
generalized spectral stability, condition,
which can then be checked either numerically
or by the variety of methods available for study of eigenvalue ODE;
see, for example,
\cite{Br1, Br2, BrZ, BDG, HuZ2, PZ, FS, BHRZ, HLZ, HLyZ1,HLyZ2, CHNZ}.
Together with the results of \cite{CHNZ}, this yields in particular
nonlinear stability of sufficiently large-amplitude boundary-layers of
the compressible Navier--Stokes equations of
isentropic ideal gas dynamics, with adiabatic index $\gamma \ge 1$,
the first such result for a large compressive, or ``shock-type'', boundary
layers.
The main new difficulty beyond the strictly parabolic case of \cite{YZ}
is to treat the more singular, hyperbolic behavior in the
high-frequency regime, both in obtaining pointwise Green function
bounds, and in deriving energy estimates by which the nonlinear
analysis is closed.

\subsection{Equations and assumptions.}
We consider the general hyperbolic-parabolic system of conservation
laws \eqref{hyper-parabolic} in conserved variable $\tilde U$, with
$$\tilde U = \begin{pmatrix}\tilde u\\
\tilde v\end{pmatrix}, \quad B=\begin{pmatrix}0 & 0 \\
b_1 & b_2\end{pmatrix}, \quad \sigma(b_2)\ge \theta>0,$$
$\tilde u\in \RR$, and $\tilde v\in \RR^{n-1}$, where,
here and elsewhere, $\sigma$ denotes spectrum of a
linearized operator or matrix.
Here for simplicity, we have restricted to the case (as in standard gas
dynamics and MHD) that the
hyperbolic part (equation for $\tilde u$) consists of a single
scalar equation.
As in \cite{MaZ3}, the results extend in
straightforward fashion to the case $\tilde u\in \RR^k$, $k>1$,
with $\sigma(A^{11})$ strictly positive or strictly negative.

Following \cite{MaZ4,Z3},
we assume that equations \eqref{hyper-parabolic} can be written,
alternatively, after a triangular change of coordinates
\begin{equation}\label{Wcoord}
\tilde W:=\tilde W(\tilde U) =\begin{pmatrix}\tilde w^I(\tilde u)\\
 \tilde w^{II}(\tilde u, \tilde v)\end{pmatrix},
\end{equation}
in {\em the
quasilinear, partially symmetric hyperbolic-parabolic form}
\begin{equation}\label{symmetric-form}\tilde A^0 \tilde W_t + \tilde A\tilde W_x = (\tilde B
\tilde W_x)_x + \tilde G,
\end{equation}
where, defining $\tilde W_+:=\tilde W(U_+)$,
\medskip

(A1) $\tilde A(\tilde W_+),\tilde A^0,\tilde A^{11}$ are symmetric,
$A^0$ block diagonal, $\tilde A^0\ge \theta_0>0$,
\medskip

(A2) no eigenvector of $\tilde A(\tilde A^0)^{-1}(\tilde W_+)$ lies in the
kernel of $\tilde B(\tilde A^0)^{-1}(\tilde W_+)$,
\medskip

(A3)
$\tilde B=\begin{pmatrix}0 & 0 \\ 0 & \tilde b\end{pmatrix}$,
$\tilde b\ge \theta>0$, and
$\tilde G=\begin{pmatrix}0\\\tilde g \end{pmatrix}$ with
$\tilde g(\tilde W_x,\tilde W_x)=\cO(|\tilde W_x|^2).$

Along with the above structural assumptions, we make the following
technical hypotheses:
\medskip

(H0) $F, B, \tilde A^0, \tilde A, \tilde B, \tilde W(\cdot),
\tilde g(\cdot,\cdot) \in C^{4}$.
\medskip

(H1) $\tilde A^{11}$ (scalar) is either strictly positive or strictly
negative, that is, either $\tilde A^{11}\ge \theta_1>0,$ or $\tilde
A^{11}\le -\theta_1<0$. (We shall 
call these cases {\em the
inflow case} or {\em the outflow case}, correspondingly.)
\medskip

(H2) The eigenvalues of $dF^{11}(U_+)$ are real, distinct, and nonzero.
\medskip

(H3) Solution $\bU$ is unique.
\medskip

\noindent Condition (H1) corresponds to noncharacteristicity, while
(H2) is the condition for the hyperbolicity of $U_+$. The
assumptions (A1)-(A3) and (H0)-(H3) are satisfied for gas dynamics
and MHD with van der Waals equation of state under inflow or outflow
conditions; see discussions in \cite{MaZ4,CHNZ,GMWZ5,GMWZ6}.

We also assume:

(B) Dirichlet boundary conditions in $\tilde W$-coordinates:
\begin{equation}\label{inBC}
(\tilde w^I, \tilde w^{II})(0,t)=\tilde h(t):=(\tilde h_1,\tilde h_2)(t)\end{equation}
for the inflow case, and
\begin{equation}\label{outBC}
\tilde w^{II}(0,t)=\tilde h(t)\end{equation} for the outflow case.
\\

This is sufficient for the main physical applications;
the situation of more general, Neumann- and mixed-type
boundary conditions on the parabolic variable $v$ can
be treated as discussed in \cite{GMWZ5,GMWZ6}.

\begin{example}\label{aeroexam}
\textup{
The main example we have in mind consists of {\it laminar solutions}
$(\rho, u, e)(x_1,t)$ of the compressible
Navier--Stokes equations
\begin{equation}
\label{NSeq}
\left\{ \begin{aligned}
 & \D_t \rho +  \div (\rho u) = 0
 \\
 &\D_t(\rho  u) + \div(\rho u^tu)+ \na p =
\eps \mu \Delta u + \eps(\mu+\eta) \nabla \div u
 \\
 &
 \D_t(\rho E) + \div\big( (\rho E  +p)u\big)=
\eps\kappa \Delta T +
\eps \mu \div\big( (u\cdot \nabla) u\big) \\
&
\qquad \qquad \qquad \qquad
\qquad \qquad
+ \eps(\mu+\eta) \nabla(u\cdot \div u),
 \end{aligned}\right.
\end{equation}
$x\in \RR^d$,
on a half-space $x_1>0$,
where $\rho$ denotes density, $u\in \RR^d$ velocity,
$e$ specific internal energy,
$E=e+\frac{|u|^2}{2}$ specific total energy,
$p=p(\rho, e)$ pressure, $T=T(\rho, e)$ temperature,
$\mu>0$ and $|\eta|\le \mu$ first and second coefficients of viscosity,
$\kappa>0$ the coefficient of heat conduction,
and $\eps>0$ (typically small) the reciprocal of the Reynolds number,
with no-slip {\it suction-type} boundary conditions on the velocity,
$$
 u_{j}(0, x_2, \dots, x_d)=0, \, j\ne 1
\quad \hbox{\rm and} \quad  u_1(0,x_2, \dots, x_d)= V(x)< 0,
$$
and prescribed temperature, $ T(0,x_2, \dots, x_d)= T_{wall}(x).  $
Under the standard assumptions $p_\rho$, $T_e>0$,
this can be seen to satisfy all of the hypotheses (A1)--(A3), (H0)--(H3);
indeed these are satisfied also under much weaker van der Waals gas
assumptions \cite{MaZ4,Z3,CHNZ,GMWZ5,GMWZ6}.
In particular, boundary-layer solutions are of noncharacteristic
type, scaling as
$(\rho, u, e)= (\bar \rho, \bar u, \bar e)(x_1/\eps)$, with layer thickness
$\sim \eps$ as compared to the $\sim \sqrt{\eps}$ thickness of the
characteristic type found for an impermeable boundary.
}

\textup{
This corresponds to the situation of an airfoil with microscopic holes
through which gas is pumped from the surrounding flow, the microscopic
suction imposing a fixed normal velocity while the macroscopic
surface imposes standard temperature conditions as in
flow past a (nonporous) plate.
This configuration was suggested by Prandtl and tested experimentally
by G.I. Taylor as a means to reduce drag by stabilizing laminar flow;
see \cite{S,Bra}.
It was implemented in the NASA F-16XL experimental aircraft
program in the 1990's with reported 25\% reduction in drag
at supersonic speeds \cite{Bra}.\footnote{
See also NASA site
http://www.dfrc.nasa.gov/Gallery/photo/F-16XL2/index.html}
Possible mechanisms for this reduction are smaller thickness
$\sim \eps<<\sqrt{\eps}$ of noncharacteristic boundary layers
as compared to characteristic type,
and greater stability, delaying the transition from laminar to
turbulent flow.
In particular, stability properties appear to be quite important for
the understanding of this phenomenon.
For further discussion,
including the related
issues of matched asymptotic expansion, multi-dimensional
effects, and more general boundary configurations, see \cite{GMWZ5}.
}
\end{example}

\begin{example}\label{shockpiece}
\textup{
For \eqref{NSeq}, or the general \eqref{hyper-parabolic},
a large class of boundary-layer solutions, sufficient for
the present purposes, may be generated as truncations
$ \bar u^{x_0}(x):= \bar u(x-x_0)$ of {\it standing shock solutions}
\begin{equation}\label{shocktype}
u=\bar u(x), \quad \lim_{x\to \pm \infty}\bar u(x)=u_\pm
\end{equation}
on the whole line $x\in \RR$, with boundary conditions
$\beta_h(t)\equiv \bar u(0)$ (inflow) or $\beta_h(t)\equiv \bar w^I(0)$
(outflow) chosen to match.
However, there are also many other boundary-layer solutions not connected with
any shock.
For more general catalogs of boundary-layer solutions of \eqref{NSeq},
see, e.g., \cite{MN,SZ,CHNZ,GMWZ5}.
}
\end{example}

\begin{lemma}[\cite{MaZ3,Z3,GMWZ5}]\label{lem-profile-decay}
Given (A1)-(A3) and (H0)-(H3), a standing wave solution \eqref{profile}
of \eqref{hyper-parabolic}, (B) satisfies
\begin{equation} \Big|(d/dx)^k(\bU - U_+)\Big|\le C
e^{-\theta x},\quad k=0,...,4, \end{equation} as $x\to +\infty$.
\end{lemma}

\begin{proof}
As in the shock case \cite{MaZ4,Z3}, this follows by the
observation that, under hypotheses (A1)-(A3) and (H0)-(H3),
$U_+$ is a hyperbolic rest point of the layer profile ODE.
See also \cite{GMWZ5}.
\end{proof}

\subsection{Main results.}
Linearizing the equations \eqref{hyper-parabolic}, (B) about the boundary
layer $\bU$, we obtain the linearized equation
\begin{equation}\label{linearized-eqs}
U_t=LU:=-(\bar AU)_x+(\bar BU_x)_x,
\end{equation}
where
$$
\bar B:= B(\bU), \quad \bar  AU:= dF(\bU)U - (dB(\bU)U)\bU_{x},
$$
with boundary conditions (now expressed in $U$-coordinates)
\begin{equation}\label{inBC-U}
(\partial \tilde W/\partial \tilde U)(\bar U_0)
U(0,t)=h(t):=\begin{pmatrix}
h_1\\h_2\end{pmatrix}(t)\end{equation} for the inflow case, and
\begin{equation}\label{outBC-U}
(\partial \tilde w^{II}/\partial \tilde U)(\bar U_0)
U(0,t)= h(t)\end{equation} for the outflow case,
where $(\partial \tilde W/\partial \tilde U)(\bar U_0)$
is constant and invertible,
\begin{equation}\label{linBCstructure}
(\partial \tilde w^{II}/\partial \tilde U)(\bar U_0)
= m \begin{pmatrix} \bar b_1&\bar b_2\end{pmatrix}(\bar U_0),
\end{equation}
(by (A1) and triangular structure \eqref{Wcoord})
is constant with $m\in \RR^{(n-1)\times (n-1)}$ invertible,
and $h:= \tilde h- \bar h$.

%
%
%

\begin{definition}\label{def-lin}
The boundary layer  $\bU$ is said to be
{\it linearly $X \to Y$ stable} if, for some $C>0$,
the problem \eqref{linearized-eqs} with initial data $U_0$ in $X$
and homogeneous boundary data $h\equiv 0$ 
has a unique global solution $U(\cdot,t)$
such that $|U(\cdot, t)|_Y\le C|U_0|_X$ for all $t$;
it is said to be {\it linearly asymptotically $X \to Y$
stable} if also $|U(\cdot, t)|_Y\to 0$ as $t\to \infty$.
\end{definition}

We define the following {\it stability criterion}, where
$D(\lambda)$ described below, denotes the Evans function associated
with the linearized operator $L$ about the layer, an analytic
function analogous to the characteristic polynomial of a
finite-dimensional operator, whose zeroes away from the essential
spectrum agree in location and multiplicity with the eigenvalues of
$L$:\medskip

(D) $\hbox{\rm There exist no zeroes of $D(\cdot)$ in the nonstable
half-plane $ Re \lambda \ge 0$}. $\medskip

As discussed, e.g., in \cite{R2,MZ1,GMWZ5,GMWZ6}, under assumptions (H0)-(H3),
this is equivalent to {\it strong spectral stability},
$\sigma(L)\subset \{Re \lambda < 0\}$, (ii) {\it transversality} of
$\bU$ as a solution of the connection problem in the associated
standing-wave ODE, and {\it hyperbolic stability} of an associated
boundary value problem obtained by formal matched asymptotics. See
\cite{GMWZ5,GMWZ6} for further discussions.

\begin{definition}\label{def-nonlin}
The boundary layer  $\bU$ is said to be
{\it nonlinearly $X \to Y$ stable}
if, for each $\eps>0$, the problem \eqref{hyper-parabolic}
with initial data $\tilde U_0$ sufficiently close to the profile $\bU$
in $|\cdot|_X$ has a unique global solution $\tilde U(\cdot,t)$ such that
$|\tilde U(\cdot, t)-\bU(\cdot)|_Y<\eps$
for all $t$; it is said to be {\it nonlinearly asymptotically $X \to Y$ stable}
if also $|\tilde U(\cdot, t)-\bU(\cdot)|_Y\to 0$ as $t\to \infty$.
We shall sometimes not explicitly define the norm $X$,
speaking instead of stability or asymptotic stability in $Y$
under perturbations satisfying specified smallness conditions.
\end{definition}
\medskip

Our first main result is as follows.

\begin{theorem}[Linearized stability]\label{theo-lin}
Assume (A1)-(A3), (H0)-(H3), and (B) with $|h(t)|\le E_0(1+t)^{-1}$.
Let $\bU$ be a boundary layer. Then linearized $L^1\cap L^p\to
L^1\cap L^p$ stability, $1\le p\le \infty$, is equivalent to (D). In
the case of stability, there holds also linearized asymptotic
$L_1\cap L^p\to L^p$ stability, $p>1$, with rate
\begin{equation} |U(\cdot,t)|_{L^p}\le C
(1+t)^{-\frac 12(1-1/p)} |U_0|_{L^1\cap L^p} + CE_0(1+t)^{-\frac
12(1-1/p)}.
\end{equation}
\end{theorem}

To state the pointwise nonlinear stability result, we need some
notations. Denoting by
\begin{equation}
a_1^+<a_2^+ < \cdots < a_n^+
\end{equation}
the eigenvalues of of the limiting convection matrix $A_+:=
dF(U_+)$,
\medskip
define
\begin{equation}\label{theta}
\theta(x,t):= \sum_{a_j^+>0}(1+t)^{-1/2}e^{-|x-a_j^+t|^2/Mt},
\end{equation}
\begin{equation}\label{psi1}
\begin{aligned}
\psi_1(x,t)&:= \chi(x,t)\sum_{a_j^+>0}
(1+|x|+t)^{-1/2} (1+|x-a_j^+t|)^{-1/2},\\
\end{aligned}
\end{equation}
and
\begin{equation}\label{psi2}
\begin{aligned}
\psi_2(x,t)&:= (1-\chi(x,t))(1+|x-a_n^+t|+t^{1/2})^{-3/2},
\end{aligned}
\end{equation}
where $\chi(x,t)=1$ for $x\in [0,a_n^+t]$ and $\chi(x,t)=0$
otherwise and $M>0$ is a sufficiently large constant.

For simplicity, we measure the boundary data by function
\begin{equation}\label{Bdry-out} \CalB_{h}(t):= \sum_{r=0}^2|(d/dt)^rh|^2\end{equation} for the outflow
case, and \begin{equation}\label{Bdry-in} \CalB_h(t):=
\sum_{r=0}^4|(d/dt)^rh_1|^2 +
\sum_{r=0}^2|(d/dt)^rh_2|^2\end{equation} for the inflow case.

Then, our next result is as follows.

\begin{theorem}[Nonlinear stability]\label{theo-nonlin}
Assuming (A1)-(A3), (H0)-(H3), (B), and the linear stability
condition (D), the profile $\bU$ is nonlinearly asymptotically
stable in $L^p\cap H^4$, $p>1$,
with respect to perturbations $U_0\in H^4$, $h\in C^4$ in
initial and boundary data satisfying
$$
\|(1+|x|^2)^{3/4}U_0\|_{H^4} \le E_0
\quad \hbox{\rm and }\quad
|\cB_h(t)|\le E_0 (1+|t|)^{-1/2}
$$
for $E_0$ sufficiently small. More precisely,
\begin{equation}\label{pointwise}
\begin{aligned}
|\tilde U(x,t)-\bU(x)|&\le C E_0
(\theta+\psi_1+\psi_2)(x,t),\\
|\tilde U_x(x,t)-\bU_x(x)|&\le C E_0 (\theta+\psi_1+\psi_2)(x,t),
\end{aligned}
\end{equation}
where $\tilde U(x,t)$ denotes the solution of
\eqref{hyper-parabolic} with
initial and boundary data $\tilde U(x,0)=\bU(x)+U_0(x)$ and
$\tilde U(0,t)=\bar U_0+h(t)$, yielding the sharp rates
\begin{align} &\label{Lp} \|\tilde U(x,t)-\bU(x)\|_{L^p}\le C E_0
(1+t)^{-\frac{1}{2}(1-\frac{1}{p})}, \quad 1\le p\le \infty,\\
&\label{H4} \|\tilde U(x,t)-\bU(x)\|_{L^1\cap H^4}\le C E_0
(1+t)^{-\frac{1}{4}}.\\\notag
\end{align}
\end{theorem}

\begin{remark} By the one dimensional Sobolev embedding, from the hypothesis on $U_0$, we automatically assume that $$\|U_0\|_{H^4} \le E_0, \quad |U_0(x)|+|U'_0(x)|\le E_0
(1+|x|)^{-3/2}.
$$
\end{remark}


A crucial step in establishing Theorems \ref{theo-lin} and
\ref{theo-nonlin} is to obtain pointwise bounds on the Green
function $G(x, t ; y)$ of the linearized evolution equations
\eqref{linearized-eqs} (more properly speaking, a distribution),
which we now describe. Let  $a_j^+, j=1,....,n$ denote the
eigenvalues of $A(+\infty)$, and $l_j^+,r_j^+$ associated left and
right eigenvectors, respectively, normalized so that $l_j^+r_k^+ =
\delta_{j}^k$. Eigenvalues $a_j(x)$, and eigenvectors
$l_j(x),r_j(x)$ correspond to large-time convection rates and modes
of propagation of the linearized model \eqref{linearized-eqs}.

Define time-asymptotic, scalar diffusion rates
\begin{equation}\label{beta} \beta_j^+:=(l_j B r_j)_+,\quad j = 1,...,n,\end{equation}
and local dissipation coefficient \begin{equation} \label{eta}
\eta_* :=- D_*(x)
\end{equation} where $$D_*(x):=A_{12}b_2^{-1}\Big[A_{21} - A_{22}b_2^{-1}b_1 + b_2^{-1}b_1 A_* + b_2 \partial_x (b_2^{-1}b_1)\Big](x)$$
is an effective dissipation analogous to the effective diffusion
predicted by formal, Chapman-Enskog expansion in the (dual)
relaxation case,
$$ A_*:=A_{11}-A_{12}b_2^{-1}b_1.
$$
Note that as a consequence of dissipativity, (A2), we obtain
\begin{equation}
\label{dissipative}\eta_*^+>0, \quad \beta_j^+ >0, \quad \mbox{for
all }j.
\end{equation}
We also define modes of propagation for the
reduced, hyperbolic part of system \eqref{linearized-eqs} as
\begin{equation} L_* = \begin{pmatrix} 1 \\0_{n-1} \end{pmatrix}
,\quad R_* = \begin{pmatrix} 1 \\-b_2^{-1}b_1
\end{pmatrix}\end{equation}

We define the Green function $G(x,t;y)$ of the linearized evolution
equations \eqref{linearized-eqs} with homogeneous boundary
conditions (more properly speaking, a distribution), by

(i) $(\partial_t -L_x)G=0$ in the distributional sense, for all
$x,y,t>0$;

(ii) $G(x, t;y)\rightarrow \delta(x-y)$ as $t\to 0$;

(iii) for all $y,t>0$, $\begin{pmatrix}\bar A_*
&0\\\bar b_1&\bar b_2\end{pmatrix}G(0,t;y)=\begin{pmatrix}*\\0
\end{pmatrix}$ where $*=0$ for the inflow case
$\bar A_*>0$ and $*$ is arbitrary for the outflow case $\bar A_*<0$, noting
that no boundary condition is needed to be prescribed on the
hyperbolic part.

By standard arguments as in \cite{MaZ3}, we have the spectral
resolution, or inverse Laplace transform formulae
\begin{equation} e^{Lt}f = \frac{1}{2\pi i}P.V. \int_{\eta -
i\infty} ^{\eta+i\infty} e^{\lambda t}(\lambda - L)^{-1}
fd\lambda\end{equation} and
\begin{equation} G(x,t;y)= \frac{1}{2\pi i} P.V. \int_{\eta -
i\infty} ^{\eta+i\infty} e^{\lambda t} G_\lambda(x,y) \, d\lambda
\end{equation}for any large positive $\eta$.

We prove the following pointwise bounds on the Green function $G(x,t;y)$.

\begin{proposition}\label{prop-Greenbounds} Under assumptions (A1)-(A3), (H0)-(H3), (B), and
(D), we obtain \begin{equation} G(x,t;y) = H(x,t;y) + \tilde
G(x,t;y),
\end{equation}
where
\begin{equation}
\begin{aligned}\label{Hterm}H(x,t;y) &= \frac{1}{2\pi} A_*(x)^{-1}
A_*(y)\delta_{x-\bar a_*t}(y)e^{-\int_y^x(\eta_*/A_*)(z)dz} R_*
L_*^{tr}\\&= \cO(e^{-\eta_0 t}) \delta_{x-\bar a_* t}(y)R_*
L_*^{tr},\end{aligned}
\end{equation}
and
\begin{equation}\label{Gbounds}
\begin{aligned}
|\partial_{x}^\gamma \partial_y^\alpha & \tilde G(x,t;y)|\le  Ce^{-\eta(|x-y|+t)}\\
& +\quad C(t^{-(|\alpha|+|\gamma|)/2}+ |\alpha| e^{-\eta|y|} +|\gamma|
e^{-\eta|x|}) \Big( \sum_{k=1}^n
t^{-1/2}e^{-(x-y-a_k^{+} t)^2/Mt} \\
&+\sum_{a_k^{+} < 0, \, a_j^{+} > 0} \chi_{\{ |a_k^{+} t|\ge |y| \}}
t^{-1/2} e^{-(x-a_j^{+}(t-|y/a_k^{+}|))^2/Mt}
 \Big), \\
\end{aligned}
\end{equation}
$0\le |\alpha|, |\gamma| \le 1$, for some $\eta$, $C$, $M>0$, where
indicator function $\chi_{\{ |a_k^{+}t|\ge |y| \}}$ is $1$ for
$|a_k^{+}t|\ge |y|$ and $0$ otherwise.

Here, the averaged convection rate $\bar a_* (x, t)$ in
\eqref{Hterm} denotes the time-averages over $[0, t]$ of $A_* (z)$
along backward characteristic paths $z_* = z_*(x,t)$ defined by
\begin{equation} \frac{dz_*}{dt} = A_*(z_*(x,t)), \quad z_*(t) =x.
\end{equation}
In all equations, $a_j^+$, $A_*$, $L_*,R_*$ are as defined just
above.
\end{proposition}

\subsection{Discussion and open problems}\label{discussion}

The stability of noncharacteristic boundary layers in
gas dynamics has been treated using energy estimates in,
e.g., \cite{MN,KNZ,R3}, for both ``compressive'' boundary layers
including the truncated shock-solutions \eqref{shocktype},
and for ``expansive'' solutions analogous to rarefaction
waves.
However, in the case of compressive waves, these and most subsequent analyses
were restricted to the {\it small-amplitude case}
\begin{equation}\label{smallcond}
\|\bar u- u_+\|_{L^1(\RR^+)} \, \hbox{\rm sufficiently small}.
\end{equation}
Examining this condition even for the special class \eqref{shocktype}
of truncated shock solutions, we find that it is extremely restrictive.

For, consider the one-parameter family $\bar u^{x_0}(x)=\bar u(x-x_0)$
of boundary-layers associated with a standing shock $\bar u$ of
amplitude $\delta:=|u_+-u_-|<<1$.
By center manifold analysis \cite{Pe},
$ \bar u-u_+\sim \delta e^{-c\delta x}, $ hence
$$
\|\bar u- u_+\|_{L^1(\RR^+)} \sim e^{-c\delta x}
\sim \frac{|u_+-u(0)|}{|u_+-u_-|}
$$
in fact measures {\it relative amplitude} with respect to the
amplitude $|u_+-u_-|$ of the background shock solution $\bar u$.
Thus, smallness condition \eqref{smallcond} requires that the
boundary layer consist of a small, nearly-constant piece of
the original shock.

The present results, extending results of \cite{YZ} in the strictly
parabolic case, remove this restriction, allowing applications
in principle to shocks of any amplitude.  In particular, in combination
with the spectral stability results obtained in \cite{CHNZ} by
asymptotic Evans function analysis, they yield stability of noncharacteristic
isentropic gas-dynamical layers of sufficiently {\it large} amplitude.
Together with further, numerical, investigations of \cite{CHNZ}
give strong evidence that in fact {\it all} noncharacteristic isentropic
gas layers are spectrally stable, independent of amplitude, which would
together with our results yield nonlinear stability.

Spectral stability of full (nonisentropic) gas layers may be investigated
numerically as for shocks in \cite{HLyZ1, HLyZ2}, in both one-
and multi-dimensions.
However, analytical results of \cite{SZ} show that in this case
instability is possible, even for ideal gas equation of state.
The numerical classification of stability for full gas dynamics,
and the extension of our present nonlinear stability results to
multi-dimensions,
are two interesting direction for further investigation.


Finally, we comment briefly on the difference between our
analysis and the earlier analysis \cite{YZ} carried out
by similar techniques based on the Evans function and
stationary phase estimates on the inverse Laplace transform
formula.
Our analysis is in the same spirit as, and borrows heavily
from this earlier work.
The main new issues are technical ones connected with the more
singular high-frequency/short-time behavior of hyperbolic-parabolic
equations as compared to the strictly parabolic equations considered
in \cite{YZ}.  In particular, linearized behavior in the
$u$ coordinate, $U=(u,v)$,
is essentially hyperbolic, governed for short times
approximately by the principle part
\begin{equation}\label{pp}
v_t + A_*(x) v_x=0, \quad A_*:=(A_0^{11})^{-1}A^{11}
\end{equation}
Thus, we may expect as in the whole-line analysis of hyperbolic-parabolic
equations in \cite{MaZ3}
that the associated Green function contain
a delta-function component transported along the hyperbolic characteristic
$$
dx/dt=A_*(x),
$$
with the difference that now we must consider also a possibly-complicated
interaction with the boundary.

A key point is that in fact this potential complication {\it does not
occur}.  For, in the special case occurring in continuum-mechanical
systems \cite{Z3} that all hyperbolic signals either enter or leave the
boundary, there is no such boundary interaction and no reflected signal.
For example, in the simple scalar example \eqref{pp}, the Green function
on the half-line with either homogeneous inflow ($A^{11}>0$) boundary
condition $v(0)=0$ or outflow ($A^{11}<0$) condition $v(0)$ arbitrary,
is by inspection exactly the whole-line Green function
$$
g(x,t;y)=\delta_{x-\bar at}(y)/A_*(x)
$$
restricted to the half-line $x,y>0$, where $\bar a$ is
the average over $[0,t]$ of $A_*(z_*(t))$ along the
backward characteristic path
$$
\frac{dz_*}{dt} = A_*(z_*(x,t)), \quad z_*(t) =x.
$$
Indeed, comparing the description of the homogeneous boundary-value
Green function in Proposition \ref{prop-Greenbounds}
with that of the whole-line Green function in
\cite{MaZ3}, we see that they are identical.
However, to prove this simple observation costs us considerable
care in the high-frequency analysis.

A further issue at the nonlinear level
is to obtain nonlinear damping estimates using energy
estimates as in \cite{MaZ4},
which are somewhat complicated by the presence of a boundary.
This is necessary to prevent a loss of derivatives in the
nonlinear iteration.

As in \cite{YZ}, we get stability also with respect to perturbations
in boundary data, something that was not accounted for in earlier works
on long-time stability.
We mention, finally, the works \cite{GR,MZ1,GMWZ5,GMWZ6} in one-
and multi-dimensions of a similar spirit but somewhat different
technical flavor on the related small
viscosity problem-- for example,
$\eps\to 0$ in \eqref{NSeq}-- which establish that the Evans condition
(or its multi-dimensional analog) is also sufficient for existence
and stability of matched asymptotic solution as viscosity goes to zero.




\section{Pointwise bounds on resolvent
kernel $G_\lambda$} In this section, we shall establish estimates on
resolvent kernel $G_\lambda(x,y)$.

\subsection{Evans function framework}
Before starting the analysis, we review the basic Evans function
methods and gap/conjugation lemma.

\subsubsection{The gap/conjugation lemma}\label{conjugation}
Consider a family of first order ODE systems on the half-line:
\begin{equation}\label{gen}
\begin{aligned}
W'&=\mathbb{A}(x,\lambda)W,
\quad \lambda \in \Omega \quad {\rm and} \quad x>0,\\
\mathbb{B}(\lambda)W&=0, \quad \lambda \in \Omega \quad {\rm and}
\quad x=0.
\end{aligned}
\end{equation}
These systems of ODEs should be considered as a generalized
eigenvalue equation, with $\lambda$ representing frequency. We
assume that the boundary matrix $\mathbb{B}$ is analytic in
$\lambda$ and that the coefficient matrix $\mathbb{A}$ is analytic
in $\lambda$ as a function from $\Omega$ into $L^\infty(x)$, $C^K$
in $x$, and approaches exponentially to a limit
$\mathbb{A}_+(\lambda) $ as $x\rightarrow \infty$, with uniform
exponentially decay estimates

\begin{equation}\label{h0}
|(\partial/\partial x)^k(\BbbA- \BbbA_+)| \le C_1e^{-\theta|x|/C_2},
\, \quad \text{\rm for } x>0, \, 0\le k\le K,
\end{equation}
$C_j$, $\theta>0$, on compact subsets of $\Omega $. Now we can state
a refinement of the ``Gap Lemma'' of \cite{GZ,KS}, relating
solutions of the variable-coefficient ODE to the solutions of its
constant-coefficient limiting equations
\begin{equation}\label{limgen}
 Z'=\mathbb{A}_+(\lambda)Z
\end{equation}
as $x\rightarrow  +\infty $.
\begin{lem}[Conjugation Lemma \cite{MZ1}]
Under assumption \eqref{h0}, there exists locally to any given
$\lambda_0\in \Omega $ a linear transformation
$P_+(x,\lambda)=I+\Theta_+(x,\lambda)$  on $x\ge 0$, $\Phi_+$
analytic in $\lambda$ as functions from $\Omega$ to $L^\infty
[0,+\infty)$, such that:
\medskip

(i) $|P_+|$ and their inverses are uniformly bounded, with
\begin{equation} \label{Theta}
|(\partial/\partial \lambda)^j(\partial/\partial x)^k \Theta_+ |\le
C(j) C_1 C_2 e^{-\theta |x|/C_2} \quad \text{\rm for } x>0, \, 0\le
k\le K+1,
\end{equation}
$j\ge 0$, where $0<\theta<1$ is an arbitrary fixed parameter, and
$C>0$ and the size of the neighborhood of definition depend only on
$\theta$, $j$, the modulus of the entries of $\BbbA$ at $\lambda_0$,
and the modulus of continuity of $\BbbA$ on some neighborhood of
$\lambda_0 \in \Omega $.
\smallskip
\\
(ii)  The change of coordinates $W:=P_+ Z$ reduces \eqref{gen} on
$x\ge 0$ to the asymptotic constant-coefficient equations
\eqref{limgen}. Equivalently, solutions of \eqref{gen} may be
conveniently factorized as
\begin{equation}
W=(I+ \Theta_+)Z_+,
\end{equation}
where $Z_+$ are solutions of the constant-coefficient equations, and
$\Theta_+$ satisfy bounds.
\end{lem}

\begin{proof}
As described in \cite{MaZ3}, for $j=k=0$ this is a straightforward
corollary of the gap lemma as stated in [Z.3], applied to the
``lifted'' matrix-valued ODE
$$
P'= \mathbb{A}_+P- P\mathbb{A}+ (\mathbb{A}-\mathbb{A}_+)P
$$
for the conjugating matrices $P_+$. The $x$-derivative bounds
$0<k\le K+1$ then follow from the ODE and its first $K$ derivatives.
Finally, the $\lambda$-derivative bounds follow from standard
interior estimates for analytic functions.
\end{proof}

\begin{defi}\label{consplit}
Following \cite{AGJ}, we define the {\it domain of consistent
splitting} for the ODE system $W'=\mathbb{A}(x,\lambda)W$ as the
(open) set of $\lambda $ such that the limiting matrix $\BbbA_+$ is
hyperbolic (has no center subspace) and the boundary matrix
$\mathbb{B}$ is full rank, with $\dim S_+=\Rank \mathbb{B}$.
\end{defi}

\begin{lem}\label{bases}
On any simply connected subset of the domain of consistent
splitting, there exist analytic bases $\{v_{1}, \dots, v_{k}\}^+$
and $\{v_{k+1}, \dots, v_{N}\}^+$ for the subspaces $S_+$ and $U_+$
defined in Definition \ref{consplit}.
\end{lem}

\begin{proof}
By spectral separation of $U_+$, $S_+$, the associated (group)
eigenprojections are analytic. The existence of analytic bases then
follows by a standard result of Kato; see \cite{Kat}, pp. 99--102.
\end{proof}

\begin{cor}\label{stablecor}
By the Conjugation Lemma  , on the domain of consistent splitting,
the stable manifold of solutions decaying as $x\to +\infty$ of
\eqref{gen} is
\begin{equation} \label{span}
\CalS^+:=\Span\{P_+v_1^+,\dots,  P_+ v_k^+\},
\end{equation}
where $W_+^j:=P_+v_j^+$ are analytic in $\lambda$ and $C^{K+1}$ in
$x$ for $\mathbb{A}\in C^K$.
\end{cor}

\subsubsection{Definition of the Evans Function} On any simply
connected subset of the domain of consistent splitting, let $W_1^+,
\dots, W_k^+=P_+v_1^+, \dots, P_+v_k^+$ be the
analytic basis described in Corollary \ref{stablecor} of the
subspace $\CalS^+$ of solutions $W$ of \eqref{gen} satisfying the
boundary condition $W\to 0$ at $+\infty$.
 Then, the {\it Evans function} for the ODE
systems $W'=\mathbb{A}(x,\lambda)W$ associated with this choice of
limiting bases is defined as the $k\times k$ Gramian determinant
\begin{equation}\label{deq1}
\begin{aligned}
D(\lambda)&:= \det \Big( \mathbb{B} W_{1}^+, \dots, \mathbb{B}W_k^+
\Big)_{|x=0, \lambda}\\
&=\det \Big( \mathbb{B}P_+ v_{1}^+, \dots, \mathbb{B}P_+ v_k^+
\Big)_{|x=0, \lambda}.\\
\end{aligned}
\end{equation}
\medskip

\begin{rem}
Note that $D$ is independent of the choice of $P_{+}$ as, by
uniqueness of stable manifolds, the exterior products (minors) $P_+
v_{1}^+\wedge \dots\wedge P_+ v_k^+$ are uniquely determined by
their behavior as $x\to + \infty$.
\end{rem}

\begin{prop}\label{2.4}
Both the Evans function and the subspace $\CalS^+$
are analytic on the entire simply connected subset of the domain of
consistent splitting on which they are defined. Moreover, for
$\lambda$ within this region, equation \eqref{gen} admits a
nontrivial solution $W\in L^2(x>0)$ if and only if $D(\lambda)=0$.
\end{prop}

\begin{proof}
Analyticity follows by uniqueness, and local analyticity of $P_{+}$,
$v_k^{+}$. Noting that the first $P_+v_j^+$ are a basis for the
stable manifold of \eqref{gen} at $x\to +\infty$, we find that the
determinant of $\mathbb{B}P_+v_j^+$ vanishes if and only if
$\mathbb{B}(\lambda)$ has nontrivial kernel on $\CalS_+(\lambda,0)$,
whence the second assertion follows.
\end{proof}

\begin{rem}
In the case (as here) that the ODE system  describes an eigenvalue equation
associated with an ordinary differential operator $L$, Proposition
\ref{2.4} implies that eigenvalues of $L$ agree in location with
zeroes of $D$. (Indeed, they agree also in multiplicity; see
\cite{GJ1,GJ2}; Lemma 6.1, \cite{ZH}; or Proposition 6.15 of
\cite{MaZ3}.)
\end{rem}

When $\ker \mathbb{B}$ has an analytic basis $v^0_{k+1}, \dots,
v^0_{N}$, for example, in the commonly occurring case, as here, that
$\mathbb{B}\equiv \const$, we have the following useful alternative
formulation. This is the version that we will use in our analysis of
the Green function and Resolvent kernel.

\begin{prop}\label{evansprop}
Let $v^0_{k+1}, \dots, v^0_{N}$ be an analytic basis of $\ker
\mathbb{B}$, normalized so that $\det \big( \mathbb{B}^*, v^0_{k+1},
\dots v^0_{N}\big) \equiv 1$. Then, the solutions $W^0_j$ of
\eqref{gen} determined by initial data $W^0_j(\lambda,0)=v^0_j$ are
analytic in $\lambda$ and $C^{K+1}$ in $x$, and
\begin{equation}\label{deq}
D(\lambda):= \det \Big( W_{1}^+, \dots,  W_k^+, W_{k+1}^0, \dots,
W_{N}^0 \Big)_{|x=0, \lambda}.
\end{equation}
\end{prop}
\begin{proof}
Analyticity/smoothness follow by analytic/smooth dependence on
initial data/parameters. By the chosen normalization, and standard
properties of Grammian determinants,
$$
D(\lambda)= \det \Big( W_{1}^+, \dots,
W_k^+, v_{k+1}^0, \dots, v_{N}^0 \Big)_{|x=0, \lambda},
$$
yielding \eqref{deq}.
\end{proof}

\subsubsection{The tracking/reduction lemma}\label{tracking}
Next, consider a family of systems
\begin{equation}\label{gen2}
\begin{aligned}
W'&=\mathbb{A}(x,p,\eps)W,
\quad p \in \cP, \, \eps\in \RR^+ \quad {\rm and} \quad x>0,\\
\mathbb{B}(p,\eps)W&=0, \quad \lambda \in \Omega \quad {\rm and}
\quad x=0
\end{aligned}
\end{equation}
parametrized by $p$, $\eps$, with $\eps\to 0$.
The main example we have in mind is \eqref{gen} with
$p=\lambda/|\lambda|$ and $\eps:=|\lambda|^{-1}$, in
the high-frequency regime $|\lambda|\to \infty$.
We assume further that by some coordinate change we can arrange that
\begin{equation}\label{blockdiag}
\begin{aligned}
\BbbA=\begin{pmatrix} M_+ & 0\\ 0 & M_-\end{pmatrix} + \Theta ,
\end{aligned}
\end{equation}
with
\begin{equation}\label{deltaeta}
|\Theta| \le \delta(\eps),
\quad \Re ( M_+-M_-)\ge 2\eta(\eps)+\alpha^\eps(x),
\end{equation}
$\|\alpha\|_{L^1(\RR^+)}$ uniformly bounded for all $\eps$ sufficiently
small, and
\begin{equation}\label{gaprel}
(\delta/\eta)(\eps)\to 0 \quad \hbox{\rm as}\quad \eps\to 0,
\end{equation}
where $\Re(Q):=(1/2)(Q+Q^*)$ denotes the symmetric part of a matrix $Q$.

Then, we have the following
analog of Lemma \ref{conjugation},
asserting that the approximately block-diagonalized equations
\eqref{gen2} may be converted by a smooth coordinate transformation
$$
\begin{pmatrix} I & \Theta^1\\
\Theta^2 & I\\
\end{pmatrix} \to I \quad \hbox{\rm as }\quad \eps \to 0
$$
to exactly diagonalized form with the same leading part $\MM$.

\begin{lem}[\cite{MaZ3}] \label{reduction}
Consider a system \eqref{blockdiag}, with $\tilde F\equiv 0$ and
$\delta/\eta\to 0$ as $\eps\to 0$. Then, (i) for all $0<\epsilon\le
\epsilon_0$, there exist (unique) linear transformations
$\Phi_1^\epsilon(z,p)$ and $\Phi_2^\epsilon(z,p)$, possessing the
same regularity with respect to the various parameters $z$, $p$,
$\epsilon$ as do coefficients $M_\pm$ and $\Theta$, for which the
graphs $\{(Z_1, \Phi^\epsilon_2 Z_1)\}$ and $\{(\Phi^\epsilon_1 Z_2
,Z_2)\}$ are invariant under the flow of \eqref{blockdiag}, and
satisfying
$$ |\Phi^\epsilon_1|, \, |\Phi^\epsilon_2| \le C
\delta(\epsilon)/\eta(\epsilon) \, \text{\rm for all } z.
$$
In particular, (ii) the subspace $E_-$ of data at $z=0$ for which
the solution decays as $z\to +\infty$, given by
$\Span\{(\Phi^\eps_1(0,p)v, v)\}$, converges as $ \eps \to 0$ to
$\tilde E_-:=\Span \{(0, v)\}$.
\end{lem}

\medskip
{\bf Proof.} Standard contraction mapping argument carried out on
the ``lifted'' equations governing the flow of the conjugating
matrices $\Phi^\eps_j$; see Appendix C, \cite{MaZ3}.

\begin{remark}\label{alphaeps}
\textup{
In practice, we usually have $\alpha^\eps\equiv 0$,
as can be obtained in general by a change of coordinates
multiplying the first coordinate
by exponential weight $e^{\int \alpha^\eps dx}$.
}
\end{remark}

\subsection{Construction of the resolvent kernel}
In this section we construct the explicit form of the resolvent
kernel, which is nothing more than the Green function
$G_\lambda(x,y)$ associated with the elliptic operator $(L-\lambda
I)$, where
\begin{equation}\label{Gdef}
(L-\lambda I)G_\lambda(\cdot,y)=\delta_yI, \quad
\begin{pmatrix}\bar A_*&0\\\bar b_1&\bar b_2\end{pmatrix}
G_\lambda(0,y)\equiv
\begin{pmatrix}*\\0\end{pmatrix}
\end{equation}
where $*=0$ for the inflow case
and is arbitrary for the outflow case.

Let $\Lambda$ be the region of consistent splitting for $L$. It is
a standard fact (see, e.g., [He]) that the resolvent $(L-\lambda
I)^{-1}$ and the Green function $G_\lambda(x,y)$ are meromorphic in
$\lambda$ on $\Lambda$, with isolated poles of finite order.

Writing the associated eigenvalue equation $LU-\lambda U =0$ in the
form of a first-order system \eqref{gen} as follows: $W:=(u,v,z)\in
\CC^{2n-1}$ with $z:=b_1 u' + b_2 v'$, and
\begin{align}\label{firstorderW}\notag u' &= A^{-1}_* (-A_{12}b_2^{-1}z - (A'_{11}+\lambda)u -
A'_{12}v),\\ v'&= b_2^{-1}z - b_2^{-1}b_1 u',\\ z'&= (A_{21} -
A_{22}b_2^{-1}b_1)u' + A_{22}b_2^{-1}z + A_{21}'u + (A_{22}'
+\lambda)v. \notag\end{align}

\subsubsection{Domain of consistent splitting} Define
\begin{equation}\label{Lambda}\Lambda := \cap \Lambda^+_j, \quad j = 1,2,...,n\end{equation} where $\Lambda_j^+$ denote the open sets bounded on the left by the
algebraic curves $\lambda_j^+(\xi)$ determined by the eigenvalues of
the symbols $-\xi^2 B_+ - i\xi A_+$ of the limiting
constant-coefficient operators \begin{equation}L_+ w: = B_+ w'' -
A_+ w'\end{equation} as $\xi$ is varied along the real axis. The
curves $\lambda_j^+$ comprise the essential spectrum of operators
$L_+$.
\begin{lemma}[\cite{MaZ3}] The set $\Lambda$ is equal to the component containing real $ +\infty$ of the domain
of consistent splitting for \eqref{firstorderW}. Moreover, if
assumption (H3) holds, then
\begin{equation} \Lambda \subset \{ \lambda~:~\R\lambda > - \eta
|\I\lambda|/(1+|\I\lambda|), \quad \eta>0.\end{equation}
\end{lemma}

\subsubsection{Basic construction}
We first recall the following duality relation derived for the
degenerate viscosity case in \cite{MaZ3}.

\begin{lem}[\cite{ZH,MaZ3}]\label{duality}
The function $W = (U,Z)$ is a solution of \eqref{firstorderW} if and
only if $\tilde W^* \tilde \CalS W \equiv$ constant for any solution
$\tilde W = (\tilde U,\tilde Z)$ of the adjoint eigenvalue equation,
where
\begin{equation}\tilde \CalS = \begin{pmatrix} -A_{11} & -A_{12} & 0 \\
-A_{21} & -A_{22} & I_r \\-b_{2}^{-1}b_1 & -I_r &
0\end{pmatrix}\end{equation} and
\begin{equation}Z= (b_1,b_2)U', \quad \tilde Z= (0,b_2^* )\tilde
U'.\end{equation}
\end{lem}

For future reference, we note the representation \begin{equation}
{\tilde \CalS}^{-1} = \begin{pmatrix}-A_*^{-1} & 0& A_*^{-1}A_{12}
\\b_2^{-1}b_1A_*^{-1} & 0& -b_2^{-1}b_1 A_*^{-1}A_{12} -I_r\\-\tilde
AA_*^{-1} & I_r & -A_{22}+\tilde A A_*^{-1}A_{12}
\end{pmatrix}\end{equation}
where $\tilde A:=A_{21} - A_{22}b_{-1}b_1$,
$A_*:=A_{11}-A_{12}b_2^{-1}b_1$,
obtained by direct computation in \cite{MaZ3}.

Denote by
\begin{equation}\label{phi0}
\Phi^0= ( \phi^0_{k+1}(x;\lambda), \cdots , \phi^0_{n+r}(x;\lambda)
),\end{equation}\begin{equation} \label{phi+} \Phi^+=
(\phi^+_1(x;\lambda), \cdots ,\phi^+_{k}(x;\lambda)=
(P_+v_1^+,\cdots, P_+v_k^+),
\end{equation} and \begin{equation} \label{phi} \Phi=
(\Phi^+,\Phi^0),
\end{equation} the matrices whose columns span the subspaces of
solutions of \eqref{gen} that, respectively, decay at $x=+\infty$,
and satisfy the prescribed boundary conditions at $x=0$,
denoting (analytically chosen) complementary subspaces by
\begin{equation}\label{psi0}
\Psi^0= ( \psi^0_1(x;\lambda) , \cdots , \psi^0_{k}(x;\lambda)),
\end{equation}
\begin{equation}\label{psi+}
\Psi^+= ( \psi^+_{k+1}(x;\lambda) , \cdots ,
\psi^+_{n+r}(x;\lambda))
\end{equation}
and \begin{equation}\label{psi} \Psi= ( \Psi^0, \Psi^+).
\end{equation}

As described in the previous subsection, eigenfunctions decaying at
$+\infty$ and satisfying the prescribed boundary conditions
at $0$
 occur precisely when the subspaces $\Span \Phi^0$
and $\Span \Phi^+$ intersect, i.e., at zeros of the Evans function
defined in \eqref{deq}:
\begin{equation}
D_L(\lambda):=\det(\Phi^0,\Phi^+)_{|x=0}.
\end{equation}

Define the solution operator from $y$ to $x$ of $(L-\lambda)U =0$,
denoted by $\CalF^{y\rightarrow x}$, as $$\CalF^{y\rightarrow x} =
\Phi (x,\lambda)\Phi^{-1}(y,\lambda)$$ and the projections
$\Pi_y^0,\Pi_y^+$ on the stable manifolds at $0,+\infty$ as
$$\Pi_y^+ =\begin{pmatrix}\Phi^+(y) & 0
\end{pmatrix}\Phi^{-1}(y), \quad \Pi_y^0 =\begin{pmatrix} 0& \Phi^0(y)
\end{pmatrix}\Phi^{-1}(y).$$

With these preparations, the construction of the Resolvent kernel
goes exactly as in the construction performed in \cite{ZH,MaZ3} on
the whole line.

\begin{lem}\label{G-rep}
We have the the representation
\begin{equation}
G_\lambda (x,y)=
\begin{cases}
   (I_n, 0)\CalF^{y\rightarrow x}\Pi_y^+ \tilde
   S^{-1}(y)(I_n,0)^{tr},\quad & for \quad x>y,\\
-(I_n, 0)\CalF^{y\rightarrow x}\Pi_y^0 \tilde
   S^{-1}(y)(I_n,0)^{tr},\quad &for \quad x<y.
\end{cases}
\end{equation}
Moreover, on any compact subset $K$ of $\rho(L)\cap \Lambda$,
\begin{equation}\label{Inter-est}|G_\lambda (x,y)|\le C e^{\eta |x-y|},\end{equation} where $C>0$ and $\eta>0$
depend only on $K,L$.
\end{lem}

We define also the dual subspaces of solutions of $(L^* - \lambda
^*)\tilde W =0$. We denote growing solutions
\begin{equation}
\tilde{\Phi}^0=
\begin{pmatrix}
\tilde{\phi}^0_1(x;\lambda) & \cdots &
          \tilde{\phi}^0_{k}(x;\lambda)
\end{pmatrix},
\end{equation}
\begin{equation}\label{tildephi+}
\tilde{\Phi}^+=
\begin{pmatrix}
\tilde{\phi}^+_{k+1}(x;\lambda) & \cdots &
          \tilde{\phi}^+_{n+r}(x;\lambda)
\end{pmatrix},
\end{equation} $\tilde{\Phi}:=(\tilde{\Phi}^0,\tilde{\Phi}^+)$ and decaying solutions
\begin{equation}
\tilde{\Psi}^0=
\begin{pmatrix}
\tilde{\psi}^0_1(x;\lambda) & \cdots &
          \tilde{\psi}^+_{k}(x;\lambda)
\end{pmatrix},
\end{equation}
\begin{equation}\label{tildepsi+}
\tilde{\Psi}^+=
\begin{pmatrix}
\tilde{\psi}^+_{k+1}(x;\lambda) & \cdots &
          \tilde{\psi}^+_{n+r}(x;\lambda),
\end{pmatrix}
\end{equation} and $\tilde{\Psi}:=(\tilde{\Psi}^0,\tilde{\Psi}^+)$,
satisfying the relations $$(\tilde \Psi \tilde \Phi)_{0,+}^* \tilde
S (\Psi \Phi)_{0,+} \equiv I.$$

Then, we have
\begin{prop}\label{G-rep1} The resolvent kernel may alternatively be expressed as
 \begin{equation}
  G_\lambda(x,y)=
  \begin{cases}
   (I_n,0)\Phi^+(x;\lambda)M^+(\lambda)\tilde\Psi^{0*}(y;\lambda)(I_n,0)^{tr}\ &x>y,\\
   -(I_n,0)\Phi^0(x;\lambda)M^0(\lambda)\tilde\Psi^{+*}(y;\lambda)(I_n,0 )^{tr}\ &x<y,
                 \end{cases}
 \end{equation}
where
\begin{equation}\label{Mexp}
  M(\lambda):=\text{\rm diag}(M^+(\lambda),M^0(\lambda))=
  \Phi^{-1}(z;\lambda)\bar \CalS^{-1}(z)\tilde\Psi^{-1*}(z;\lambda).
\end{equation}
\end{prop}

From Proposition \ref{G-rep1}, we obtain the following scattering
decomposition, generalizing the Fourier transform representation in
the constant-coefficient case

\begin{cor}\label{G-rep2} On $\Lambda \cap \rho (L)$,
\begin{equation}\label{G-rep3} G_\lambda(x,y)=
\sum_{j,k}d_{jk}^+\phi_j^+(x;\lambda)\tilde \psi_k^+(y;\lambda)^* +
\sum_k \phi^+_k(x;\lambda)\tilde
\phi_k^+(y;\lambda)^*\end{equation}for $0\le y\le x$, and
\begin{equation}\label{G-rep4} G_\lambda(x,y)=
\sum_{j,k}d_{jk}^0(\lambda)\phi_j^+(x;\lambda)\tilde
\psi_k^+(y;\lambda)^* + \sum_k \psi_k^+(x;\lambda)\tilde
\psi_k^+(y;\lambda)^*\end{equation}for $0\le x\le y$, where
$d_{jk}^{0,+}(\lambda)=\CalO(\lambda^{-K})$ are scalar meromorphic
functions with pole of order $K$ less than or equal to the
order to which the Evans function $D(\lambda)$ vanishes
at $\lambda=0$ (note that $K=0$ under assumption (D)).
\end{cor}

\begin{proof}
Matrix manipulation of expression \eqref{Mexp}, Kramer's rule, and the
definition of the Evans function; see \cite{MaZ3}.
\end{proof}

\begin{remark}\label{G-rep-constv} In the constant-coefficient case,
with a choice of common bases $\Psi^{0,+} = \Phi^{+,0}$ at
$0,+\infty$, the above representation \eqref{G-rep2} reduces to the
simple formula \begin{equation}G_\lambda(x,y) = \begin{cases}
   \sum_{j=k+1}^N\phi_j^+(x;\lambda)\tilde \phi_j^{+*}(y;\lambda) &x>y,\\
   -\sum_{j=1}^k\psi_j^+(x;\lambda)\tilde \psi_j^{+*}(y;\lambda) &x<y.
                 \end{cases}
\end{equation}
\end{remark}

\subsection{High frequency estimates}

We now turn to the crucial estimation
of the resolvent kernel in the high-frequency regime $|\lambda|\to
+\infty$,
following the general approach of \cite{MaZ3}.
Define sectors \begin{equation}\label{sectorP}
\Omega_P:=\{\lambda~:~\R\lambda\ge -\theta_1 |\I\lambda| +
\theta_2\}, \quad\theta_j >0.\end{equation}
and
\begin{equation} \Omega:=\{\lambda ~:~-\eta_1 \le \R\lambda\}
\end{equation} with $\eta_1$ sufficiently small such that $\Omega
\setminus B(0,r)$ is compactly contained in the set of consistent
splitting $\Lambda$, for some small $r$ to be chosen later.
Then, we have the following crucial result analogous to
the estimates on the whole line performed in \cite{MaZ3}.

\begin{proposition} \label{prop-high-est} Assume that (H0)-(H3) hold. Then for any $r>0$
and $\eta_1 = \eta_1(r)>0$ chosen sufficiently small such that
$\Omega \setminus B(0,r) \subset \Lambda \cap \rho (L)$. Moreover
for $R>0$ sufficiently large, the following decomposition holds on
$\Omega \setminus B(0,R)$:
\begin{equation}G_\lambda(x,y) = H_\lambda(x,y) + P_\lambda(x,y) +
\Theta_\lambda^H(x,y) + \Theta_\lambda^P(x,y),\end{equation} where
\begin{equation}H_\lambda(x,y) = \begin{cases}
   \chi_{\{A_*>0\}}A_*(x)^{-1}e^{\int_y^x(-\lambda/A_* - \eta_*/A_*)(z)dz}R_*L_*^{tr}&x>y,\\
   \chi_{\{A_*<0\}}A_*(x)^{-1}e^{\int_y^x(-\lambda/A_* - \eta_*/A_*)(z)dz}R_*L_*^{tr}&x<y,
                 \end{cases}
\end{equation} and \begin{equation}\begin{aligned}
\Theta_\lambda^H(x,y) &= \lambda^{-1}B_\lambda(x,y;\lambda) +
\lambda^{-1}(x-y)C_\lambda(x,y;\lambda),\\\Theta_\lambda^P(x,y)&=
\lambda^{-2}D_\lambda(x,y;\lambda)\end{aligned}\end{equation} where
\begin{equation}\label{BC-est}B_\lambda(x,y) =C_\lambda(x,y)= \begin{cases}
   \chi_{\{A_*>0\}}e^{-\int_y^x\lambda/A_*(z)dz}b_*(x,y)&x>y,\\
   \chi_{\{A_*<0\}}e^{-\int_y^x\lambda/A_*(z)dz}b_*(x,y)&x<y,
                 \end{cases}
\end{equation}with \begin{equation} b_*:=e^{\int_y^x(-\eta_*/A_*)(z)dz} =
\cO(e^{-\theta|x-y|}),\end{equation} due to \eqref{dissipative}, and
\begin{equation}\label{D-est}D_\lambda(x,y;\lambda) = \cO(e^{-\theta(1+Re \lambda)|x-y|}+
e^{-\theta |\lambda|^{1/2}|x-y|}),\end{equation} for some uniform
$\theta>0$ independent of $x, y, z,$ each described term separately
analytic in $\lambda$, and $P_\lambda$ is analytic in $\lambda$ on a
(larger) sector $\Omega_P$ as in \eqref{sectorP}, with $\theta_1$
sufficiently small, and $\theta_2$ sufficiently large, satisfying
uniform bounds
\begin{equation}\label{Plamda-est} (\partial /\partial
x)^\alpha(\partial/\partial y)^{\beta} P_\lambda(x,y) =
\cO(|\lambda|^{(|\alpha|+|\beta|-1)/2})
e^{-\theta|\lambda|^{1/2}|x-y|}, \quad \theta>0,\end{equation} for
$|\alpha|+|\beta| \le 2$ and $0\le |\alpha|,|\beta|\le 1.$

Likewise, the following derivative bounds also hold:
\begin{align*}
(\partial/\partial x)\Theta_\lambda(x,y) =&\Big(B_x^0(x,y;\lambda) + (x-y)C_x^0(x,y;\lambda)\Big) + \lambda^{-1}\Big(B_x^1(x,y;\lambda)\\&+(x-y)C_x^1(x,y;\lambda) +(x-y)^2D_x^1(x,y;\lambda)\Big)\\&+\lambda^{-3/2}E_x(x,y;\lambda)
\end{align*}
and
\begin{align*}
(\partial/\partial y)\Theta_\lambda(x,y) =&\Big(B_y^0(x,y;\lambda) + (x-y)C_y^0(x,y;\lambda)\Big) + \lambda^{-1}\Big(B_y^1(x,y;\lambda)\\&+(x-y)C_y^1(x,y;\lambda) +(x-y)^2D_y^1(x,y;\lambda)\Big)\\&+\lambda^{-3/2}E_y(x,y;\lambda)
\end{align*}
where $B_\beta^\alpha$, $C_\beta^\alpha$, and $D_\beta^1$ satisfy bounds of the form \eqref{BC-est}, and $E_\beta$ satisfies a bound of the form \eqref{D-est}.\end{proposition}

\begin{proof}
We shall follow closely the argument in \cite{MaZ3},
with the new feature of boundary treatments,
or estimates of $\Phi^0,\Psi^0$.
Writing the associated eigenvalue equation $LU-\lambda U =0$ in the
form of a first-order system as follows: $W:=(u,v,z)\in \CC^{2n-1}$
with $z:=b_1 u' + b_2 v'$, and
\begin{align}\label{firstorderWe}\notag u' &= A^{-1}_* (-A_{12}b_2^{-1}z - (A'_{11}+\lambda)u -
A'_{12}v),\\ v'&= b_2^{-1}z - b_2^{-1}b_1 u',\\ z'&= (A_{21} -
A_{22}b_2^{-1}b_1)u' + A_{22}b_2^{-1}z + A_{21}'u + (A_{22}'
+\lambda)v \notag\end{align} or \begin{equation} W' = A
W.\end{equation}

Recall from Lemma \ref{G-rep} that we have the the representation
\begin{equation}\label{repW}
G_\lambda (x,y)=
\begin{cases}
   (I_n, 0)\CalF_W^{y\rightarrow x}\Pi_W^+ (y)\tilde
   S^{-1}(y)(I_n,0)^{tr},\quad & for \quad x>y,\\
-(I_n, 0)\CalF_W^{y\rightarrow x}\Pi_W^0(y) \tilde
   S^{-1}(y)(I_n,0)^{tr},\quad &for \quad x<y.
\end{cases}
\end{equation}

We shall find it more convenient to use the ``local'' coordinates
$\tilde u: = A_* u, \tilde v := b_1 u + b_2 v$. yielding from
\eqref{firstorderW}:
\begin{align}\label{firstorderY1}\notag \tilde u_x &= -\lambda A^{-1}_* \tilde u - (A_{12}b_2^{-1}\tilde v)_x \\ (\tilde v_x)_x & = \Big[
((A_{21} - A_{22}b_2^{-1}b_1 + b_2
\partial_x(b_2^{-1}b_1))A_*^{-1}\tilde u)_x\\&\quad + ((A_{22}+ \partial_x(b_2)b_2^{-1})\tilde v)_x + \lambda b_2^{-1}b_1 A_*^{-1}\tilde u + \lambda b_2^{-1}\tilde v\Big]\notag.\end{align}

Following standard procedure (e.g., \cite{AGJ,GZ,ZH,MaZ3}),
performing the rescaling
\begin{equation}\tilde x: = |\lambda| x, \quad \tilde \lambda: = \lambda / |\lambda|, \end{equation}
and changing coordinates $W \mapsto Y = \cQ W$, where
\begin{equation}Y = (\tilde u, \tilde v, \tilde v_x)^{tr} = (A_* u, b_1 u
+ b_2 v, (b_1 u + b_2 v)_x)^{tr},\end{equation} \begin{equation} \cQ
=\begin{pmatrix} A_* & 0 & 0\\b_1 & b_2 &
0\\|\lambda|^{-1}\partial_x b_1 & |\lambda|^{-1}\partial_x b_2 &
|\lambda|^{-1}I_r,
\end{pmatrix}
\end{equation}
and \begin{equation} \cQ^{-1} =\begin{pmatrix} A_*^{-1} & 0 &
0\\-b_2^{-1}b_1 A_*^{-1}& b_2^{-1} &
0\\-|\lambda|b_2\partial_x(b_2^{-1} b_1) A_*^{-1} &
-|\lambda|\partial_x (b_2)b_2^{-1} & |\lambda|I_r,
\end{pmatrix}
\end{equation} we obtain the first
order equations
\begin{equation}\label{firstorderY} Y' = A(\tilde x, |\lambda|^{-1})
Y, \quad Y:=(\tilde u, \tilde v, \tilde v')^{tr}, \quad ' :=
\partial_{\tilde x}\end{equation}
where \begin{equation}A(\tilde x, |\lambda|^{-1}) = A_0(\tilde x) +
|\lambda|^{-1}A_1(\tilde x) + \cO(|\lambda|^{-2}),\end{equation}
with \begin{equation} \begin{aligned} A_0(\tilde x) &=
\begin{pmatrix} -\tilde \lambda A_*^{-1} & 0 & - A_{12}b_2^{-1}
\\0&0&I_r
\\0&0&0\end{pmatrix} \\ A_1(\tilde x)&= \begin{pmatrix} 0 & -\partial_x(A_{12}b_2^{-1}) & 0
\\0&0&0
\\-\tilde\lambda d_*A_*^{-2}&\tilde \lambda
b_2^{-1}&e_*b_2^{-1}\end{pmatrix}\end{aligned}\end{equation}
\begin{equation}\begin{aligned}d_*&:= A_{21} - A_{22} b_2^{-1}b_1 - b_2^{-1}b_1 A_* + b_2 \partial_x(b_2^{-1}b_1),\\e_*&:=A_{22} + d_* A_*^{-1}A_{12} + \partial_x(b_2)
.\end{aligned}
\end{equation}

We will carry out the details of the lower-order estimates in
Proposition \ref{prop-high-est}, leaving high-order estimates and
derivative bounds as brief remarks at the end. First, observe that
the representation \eqref{repW} becomes
\begin{equation}\label{repY}
G_\lambda (x,y)=
\begin{cases}
   (I_n, 0)\cQ^{-1}\CalF_Y^{y\rightarrow x}\Pi_Y^+ (y)\cQ\tilde
   S^{-1}(y)(I_n,0)^{tr},\quad & for \quad x>y,\\
-(I_n, 0)\cQ^{-1}\CalF^{y\rightarrow x}_Y\Pi_Y^0 (y)\cQ \tilde
   S^{-1}(y)(I_n,0)^{tr},\quad &for \quad x<y
\end{cases}
\end{equation}
where $\Pi_Y^{0,+}$ and $\cF_Y^{y\to x}$ denote projections and
flows in $Y-$coordinates.

\subsubsection{Initial diagonalization.}
Applying the formal iterative diagonalization
procedure described in \cite[Proposition 3.12]{MaZ3}, one obtains the
approximately block-diagonalized system
\begin{align} Z' &= D(\tilde x, |\lambda|^{-1})Z, \quad TZ:=Y, \quad
D:=T^{-1}AT, \\T(\tilde x,|\lambda|^{-1})&= T_0(\tilde x) +
|\lambda|^{-1}T_1(\tilde x) + \cdots +|\lambda|^{-3}T_3(\tilde x) \\
D(\tilde x, |\lambda|^{-1})&= D_0(\tilde x) +
|\lambda|^{-1}D_1(\tilde x) + \cdots + D_3(\tilde
x)|\lambda|^{-3}+\cO(|\lambda|^{-4}),
\end{align} where without loss of generality (since $T_0$ is uniquely determined up to a constant linear coordinate change)
\begin{equation}
T_0:=\begin{pmatrix} 1 & 0 &-\tilde \lambda ^{-1}A_* A_{12}b_2^{-1}
\\0&I_r&0 \\0&0&I_r\end{pmatrix} , \quad T_0^{-1}=\begin{pmatrix} 1 & 0 &\tilde \lambda^{-1} A_* A_{12}b_2^{-1}
\\0&I_r&0 \\0&0&I_r\end{pmatrix}\end{equation}
and \begin{equation} D_0:=\begin{pmatrix} -\tilde \lambda A_*^{-1} &
0 &0\\0&0&I_r
\\0&0&0\end{pmatrix} , \quad D_1:=\begin{pmatrix} -\eta_* A_*^{-1} & 0
&0
\\0&0&0 \\0&\tilde\lambda b_2^{-1} & *\end{pmatrix}\end{equation}
with $\eta_*$ as defined in \eqref{eta};
see Proposition 3.12 \cite{MaZ3}.
(Here, the
simple block upper-triangular form of $A_0$ has been used to deduce the above
simple form of $D_0$, $D_1$.)

\subsubsection{The parabolic block.} At this point, we have approximately
diagonalized our system
into a $1\times 1$ hyperbolic block with eigenvalue $\tilde \mu = -\tilde \lambda / A_*$ of $A_0$,
and a $2r\times 2r$ parabolic block
\begin{equation} \label{Zpeqs}Z_p' = N Z_p\end{equation} with
\begin{equation}N:= \begin{pmatrix} 0&I_r \\0&0\end{pmatrix}+
|\lambda|^{-1} \begin{pmatrix} 0&0 \\\tilde\lambda
b_2^{-1}&*\end{pmatrix} + \cO(|\lambda|^{-2}).\end{equation}

Balancing this matrix $N$ by transformations $\cB:
=diag\{I_r,|\lambda|^{-1/2}I_r\}$  we get
\begin{equation}\tilde
M:=\cB^{-1}N\cB= |\lambda|^{-1/2} \tilde M_1 +
\cO(|\lambda|^{-1}),\quad\tilde M_1:=\begin{pmatrix} 0&I_r
\\\tilde \lambda b_2^{-1}&0\end{pmatrix}\end{equation}

Observe that $\sigma(\tilde M_1) = \pm \sqrt{\sigma(\tilde \lambda
b_2^{-1})}$ has a uniform spectral gap of order one. Thus, there is
a well-conditioned transformation $S = S(\tilde M_1)$ depending
continuously on $\tilde M_1$ such that
\begin{equation} \hat M_1 := S^{-1}\tilde M_1 S =\diag\{\hat M^-,\hat M^+\}\end{equation}
with $\hat M_1^{\pm}$ uniformly positive/negative definite,
respectively. Applying this coordinate change, and noting that the
``dynamic error'' $S^{-1}\partial_{\tilde x}S$ is of order
$\partial_{\tilde x} \tilde M_1 = \cO(|\lambda|^{-1})$, we obtain
the formal expansion
\begin{equation} \hat M(\tilde x, |\lambda|^{-1})  =
|\lambda|^{-1/2} \diag \{\hat M_1^-, \hat M_1^+\} +
\cO(|\lambda|^{-1}).\end{equation}

Finally, on sector $\Omega_P$, blocks
$|\lambda|^{-1/2}\hat M_1^\pm$ are exponentially
separated to order $|\lambda|^{-1/2}$. Thus, by the {\it reduction lemma},
Lemma \ref{reduction},
there is a further transformation $\hat S:=I_{2r} + \cO(|\lambda|^{-1/2})$
converting $\hat M$ to the fully
diagonalized form
\begin{align*}
M (\tilde x, |\lambda|^{-1})&:=
|\lambda|^{-1/2}\hat S^{-1}\Big(\hat M_1 +\cO(|\lambda|^{-1/2})\Big)\hat S \\
&=
\cO(|\lambda|^{-1/2})\diag\{ M^-_1,M^+_1\}\end{align*}
where
$M_1^\pm = \hat M_1^\pm + \cO(|\lambda|^{-1/2})$ are still uniformly
positive/negative definite.


 In summary, changing coordinates \begin{equation} \cB
S \hat S \hat Z_p = Z_p,\end{equation} \eqref{Zpeqs} yields
\begin{equation}\label{Zhateqs} \hat Z_p ' =
\cO(|\lambda|^{-1/2})\begin{pmatrix} M^-_1&0
\\0&M^+_1\end{pmatrix} \hat Z_p + \cO(|\lambda|^{-3/2})
\end{equation}

Therefore the transformation
\begin{equation}\cT: = (T_0 + |\lambda|^{-1}T_1)\begin{pmatrix} 1&0
\\0&\cB S \hat S \end{pmatrix} \end{equation}
converts equations \eqref{firstorderY} to the following:
\begin{equation}\label{firstordercZ} \begin{aligned}\zeta' &= -(\tilde \lambda A_*^{-1} + |\lambda|^{-1}\eta_*A_*^{-1})\zeta + \cO(|\lambda|^{-2})
\\\rho_\pm' &= |\lambda|^{-1/2}M_1^\pm \rho_\pm + \cO(|\lambda|^{-3/2})\end{aligned}\end{equation}
by relation
\begin{equation}\cT \cZ = Y , \quad \cZ = (\zeta,\rho_-,\rho_+)^{tr}.\end{equation}

Then, we have the the representation
\begin{equation}\label{repcZ}
G_\lambda (x,y)=
\begin{cases}
   (I_n, 0)\cQ^{-1}\cT\CalF_{\cZ}^{y\rightarrow x}\Pi_{\cZ}^+ (y)\cT^{-1}\cQ\tilde
   S^{-1}(y)(I_n,0)^{tr}, & for \quad x>y,\\
-(I_n, 0)\cQ^{-1}\cT\CalF^{y\rightarrow x}_{\cZ}\Pi_{\cZ}^0
(y)\cT^{-1}\cQ \tilde
   S^{-1}(y)(I_n,0)^{tr},&for \quad x<y,
\end{cases}
\end{equation} thanks to the fact that \begin{equation} \CalF_{Y}^{y\rightarrow x} = \cT \CalF_{\cZ}^{y\rightarrow x} \cT^{-1}, \quad \Pi_{Y}^+= \cT\Pi_{\cZ}^+\cT^{-1}.\end{equation}

Computing, we have
\begin{align*}\cT=\begin{pmatrix} 1& |\lambda|^{-1/2}&|\lambda|^{-1/2}\\0&\cO(1)&\cO(1)\\0&|\lambda|^{-1/2}&|\lambda|^{-1/2}\end{pmatrix}
\quad\cT^{-1}= \begin{pmatrix} 1& 0&\tilde\lambda^{-1}A_*
A_{12}b_2^{-1}\\0&\cO(1)&|\lambda|^{1/2}\\0&\cO(1)&|\lambda|^{1/2}\end{pmatrix}\end{align*}
and
\begin{equation}(I_n,0)\cQ^{-1}
= \begin{pmatrix} A_*^{-1}&0&0
\\-b_2^{-1}b_1 A_*^{-1}& b_2^{-1} &
0\end{pmatrix}\end{equation}\begin{equation}(I_n,0)\cQ^{-1}\cT =
\begin{pmatrix} A_*^{-1}&\cO(|\lambda|^{-1/2})&\cO(|\lambda|^{-1/2})
\\-b_2^{-1}b_1A_*^{-1}&\cO(1)&\cO(1)\end{pmatrix}\end{equation}
 and \begin{equation}\cQ
\tilde S^{-1}(I_n,0)^{tr} = \begin{pmatrix} -1&0
\\0&0\\|\lambda|^{-1}&|\lambda|^{-1}I_r\end{pmatrix}\end{equation}
\begin{equation}\cT^{-1}\cQ \tilde S^{-1}(I_n,0)^{tr} =
\begin{pmatrix} -1+|\lambda|^{-1}&\cO(|\lambda|^{-1})
\\\cO(|\lambda|^{-1/2})&\cO(|\lambda|^{-1/2})\\\cO(|\lambda|^{-1/2})&\cO(|\lambda|^{-1/2})\end{pmatrix}\end{equation}
Therefore now we are ready to estimate $\CalF_{\cZ}^{y\rightarrow
x}\Pi_{\cZ}^+ $ and $\CalF_{\cZ}^{y\rightarrow x}\Pi_{\cZ}^+$.

\subsubsection{Estimates on projections and solution operators} We shall give estimates on the
projections: \begin{equation} \Pi_{\cZ}^+ = (\Phi^+,
0)(\Phi^+,\Phi^0)^{-1}, \quad  \Pi_{\cZ}^0 =
(0,\Phi^0)(\Phi^+,\Phi^0)^{-1}\end{equation} and the solution
operators:
\begin{equation}
\CalF_{\cZ}^{y\rightarrow x} =
(\Phi^+(x),\Phi^0(x))(\Phi^+(y),\Phi^0(y))^{-1}.
\end{equation}

First, let $\Phi^{p+} / \Psi^{p+}$ be the decaying/growing basis
solutions of
\begin{equation} \label{para-eqsZ}\rho_-' =
|\lambda|^{-1/2}M_1^-\rho_-\quad \mbox{and}\quad \rho_+' =
|\lambda|^{-1/2}M_1^+\rho_+ \end{equation} and $\phi^{h+} /
\psi^{h+}$ be the decaying/growing basis solutions of
\begin{equation}\label{hyp-eqZ} \zeta' = -(\tilde
\lambda A_*^{-1} + |\lambda|^{-1}\eta_*A_*^{-1})\zeta.
\end{equation}

\begin{lemma}{\em [Inflow case]} For the inflow case $A_*>0$, we obtain
\begin{equation} \Pi_{\cZ}^+ =
\begin{pmatrix} 1& 0 &
-|\lambda|^{-1/2}\phi^{h+}e(\lambda){\Psi^{p+}}^{-1}
\\0& I_r &-\Phi^{p+}E(\lambda){\Psi^{p+}}^{-1}\\0&0&0\end{pmatrix}\end{equation}
\begin{equation}
\Pi_{\cZ}^0 =
\begin{pmatrix} 0& 0 &
|\lambda|^{-1/2}\phi^{h+}e(\lambda){\Psi^{p+}}^{-1}
\\0& 0 &\Phi^{p+}E(\lambda){\Psi^{p+}}^{-1}\\0&0&I_r\end{pmatrix}\end{equation}
with bounded functions $e(\lambda),E(\lambda)$, and
\begin{equation}
\CalF_{\cZ}^{y\rightarrow x} =
\begin{pmatrix}\phi^{h+}(x)\phi^{h+}(y)^{-1} & 0 & 0\\0&\Phi^{p+}(x)\Phi^{p+}(y)^{-1}&0\\0&0&\Psi^{p+}(x)\Psi^{p+}(y)^{-1}\end{pmatrix}
\end{equation}
\end{lemma}

\begin{proof}
We have the decaying basis solution in $\cZ$-coordinates of
the first order equations \eqref{firstordercZ}
\begin{equation} \Phi^+ =
\begin{pmatrix} \phi^{h+} & 0 \\0& \Phi^{p+}
\\0&0\end{pmatrix} + \cO(|\lambda|^{-1}).\end{equation}

Since $\Phi^+$ and $\Psi^+$ (exactly $\Psi^{p+}$) form a basis
solution, we can write
\begin{equation}\label{Phi-0} \Phi^0(x) = e (\lambda)\begin{pmatrix} \phi^{h+}
\\0
\\0\end{pmatrix} + \begin{pmatrix} 0\\\Phi^{p+}(x)E(\lambda)
\\0\end{pmatrix} +
\begin{pmatrix} 0\\0
\\\Psi^{p+}(x)F(\lambda)\end{pmatrix}\end{equation}

Now since $\{\psi_j^{p+}\}_j$ forms a basis, we can take
$\{\psi_j^{p+}(0)\}$ to be the analytic basis for $Y$ at $x=0$. Also
as we recall that $\cZ = \cT^{-1} Y$, we compute
\begin{equation}\phi^0_{j|_{x=0}}  = \cT^{-1}\begin{pmatrix}
0\\0\\\psi_j^{p+}(0)\end{pmatrix} = \begin{pmatrix}
\cO(1)\\|\lambda|^{1/2}\psi_j^{p+}(0)\\|\lambda|^{1/2}\psi_j^{p+}(0)\end{pmatrix}\end{equation}
This and \eqref{Phi-0} yield
\begin{equation} \Phi^0(x) = \begin{pmatrix} e(\lambda)\phi^{h+}(x)
\\|\lambda|^{1/2}\Phi^{p+}(x)E(\lambda)
\\|\lambda|^{1/2}\Psi^{p+}(x)\end{pmatrix} +\cO(|\lambda|^{-1/2})\end{equation} where \begin{equation}E(\lambda) =
(E_1(\lambda),\dots, E_r(\lambda))^{tr},\quad E_j (\lambda)=
\psi_j^{p+}(0,\lambda) \Phi^{p+}(0,\lambda)^{-1}\end{equation} and
$e(\lambda),E_j(\lambda) \in \RR^r$ are bounded functions in
$\lambda$. Therefore computing, we get
\begin{equation}(\Phi^+,\Phi^0)^{-1}
=\begin{pmatrix} {\phi^{h+}}^{-1} & 0 &
-|\lambda|^{-1/2}e(\lambda){\Psi^{p+}}^{-1}
\\0& {\Phi^{p+}}^{-1} &-E(\lambda){\Psi^{p+}}^{-1}\\0&0&|\lambda|^{-1/2}{\Psi^{p+}}^{-1}\end{pmatrix}
\end{equation}
and hence straightforward computations give the lemma.\end{proof}


\begin{lemma}\label{lem-proj-out}{\em [Outflow case]}
For the outflow case $A_*<0$, we obtain
\begin{equation}
\Pi_{\cZ}^+ = \begin{pmatrix} 0& 0 & 0
\\0& I_r &-\Phi^{p+}E{\Psi^{p+}}^{-1}\\0&0&0\end{pmatrix},
\quad\Pi_{\cZ}^0 =\begin{pmatrix}1 & 0&0
\\0& 0
&\Phi^{p+}E{\Psi^{p+}}^{-1}
\\0& 0 &I_r\end{pmatrix},
\end{equation}
where $E(\lambda)$ is a bounded function in $\lambda$ determined
below. Moreover,
\begin{equation}
\CalF_{\cZ}^{y\rightarrow x} =
\begin{pmatrix}\psi^{h+}(x)\psi^{h+}(y)^{-1} & 0 & 0\\0&\Phi^{p+}(x)\Phi^{p+}(y)^{-1}&0\\0&0&\Psi^{p+}(x)\Psi^{p+}(y)^{-1}\end{pmatrix}
\end{equation}
\end{lemma}

\begin{proof} Similarly, we have $\Phi^+ = \Phi^{p+}$ and $\Phi^0 =
(\phi^{h0},\Phi^{p0})$ where we can write \begin{equation} \Phi^0(x)
=
\begin{pmatrix} 0\\\Phi^{p+}(x) E(\lambda)
\\0\end{pmatrix} + e
(\lambda)\begin{pmatrix} \psi^{h+} \\0
\\0\end{pmatrix}+
\begin{pmatrix} 0\\0
\\\Psi^{p+}(x)F(\lambda)\end{pmatrix}.\end{equation}
%
%

As before, using the form of the linearized boundary conditions
\eqref{outBC-U},
we can take
\begin{equation}\phi^{p0}_{j|_{x=0}}  = \cT^{-1}\begin{pmatrix}
0\\0\\\psi_j^{p+}(0)\end{pmatrix} =
\begin{pmatrix}
\cO(1)\\|\lambda|^{1/2}\psi_j^{p+}(0)\\|\lambda|^{1/2}\psi_j^{p+}(0)\end{pmatrix}\end{equation}
and thus\begin{equation} \Phi^{p0}(x) = \begin{pmatrix}
e(\lambda)\psi^{h+}(x)
\\|\lambda|^{1/2}\Phi^{p+}(x)E(\lambda)
\\|\lambda|^{1/2}\Psi^{p+}(x)\end{pmatrix}\end{equation} with bounded functions $e(\lambda)$ and $ E_j (\lambda)=
\psi_j^{p+}(0,\lambda) \Phi^{p+}(0,\lambda)^{-1}$.

Similarly, we take \begin{align*}\phi^{h0}_{|_{x=0}}
&=\cT^{-1}\begin{pmatrix}1\\0\\0\end{pmatrix}=
\begin{pmatrix}
1\\0\\0\end{pmatrix}\end{align*} and thus
\begin{equation} \phi^{h0}(x) = \begin{pmatrix} \psi^{h+}(x)
\\0
\\0\end{pmatrix} \end{equation}

Putting together and computing, we
obtain\begin{equation}(\Phi^+,\Phi^0) =
\begin{pmatrix} 0 &\psi^{h+}& e(\lambda)\psi^{h+}
\\\Phi^{p+}& 0&|\lambda|^{1/2}\Phi^{p+}E(\lambda)\\0&0&|\lambda|^{1/2}\Psi^{p+}\end{pmatrix}
\end{equation}and
\begin{equation}(\Phi^+,\Phi^0)^{-1} = \begin{pmatrix}
0 &{\Phi^{p+}}^{-1} & -E(\lambda){\Psi^{p+}}^{-1}\\{\psi^{h+}}^{-1}&
0
&-|\lambda|^{-1/2}e(\lambda){\Psi^{p+}}^{-1}\\
0&0&|\lambda|^{-1/2}{\Psi^{p+}}^{-1}\end{pmatrix}
\end{equation}
Direct computations yield the lemma. \end{proof}




\subsubsection{Estimates on $G_\lambda$: Inflow case $A_*>0$.} Now we are ready to combine all above estimates to give the bounds on resolvent kernel
$G_\lambda$. We shall work in detail for the case $x>y$. Similar
estimates can be easily obtained for $x<y$. First decompose the
projection as $\Pi_{\cZ}^+ =\Pi_{\cZ}^{h+}+\Pi_{\cZ}^{p+}$ where
\begin{equation} \begin{aligned}&\Pi_{\cZ}^{h+}=\begin{pmatrix} 1& 0 &
-|\lambda|^{-1/2}\phi^{h+}e(\lambda){\Psi^{p+}}^{-1}
\\0& 0 &0\\0&0&0\end{pmatrix}\\&\Pi_{\cZ}^{p+}=\begin{pmatrix} 0& 0 &
0\\0& I_r
&-\Phi^{p+}E(\lambda){\Psi^{p+}}^{-1}\\0&0&I_r\end{pmatrix}
\end{aligned}
\end{equation}

Hence
\begin{align*}H_\lambda(x,y) &=(I_n, 0)\cQ^{-1}\cT\CalF_{\cZ}^{y\rightarrow x}\Pi_{\cZ}^{h+} (y)\cT^{-1}\cQ\tilde
   S^{-1}(y)(I_n,0)^{tr}\\&=\phi^{h+}(x)\phi^{h+}(y)^{-1}\begin{pmatrix}(-1+\cO(|\lambda|^{-1}))A_*^{-1} & \cO(|\lambda|^{-1})A_*^{-1}
   \\(1+\cO(|\lambda|^{-1}))b_2^{-1}b_1A_*^{-1}&\cO(|\lambda|^{-1})b_2^{-1}b_1A_*^{-1}\end{pmatrix}
   \\&=\phi^{h+}(x)\phi^{h+}(y)^{-1}\begin{pmatrix}-A_*^{-1}(x) & 0
   \\b_2^{-1}b_1A_*^{-1}(x)&0\end{pmatrix}+\cO(|\lambda|^{-1})\phi^{h+}(x)\phi^{h+}(y)^{-1},
   \\&=\phi^{h+}(x)\phi^{h+}(y)^{-1}R_*L_*^{tr}+\cO(|\lambda|^{-1})\phi^{h+}(x)\phi^{h+}(y)^{-1},\end{align*}
recalling that $\phi^{h+}(x)\phi^{h+}(y)^{-1}$ is the solution
operator of hyperbolic equation in \eqref{hyp-eqZ} and thus
satisfies
\begin{equation} \phi^{h+}(x)\phi^{h+}(y)^{-1} =e^{\int_{\tilde y}^{\tilde x}(-1/A_* - |\lambda|^{-1}\eta_*/A_*)(z)dz}=
e^{\int_y^x(-\lambda/A_* - \eta_*/A_*)(z)dz}.\end{equation}

At the same time, computing $P_\lambda(x,y)$, we obtain
\begin{align*}P_\lambda(x,y) &=(I_n, 0)\cQ^{-1}\cT\CalF_{\cZ}^{y\rightarrow x}\Pi_{\cZ}^{p+} (y)\cT^{-1}\cQ\tilde
   S^{-1}(y)(I_n,0)^{tr}\\&=\cO(|\lambda|^{-1/2})\Phi^{p+}(x)\Phi^{p+}(y)^{-1}\end{align*}
recalling that $\Phi^{p+}(x)\Phi^{p+}(y)^{-1}$ is the (stable)
solution operator of parabolic equation \eqref{para-eqsZ}, with
$M_1^{-}$ uniformly negative definite, and thus we have an obvious
estimate
\begin{equation} |\Phi^{p+}(x)\Phi^{p+}(y)^{-1}| \le
Ce^{-\theta|\lambda|^{-1/2}(\tilde x-\tilde y)}\le
Ce^{-\theta|\lambda|^{1/2}(x- y)}.\end{equation}

We therefore obtain \begin{equation} P_\lambda(x,y) =
\cO(|\lambda|^{-1/2})e^{-\theta|\lambda|^{1/2}(x-y)}.\end{equation}

\subsubsection{Estimates on $G_\lambda$: Outflow
case $A_*<0$.} Again as above, we shall work in detail for the case
$x>y$. Similar estimates can be easily obtained for $x<y$.
Estimates in Lemma \ref{lem-proj-out} yield
\begin{equation}
\CalF_{\cZ}^{y\rightarrow x} \Pi_{\cZ}^+(y) =
\begin{pmatrix}0& 0 & 0\\ 0&\Phi^{p+}(x)\Phi^{p+}(y)^{-1}&-\Phi^{p+}(x)E(\lambda)
\Psi^{p+}(y)^{-1}\\0&0&0\end{pmatrix}
\end{equation}where $\Phi^{p+}(x)E(\lambda) \Psi^{p+}(y)^{-1} \le C \Phi^{p+}(x)
\Phi^{p+}(y)^{-1}$. Observe that $\Pi_{\cZ}^{h+} \equiv 0$.
Therefore, $H_\lambda(x,y)=0$ and
\begin{align*}P_\lambda(x,y) &=(I_n,
0)\cQ^{-1}\cT\CalF_{\cZ}^{y\rightarrow x}\Pi_{\cZ}^{p+}
(y)\cT^{-1}\cQ\tilde
   S^{-1}(y)(I_n,0)^{tr}\\&=\Phi^{p+}(x)\Phi^{p+}(y)^{-1}\begin{pmatrix}\cO(|\lambda|^{-1})&\cO(|\lambda|^{-1})\\\cO(|\lambda|^{-1/2})&\cO(|\lambda|^{-1/2})\end{pmatrix}
\\&\le C|\lambda|^{-1/2}e^{-\theta |\lambda|^{1/2}(x-y)}\end{align*}
We thus complete the proof of estimates $H_\lambda$ and of
$P_\lambda$ appearing in Proposition \ref{prop-high-est}.


\subsubsection{Derivative estimates.} Derivative estimates now follow in a straightforward fashion, by differentiation of
\eqref{repcZ}, noting from the approximately decoupled equations
that differentiation of the flow brings down a factor (to absorbable
error) of $\lambda$ in hyperbolic modes, $\lambda^{1/2}$ in
parabolic modes. This completes the proof of Proposition
\ref{prop-high-est}. \end{proof}

\subsection{Low frequency estimates} Our goal in this section is the estimation of the resolvent kernel
 in the critical regime $|\lambda|\to 0$, i.e., the large time
 behavior of the Green function G, or global behavior in space and
 time. We are basically following the same treatment as that
carried out for viscous shock waves of strictly parabolic
conservation laws in \cite{ZH,MaZ3}; we refer to those references
for details.
 In the low frequency case the behavior is essentially governed by
 the limiting far-field equation
\begin{equation}\label{lim-eqs}
U_t=L_+U :=-A_+U_x+B_+U_{xx}
\end{equation}

\begin{lem}[\cite{MaZ3}]\label{lem-low-solnbasis}
Assuming (H0)-(H3), for $|\lambda|$ sufficiently small, the
eigenvalue equation $(L_{+}-\lambda)W = 0$ associated with the
limiting, constant-coefficient operator $L_+$, considered as a
first-order system $W'=\BbbA_+ W$, $W=(u,v,v')$, has a basis of
$2n-1$ solutions
$\bar{W}^{+}_j=e^{\BbbA_+(\lambda)x}V_{j}(\lambda),$
consisting of $n-1$ ``fast'' modes (not necessarily eigenmodes)
\begin{equation} \label{fastbound}
|e^{\BbbA_+(\lambda)x}V_j|\le Ce^{-\theta|x|}, \quad \theta>0,
\end{equation}
 and $n$ analytic ``slow'' (eigen-)modes
\begin{equation} \label{mubound}
\begin{aligned}
e^{\BbbA_+(\lambda)x}V_j&=e^{\mu_j(\lambda)x}V_j,\\
\mu_{n-1+j}^{+}(\lambda)&:=
-\lambda/a^{+}_{j}+\lambda^2\beta^{+}_{j}/a^{+^{3}}_{j}+\CalO(\lambda^3),\\
V_{n-1+j}^{+}(\lambda) &:= r^{+}_{j} +\CalO(\lambda),\\
\end{aligned}
\end{equation}
where $a_j^{+}$, $l_j^+$, $r_j^{+},\beta_j^+$ are defined as in
Proposition \ref{prop-Greenbounds}.
The same is true for the adjoint eigenvalue equation
$$
(L_+-\lambda)^*Z=0,
$$
i.e, it has a basis of solutions $ \bar{ \tilde W}_{j}^+ =
e^{-\BbbA^*_+(\lambda)x}\tilde{V}_j(\lambda) $
with $n-1$ analytic ``fast'' modes
\begin{equation} \label{dualfastbound}
|e^{-\BbbA^*_+(\lambda)x}\tilde V_j|\le Ce^{-\theta|x|}, \quad
\theta>0,
\end{equation}
and $n$ analytic ``slow'' (eigen-)modes
\begin{equation}
\tilde{V}_{n-1+j}^+(\lambda)= l_j^+ + \CalO(\lambda).
\end{equation}
\end{lem}

\begin{proof}
Standard matrix perturbation theory; see \cite{MaZ3}, Appendix B.
\end{proof}

Also we recall from the representation of $G_\lambda$ in Corollary \ref{G-rep2}:
\begin{prop}\label{prop-greenlow}
Assuming (H0)-(H3), let $K$ be the order of the pole of $G_\lambda$
at $\lambda=0$ and r be sufficiently small that there are no other
poles in $B(0,r)$. Then for $\lambda \in \Omega_\theta$ such that
$|\lambda|\leq r$ and we have
\begin{equation} \label{greenlow}
G_\lambda(x,y)=\sum_{j,k}d_{jk}^+(\lambda)\phi_j^+(x)\tilde{\psi}_k^+(y)+
\sum_{k}\phi_k^+(x)\tilde{\phi}_k^+(y),
\end{equation}
for $x>y>0$, and
\begin{equation} \label{greenlow1}G_\lambda(x,y)=\sum_{j,k}d_{jk}^0(\lambda)\phi_j^+(x)\tilde{\psi}_k^+(y)+
\sum_{k}\psi_k^+(x)\tilde{\psi}_k^+(y),
\end{equation}
for $0<x<y$, where $d_{jk}^{0,+}(\lambda)=\CalO(\lambda^{-K})$ are
scalar meromorphic functions, moreover $K \leq$ order of vanishing
of the Evans function $D(\lambda)$ at $\lambda=0$.
\end{prop}

\begin{proof}
See \cite[Proposition 7.1]{ZH} for the first statement and theorem
6.3 for the second statement linking order $K$ of the pole to
multiplicity of the zero of the Evans Function.
\end{proof}

Our main result of this section is then:



\begin{prop}\label{prop-Gbounds-lambda-low}
Assume (H0)-(H3) and (D). Then, for $r>0$ sufficiently small, the
resolvent kernel $G_\lambda$ associated with the  linearized
evolution equation \eqref{lim-eqs} satisfies,
for $0\le  y \le x$:
\begin{equation}\label{ptres}
\begin{aligned}
|\partial_x^\gamma \partial_y^\alpha &G_\lambda(x,t;y)|\\&\le C
(|\lambda|^{\gamma} +  e^{-\theta |x|}) (|\lambda|^{\alpha} +
e^{-\theta |y|}) \Big(\sum_{a_k^{+}>0} \big|e^{ (-\lambda/a^{+}_k +
\lambda^2 \beta^{+}_k/{a^{+}_k}^3 )(x-y)}\big|\\
&\quad +\sum_{a_k^{+} < 0, \,  a_j^{+} > 0}
\big|e^{(-\lambda/a^{+}_j + \lambda^2 \beta^{+}_j/{a^{+}_j}^3 )x
+(\lambda/a^{+}_k - \lambda^2 \beta^{+}_k/{a^{+}_k}^3 )y}\big|
\Big),
\end{aligned}
\end{equation}
$0\le |\alpha|, |\gamma| \le 1$, $\theta>0$, with similar bounds for
$0\le x\le y$. Moreover, each term in the summation on the righthand
side of \eqref{ptres} bounds a separately analytic function.

\end{prop}
\begin{proof} By condition (D), $D(\lambda)$ does not vanish on $Re(\lambda)\geq 0$, hence, by continuity,
on $|\lambda|\le r$. Thus, according to Proposition
\ref{prop-greenlow}, all $|d_{jk}(\lambda)|$ are uniformly bounded
on $|\lambda|\le r$, and thus it is enough to find estimates for
fast and slow modes $\phi_j^+$, $\tilde{\phi}^+_j$, $\psi_j^+$ and
$\tilde{\psi}^+_j$.
By applying Lemma \ref{lem-low-solnbasis} and using \eqref{phi+} we
find:
\begin{equation}  \label{lphi+}
  \begin{pmatrix} \phi^+_j \\ \partial_x \phi^+_j
 \end{pmatrix}=e^{\BbbA_+(\lambda)x}P^+\begin{pmatrix}v_j \\ \mu_j
 v_j\end{pmatrix}=e^{\BbbA_+(\lambda)x}(I+\Theta)\begin{pmatrix}v_j \\ \mu_j
 v_j\end{pmatrix}
\end{equation}
and similarly for $\tilde{\phi}^+_j$, $\psi_j^+$ and
$\tilde{\psi}^+_j$. Now using \eqref{Theta} and the fact, by Lemma
\ref{lem-low-solnbasis}, that $e^{\mu_j(\lambda)x}$ is of order $e^{
-(\lambda/a^{+}_{j}+\lambda^2\beta^{+}_{j}/a^{+^{3}}_{j}+\CalO(\lambda^3))x}$
for slow modes and order $e^{-\theta |x |}$ for fast modes,
so by substituting this and corresponding dual estimates in
\eqref{lphi+} and grouping terms, we obtain the result.

\end{proof}

\section{Pointwise bounds on Green function $G(x,t;y)$}
In this section, we prove the pointwise bounds on the Green
function $G$ following the general approach of \cite{MaZ3}
in the whole-line, shock, case.
Our starting point is the representation
\begin{equation}\label{inv-Laplace} G(x,t;y) = \frac{1}{2\pi i}P.V.\int_{\eta - i\infty}^{\eta + i\infty}e^{\lambda t}
G_\lambda(x,y) \, d\lambda\end{equation} where $\eta$ is any
sufficiently large positive real number.
\\
\\
{\bf Case I. $|x-y|/t$ large.} We first treat the simple case that
$|x-y|/t\ge S$, $S$ sufficiently large. Fixing $x,y,t$, set $\lambda
= \eta + i\xi$, for $\eta>0$ sufficiently large. Applying
Proposition \ref{prop-high-est}, we obtain the decomposition
\begin{align*} G(x,t;y) &= \frac{1}{2\pi i}P.V \int_{\eta-i\infty}^{\eta +
i\infty}e^{\lambda t}\Big[H_\lambda + \Theta_\lambda^H+P_\lambda +
\Theta_\lambda^P\Big](x,y)d\lambda \\&=:I + II + III +
IV.\end{align*}

For definiteness considering the inflow case $A_*>0$ and taking $x > y$, we estimate each term in turn.\\
\\
{\it Term I.} Computing,
\begin{align*} I &=\frac{1}{2\pi i}P.V \int_{\eta-i\infty}^{\eta +
i\infty}e^{\lambda t}H_\lambda (x,y)d\lambda\\&= \frac{1}{2\pi}
A_*(x)^{-1} e^{\eta (t -
\int_y^x1/A_*(z)dz)}e^{-\int_y^x(\eta_*/A_*)(z)dz}P.V
\int_{-\infty}^{+\infty}e^{i\xi (t - \int_y^x1/A_*(z)dz)}d\xi
\\&= \frac{1}{2\pi}
A_*(x)^{-1} \delta(t -
\int_y^x1/A_*(z)dz)e^{-\int_y^x(\eta_*/A_*)(z)dz}
\\&= \frac{1}{2\pi}
A_*(x)^{-1} A_*(y)\delta_{x-\bar
a_*t}(y)e^{-\int_y^x(\eta_*/A_*)(z)dz}\end{align*} where $\bar a_*$
is defined as in Proposition \ref{prop-Greenbounds}. Noting that
$\bar a_* \ge \inf _xA_*(x)>0$ and $\eta_*^+>0$, we get
$e^{-\int_y^x(\eta_*/A_*)(z)dz} = \cO(e^{-\theta(x-y)})$ and thus
\begin{equation}\label{Hest} I= \cO(e^{-\theta t})\delta_{x-\bar
a_*t}(y),\end{equation} vanishing for $|x-y|/t$ large.
\\
\\{\it Term II.} Similar calculations show that the ``hyperbolic error term'' $II$ also
vanishes. For example, the term $e^{\lambda
t}\lambda^{-1}B(x,y;\lambda)$ contributes
\begin{align*} \frac{1}{2\pi}e^{\eta (t -
\int_y^x1/A_*(z)dz)}e^{-\int_y^x(\eta_*/A_*)(z)dz}P.V
\int_{-\infty}^{+\infty} (\eta + i\xi)^{-1}e^{i\xi (t -
\int_y^x1/A_*(z)dz)}d\xi.\end{align*} The integral though not
absolutely convergent, is integrable and uniformly bounded as a
principal value integral, for all real $\eta$ bounded away from
zero, by explicit computation. On the other hand
$$e^{\eta (t - \int_y^x1/A_*(z)dz)} \le  e^{\eta (t -
|x-y|/\min_zA_*(z))} \le e^{\eta t(1- S/\min_zA_*(z))}\to 0,$$ as
$\eta \to +\infty$, for $S$ sufficiently large. Thus, we find that
the above integral term goes to zero. Likewise, the result applies
for the term of $e^{\lambda t}C(x,y;\lambda)$, since
$(x-y)e^{-\int_y^x(\eta_*/A_*)(z)dz} \le C(x-y)e^{-\theta (x-y)}$ is
also bounded. Thus, each term of $II$ vanishes as $\eta \to
+\infty$.\\
\\
{\it Term $III$.} The parabolic term $III$ may be treated exactly as
in the strictly parabolic case \cite{ZH}. Precisely, we may first
deform the contour in the principle value integral to
\begin{equation}
\int_{\Gamma_1\cup \Gamma_2} e^{\lambda t} P_\lambda(x,y) \,
d\lambda,
\end{equation}
where $\Gamma_1:= \partial B(0,R)\cap \bar \Omega_P$ and $\Gamma_2:=
\partial \Omega_P \setminus B(0,R)$, recalling the parabolic sector $\Omega_P$ defined in
\eqref{sectorP}. Choose
\begin{equation}\label{largedef}
\bar{\alpha} :={\frac{|x-y|}{2 \theta t}}, \quad R:=
\theta\bar{\alpha}^2,
\end{equation}
where $\theta$ is as in \eqref{Plamda-est}. Note that the
intersection of $\Gamma$ with the real axis is
$\lambda_{min}=R=\theta \bar{\alpha}^2$. By the large $|\lambda|$
estimates of Proposition \ref{prop-high-est}, we have for all
$\lambda \in \Gamma_1\cup \Gamma_2$ that
$$
|P_\lambda (x,y)|\le C |\lambda|^{-1/2}
e^{-\theta|\lambda|^{1/2}|x-y|}.
$$
Further, we have
\begin{equation}
\begin{aligned}
Re \lambda &\le  R(1- \eta \omega^2), \quad \lambda\in \Gamma_1,
\\Re \lambda &\le Re \lambda_0 - \eta (|Im \lambda| - |Im
\lambda_0|), \quad \lambda \in \Gamma_2
\end{aligned}
\end{equation}
for $R$ sufficiently large, where $\omega$ is the argument of
$\lambda$ and $\lambda_0$ and $\lambda_0^*$ are the two points of
intersection of $\Gamma_1$ and $\Gamma_2$, for some $\eta>0$
independent of $\bar{\alpha} $. Combining these estimates, we obtain
\begin{equation}
\begin{aligned}
\Big|\int_{\Gamma_{1}} e^{\lambda t} P_\lambda  d\lambda\Big| &\le C
\int_{\Gamma_{1}} |\lambda|^{-1/2}\, e^{Re \lambda t
-\theta|\lambda|^{1/2}|x-y| } \, d\lambda \\ &\le C e^{-\theta
\bar{\alpha} ^{2}t} \int_{-arg(\lambda_0)}^{+arg(\lambda_0)}
R^{-1/2}e^{-\theta R \eta \omega^2 t} \, R d\omega \\&\le Ct^{-1/2}
e^{-\theta \bar{\alpha} ^{2}t}.
\end{aligned}
\end{equation}
Likewise,
\begin{equation}
\begin{aligned}
|\int_{\Gamma_{2}} e^{\lambda t} P_\lambda  d\lambda| &\le
\int_{\Gamma_{2}}C |\lambda|^{-1/2}\, C e^{Re \lambda t -\theta
|\lambda|^{1/2}|x-y|} d\lambda \\ &\le C
e^{Re(\lambda_{0})t-\theta|\lambda_0|^{1/2} |x-y|} \int_{\Gamma_{2}}
|\lambda|^{-1/2}e^{(Re \lambda-Re\lambda_{0})t}\
|d\lambda|\\
&\le C e^{-\theta \bar{\alpha} ^{2}t} \int_{\Gamma_2}  |Im \,
\lambda|^{-1/2} e^{-\eta|Im \, \lambda-Im \, \lambda_{0}|t}\ |d\, Im
\lambda|\\ & \le  Ct^{-1/2} e^{-\theta \bar{\alpha}^{2}t}.
\end{aligned}
\end{equation}
Combining these last two estimates, we have
\begin{equation} \label{g1}
III\le Ct^{-1/2} e^{{-\theta \bar{\alpha} ^{2}t}/{2}}
e^{{-(x-y)^{2}}/{8\theta t}} \le Ct^{-1/2}e^{-\eta t}
e^{-(x-y)^{2}/{8\theta t}},
\end{equation}
for $\eta>0$ independent of $\bar{\alpha}$. Observing that ${|x-at|/
2t} \le {{|x-y|}/{t}} \le{{2|x-at|}/{t}}$ for any bounded $a$, for
${|x-y|}/{t}$ sufficiently large, we find that $III$ may be absorbed
in any summand $ t^{-1/2}e^{{-(x-y-a_k^+ t)^2}/{Mt}}. $\\
\\{\it Term $IV$.} Similarly, as in the treatment of the term $III$,
the principle value integral for the ``parabolic error term $IV$ may
be shifted to $\eta = R = \theta \bar\alpha^2$, $\bar \alpha$ as
above. This yields an estimate $$|IV|\le Ce^{-\theta \bar \alpha^2
t}\int_{-\infty}^{+\infty} |\eta_0+i\xi|^{-2} d\xi \le Ce^{-\theta
\bar \alpha^2 t},$$ absorbed in $\cO(e^{-\eta t}e^{-|x-y|^2/Mt})$
for all $t$.
\\
\\
{\bf Case II. $|x-y|/t$ bounded.} We now turn to the critical case
where $|x-y|/t \le S$, for some fixed $S$.\\
\\
{\it Decomposition of the contour:} We begin by converting the
contour integral \eqref{inv-Laplace} into a more convenient form
decomposing high, intermediate, and low frequency contributions.

We first observe that $L$ has no spectrum on the portion of $\Omega$
lying outside the rectangle \begin{equation} \R: =
\{\lambda~:~-\eta_1 \le \Re \lambda \le \eta, -R \le \Im \lambda\le
R\}\end{equation} for $\eta>0$, $R>0$ sufficiently large, hence
$G_\lambda$ is analytic on this region. Since, also, $H_\lambda$ is
analytic on the whole complex plane, contours involving either one
of these contributions may be arbitrarily deformed within $\Omega
\setminus \R$ without affecting the result, by Cauchy's theorem.
Likewise, $P_\lambda$ is analytic on $\Omega_P\setminus \R$, and so
contours involving this contribution may be arbitrarily deformed
within this region. Thus, we obtain

\begin{observation}[\cite{MaZ3}] Assume (H0)-(H3) and (D) hold. The principle
value integral \eqref{inv-Laplace} may be replaced by
 \begin{equation}\label{G-decomp}G(x,t;y) = I_a + I_b + I_c + II_a + II_b +
 III\end{equation} where
\begin{align*}I_a&:= P.V. \int_{\eta-i\infty}^{\eta +i\infty} e^{\lambda
t}H_\lambda(x,y) d\lambda\\I_b&:= P.V \Big(\int_{-\eta_1 -
i\infty}^{-\eta_1 - iR} + \int_{-\eta_1 + iR}^{-\eta_1 +
i\infty}\Big) e^{\lambda t}(G_\lambda - H_\lambda -
P_\lambda)(x,y)d\lambda\\I_c:&=\int_{\Gamma_2} e^{\lambda
t}P_\lambda(x,y)d\lambda\\II_a&:= \Big(\int_{-\eta_1 - iR}^{-\eta_1
- ir/2} + \int_{-\eta_1 + ir/2}^{-\eta_1 + iR}\Big) e^{\lambda t}
G_\lambda(x,y) d\lambda\\II_b&:=-\int_{-\eta_1 - iR}^{-\eta_1 + iR}
e^{\lambda t}H_\lambda(x,y)d\lambda \\III&:= \int_{
\Gamma_1}e^{\lambda t}G_\lambda(x,y)d\lambda\end{align*} with
\begin{align*}\Gamma_1:=&[-\eta_1 -
ir/2,\eta - ir/2] \cup [\eta - ir/2,\eta +
ir/2]\cup[\eta+ir/2,-\eta_1 + ir/2]\\\Gamma_2:=&
\partial\Omega_P \setminus \Omega,
\end{align*}
for any $\eta,r>0$, $R$ sufficiently large, and $\eta_1$
sufficiently small with respect to $r$.
\end{observation}

Using the above decomposition \eqref{G-decomp}, we shall estimate in
turn the high-frequency contributions $I_a, I_b,$ and $I_c$, the
intermediate-frequency contributions $II_a$ and $II_b$, and the
low-frequency contributions $III$.
\\
\\
{\it High-frequency contribution.} We first carry out the
straightforward estimation of the high-frequency terms $I_a, I_b,$
and $I_c$. The principal term $I_a$ has already been computed in
\eqref{Hest} to be $H(x, t; y)$. Likewise, calculations similar to
those of {\it Term II} show that the term $I_b$ is
time-exponentially small. For example, the term $e^{\lambda
t}\lambda^{-1}B(x,y;\lambda)$ contributes
\begin{align}\label{Ib1}P.V.&\Big(\int_{-\infty}^{-R}+\int_R^{+\infty}\Big) (-\eta_1 +
i\xi)^{-1} e^{i\xi (t-\int_y^x 1/A_*(z)dz)}d\xi\notag\\&\times e^{-\eta_1
(t - \int_y^x1/A_*(z)dz)}e^{-\int_y^x(\eta_*/A_*)(z)dz}
\end{align}
where \begin{equation}\label{Ib2}P.V.\Big(\int_{-\infty}^{-R}+\int_R^{+\infty}\Big) (-\eta_1
+ i\xi)^{-1} e^{i\xi (t-\int_y^x 1/A_*(z)dz)}d\xi <\infty\end{equation} and
\begin{equation}\label{Ib3}e^{\eta_1 \int_y^x1/A_*(z)dz}e^{-\int_y^x(\eta_*/A_*)(z)dz} \le
C e^{\frac{\eta_1|x-y|}{\min_zA_*(z)}}e^{-\theta|x-y|} \le
Ce^{-\theta|x-y|/2},\end{equation} for $\eta_1$ sufficiently small.
This contributes in the term $\cO(e^{-\eta_1(t+|x-y|)})$ of $R$.
Likewise, the contributions of terms $e^{\lambda
t}\lambda^{-1}(x-y)C(x,y;\lambda)$ and $e^{\lambda
t}\lambda^{-2}D(x,y;\lambda)$ split into the product of a
convergent, uniformly bounded integral in $\xi$, a bounded factor
analogous to \eqref{Ib3}, and the factor $e^{-\eta_1 t}$,
giving the result.

The term $I_c$ may be estimated exactly as was term $III$ in the large $|x-y|/t$
case, to obtain contribution $\cO(t^{-1/2}e^{-\eta_1 t})$ absorbable again in the residual term
$\cO(e^{-\eta t}e^{-|x-y|^2/Mt})$ for $t\ge \epsilon$, any $\epsilon>0$, and by any summand $\cO(t^{-1/2}(1+t)^{-1/2}e^{-(x-y-a_k^+)^2/Mt}) e^-\eta (x+y)$ for $t$ small.
\\
\\
{\it Intermediate-frequency contribution.} Error term $II_b$ is time-exponentially small for
$\eta_1$ sufficiently small, by the same calculation as in \eqref{Ib1}-\eqref{Ib3}, hence negligible.
Likewise, term $II_a$ by the basic estimate \eqref{Inter-est} is seen to be time-exponentially
small of order $\cO(-\eta_1 t)$ for any $\eta_1 > 0$ sufficiently small that the associated contour
lies in the resolvent set of $L$.
\\
\\
{\it Low-frequency contribution}. It remains to estimate the
low-frequency term $III$, which is of essentially the same form as the low-frequency contribution analyzed
in \cite{ZH,YZ} in the strictly parabolic case, in that the contour is the same and
the resolvent kernel $G_\lambda$ satisfies same bounds (with no $E_\lambda$ term) in this regime. Thus, we may
conclude from these previous analyses that $III$ gives contribution 
as claimed, exactly as in the strictly parabolic case. For
completeness, we indicate the main features of the argument here.
\\
\\
{\bf Bounded time}.  For $t$ bounded, we can use the
medium-$\lambda$ bounds $|G_\lambda|$, $|G_{\lambda_x}|$,
$|G_{\lambda_y}| \le C$ to obtain $|\int_{\Gamma_1} e^{\lambda t}
G_\lambda d \lambda| \le C_2 |\Gamma_1|$. This contribution is order
$Ce^{-\eta t}$ for bounded time, hence
can be absorbed.\\
\\
{\bf Large time}. For $t$ large, we must instead estimate
$\int_{\Gamma_1} e^{\lambda t} G_\lambda d \lambda$ using the
small-$|\lambda|$ expansions. First, observe that, all coefficient
functions $d_{jk}(\lambda)$ are uniformly bounded (since $|\lambda|$
is bounded in this case).

~\\
{\bf Case II(i). $(0<y<x)$}.  By our low-frequency estimates in
Proposition \ref{prop-greenlow}, we have
\begin{equation} \label{gamma}
\begin{aligned}
\int_{\Gamma_1} e^{\lambda t}G_\lambda(x,y) \, d\lambda &=
\int_{\Gamma_1} \sum_{j,k} e^{\lambda t} d_{jk} \phi^+_j (x)  \tilde
\psi^+_k (y) d\lambda\\
&\quad +\int_{\Gamma_1} \sum_{k} e^{\lambda t} \phi^+_k (x) \tilde
\phi^+_k (y) d\lambda,\\
\end{aligned}
\end{equation}
where each $d_{jk}$ is analytic, hence bounded. We estimate
separately each of the terms
$$
\int_{\Gamma_1} e^{\lambda t} d_{jk} \phi^+_j (x) \tilde \psi^+_k
(y) d\lambda
$$
on the righthand side of \eqref{gamma}. Estimates for terms
$$
\int_{\Gamma_1} e^{\lambda t} \phi^+_k (x)  \tilde \phi^+_k (y)
d\lambda
$$
go similarly.
\medskip

 {\bf Case II(ia).} First, consider the critical case $a^+_j
>0$, $a^+_k
<0$ . For this case,
$$|d_{jk} \phi^+_j{(x)} \tilde \psi^+_k (y)| \le C e^{Re(\rho^+_j x - \nu^+_k y)},$$
where
$$
\begin{cases} \nu^+_k (\lambda) = -\lambda/a^+_k + \lambda^2 \beta^+_k
/ (a^+_k)^3 + \CalO (\lambda^3) \cr \rho^+_j (\lambda) = -\lambda /
a^+_j + \lambda^2 \beta^+_j / (a^+_j)^3 + \CalO(\lambda^3).
\end{cases}
$$

Set
$$
\bar{\alpha}= \frac{a^+_k x/a^+_j -y - a^+_k t}{2t}, \quad
p:=\frac{\beta^+_j a^+_k x / (a^+_j)^3 - \beta^+_k y / (a^+_k)^2
}{t} <0.
$$
Define $\Gamma_{1a}'$ to  be the portion contained in
$\Omega_\theta$ of the hyperbola
\begin{equation}
\begin{aligned} \label{rho-nu}
&Re(\rho_j^+x - \nu_k^+y) + \CalO(\lambda ^3)(|x|+|y|)\cr &\quad=
(1/a_k^+)Re[\lambda(-a^+_k x/ a^+_j + y) + \lambda^2 (x \beta^+_j
a^+_k / (a^+_j)^3 - y \beta^+_k / (a^+_k)^2 )] \cr &\quad \equiv
\const \cr {} &\quad= (1/a_k^+) [(\lambda_{min}(-a^+_k x/ a^+_j + y)
+ \lambda^2_{min} (x \beta^+_j a^+_k / (a^+_j)^3 - y \beta^+_k /
(a^+_k)^2)], \cr
\end{aligned}
\end{equation}
where
\begin{equation}
\lambda_{min} :=
\begin{cases}
\frac{\bar{\alpha}}{p} &  if  \quad |\frac{\bar{\alpha}}{p}|\le
\epsilon \cr \pm \epsilon & if \quad \frac{\bar{\alpha}}{p} \gtrless
\epsilon
\end{cases}
\end{equation}

Denoting by $\lambda_1$, $\lambda^*_1$, the intersections of this
hyperbola with $\partial \Omega_\theta$, define $\Gamma_{1_b}'$ to
be the union of $\lambda_1 \lambda_0$ and $\lambda^*_0 \lambda^*_1$,
and define $\Gamma_1' = \Gamma_{1_a}' \cup \Gamma_{1_b}'$. Note that
$\lambda = \bar{\alpha}/p$ minimizes the left hand side of
\eqref{rho-nu} for $\lambda$ real. Note also that that $p$ is
bounded for $\bar{\alpha}$ sufficiently small, since
$\bar{\alpha}\le \epsilon$ implies that
$$
(|a_k^+x/ a_j^+| + |y|)/t \le 2|a_k^+| + 2\epsilon
$$
i.e. $(|x|+|y|)/t$ is controlled by $\bar{\alpha}$.

With these definitions, we readily obtain that
\begin{equation}
\begin{aligned}
Re(\lambda t + \rho^+_j x - \nu^+_k y) &\le -(t/a^-_k)
(\bar{\alpha}^2/4p) - \eta Im (\lambda)^2 t \cr {} &\le -
\bar{\alpha}^2 t/M - \eta Im (\lambda)^2 t,
\end{aligned}
\end{equation}
for $\lambda \in \Gamma_{1_a}'$ (note: here, we have used the
crucial fact that $\bar{\alpha}$ controls $(|x|+|y|)/t$, in bounding
the error term $\CalO(\lambda^3)(|x|+|y|)/t$ arising from expansion
Likewise, we obtain for any $q$ that
\begin{equation} \label{gamma'}
\int_{\Gamma_{1_a}'} |\lambda|^q e^{Re(\lambda t + \rho^+_j x -
\nu^-_k y)} d\lambda \le C t^{-\frac{1}{2} -\frac{q}{2}}
e^{-\bar{\alpha}^2 t / M},
\end{equation}
for suitably large $C,\, M>0$ (depending on $q$). Observing that
$$
\bar{\alpha}= (a_k^+/a_j^+)(x- a_j^+ (t - |y/a_k^+|))/2t,
$$
we find that the contribution of \eqref{gamma'} can be absorbed in
the described bounds for $t\ge |y/a_k^-|$. At the same time, we find
that $\bar{\alpha}\ge x > 0$ for $t\le |y/a_k^+|$, whence
$$
\bar{\alpha}\ge (x- y - a_j^+t)/Mt +|x|/M,
$$
for some $\epsilon>0$ sufficiently small and  $M>0$ sufficiently
large.

This gives
$$
e^{-\bar{\alpha}^2/|p|}\le e^{- (x-y- a_k^+t)^2/Mt} e^{-\eta |x|}
$$
provided $|x|/t >a_j^+$, a contribution which can again be absorbed.
On the other hand, if $t\le |x/a_j^+|$, we can use the dual estimate
\begin{equation}
\begin{aligned}
\bar{\alpha}&= (-y- a_k^+(t - |x/a_j^+|))/2t \cr &\ge  (x-y-
a_k^+t)/Mt  + |y|/M,
\end{aligned}
\end{equation}
together with $|y|\ge |a_k^- t|$, to obtain
$$
e^{-\bar{\alpha}^2/|p|}\le e^{- (x-y- a_j^+t)^2/Mt} e^{-\eta |y|},
$$
a contribution that can likewise be absorbed.


{\bf Case II(ib).} In case $a^+_j<0$ or $a^{+}_k >0$, terms
$|\varphi_j^+| \le C
e^{-\eta |x|}$ and $|\tilde{\psi}_j^+| \le C e^{-\eta |y|}$ are
strictly
smaller than those already treated in  Case II(ia), so may be
absorbed in previous terms.
\bigskip
\\
{\bf Case II(ii) $(0<x<y)$.} The case $0<x<y$ can be treated very
similarly to the previous one; see \cite{ZH} for details. This
completes the proof of Case II, and the theorem.

\section{Energy estimates}
\subsection{Energy estimate I}
We shall require the following energy estimate adapted from
\cite{MaZ4,Z2}. Define the nonlinear perturbation variables $U = (u,
v)$ by
\begin{equation}\label{per-var}
U(x,t):=\tilde U(x,t)-\bU(x).
\end{equation}

\begin{proposition}\label{prop-energy-est} Under the hypotheses of Theorem \ref{theo-nonlin}, let $U_0 \in H^4$ and $U=(v,u)^T$ be a solution of \eqref{hyper-parabolic} and \eqref{per-var}. Suppose that, for $0\le t\le T$, the $W^{2,\infty}_x$ norm of
the solution $U$ remains bounded by a sufficiently small constant
$\zeta>0$. Then
\begin{align}\label{energy-ineq}\|U(t)\|_{H^4}^2 \le Ce^{-\theta t}\|U_0\|_{H^4}^2 +
C \int_0^t
e^{-\theta(t-\tau)}\Big(\|U(\tau)\|_{L^2}^2+\CalB_h(\tau)\Big)d\tau\end{align}
for all $0\le t\le T$, where the boundary operator $\CalB_h$ is
defined in Theorem \ref{theo-nonlin}.
\end{proposition}

\begin{proof}Observe that a straightforward calculation shows that
$|U|_{H^r}\sim|W|_{H^r},$
\begin{equation}\label{per-varW}W = \tilde W - \bar W := W(\tilde U ) -W(\bU),\end{equation}
for $0\le r\le 4$, provided $|U|_{W^{2,\infty}}$ remains bounded, hence it is sufficient to prove a
corresponding bound in the special variable $W$. We first carry out a complete proof
in the more straightforward case with conditions (A1)-(A3) replaced by the following global
versions, indicating afterward by a few brief remarks the changes needed to carry
out the proof in the general case.
\medskip

(A1') $\tilde A(\tilde W),\tilde A^0,\tilde A^{11}$ are symmetric, $\tilde
A^0\ge \theta_0>0$,
\medskip

(A2') no eigenvector of $\tilde A(\tilde A^0)^{-1}(\tilde W)$ lies in the
kernel of $\tilde B(\tilde W)$,
\medskip

(A3') $\tilde W = \begin{pmatrix}\tilde w^I\\\tilde w^{II}\end{pmatrix}, \tilde B=\begin{pmatrix}0 & 0 \\
0 & \tilde b\end{pmatrix}, \tilde b\ge \theta>0,$ and $\tilde
G\equiv 0$.

Substituting \eqref{per-varW} into \eqref{symmetric-form}, we obtain the quasilinear perturbation equation
\begin{align}\label{perturb-eqs} A^0 W_t + AW_x =
(BW_{x})_x + M_1\bW_x + (M_2\bW_x)_x + g(\tilde
W_x)-g(\bW_x)\end{align} where $A^0:=A^0(W+\bW)$ is positive
definite symmetric, $A:=A(W+\bW) $ is symmetric,
\begin{align*}M_1 &= A(W+\bW) - A(\bW) = \Big(\int_0^1 dA(\bW +
\theta
W)d\theta\Big)W,\\
M_2 &= B(W+\bW) - B(\bW) = \begin{pmatrix}0 & 0 \\
0 & (\int_0^1 dA(\bW + \theta W)d\theta)W\end{pmatrix}.\end{align*}

As shown in \cite{MaZ4}, we have bounds
\begin{align}\label{bound-A1}|A^0|\le C,\quad |A^0_t| &\le C|W_t|\le C(|W_x|+|w^{II}_{xx}|)\le
C\zeta,\\|\partial_xA^0|+|\partial_x^2A^0|&\le
C(\sum_{k=1}^2|\partial_x^kW|+|\bW_x|)\le
C(\zeta+|\bW_x|).\label{bound-A2}\end{align} We have the same bounds
for $A,B,K$, and also due to the form of $M_1,M_2$,
\begin{align}\label{bound-M}|M_1|,|M_2|\le C(\zeta+|\bW_x|)|W|. \end{align}

Note that thanks to Lemma \ref{lem-profile-decay} we have the bound
on the profile: $|\bW_x|\le Ce^{-\theta |x|}$, as $x \to +\infty$.

The following results assert that hyperbolic effects can compensate
for degenerate viscosity $B$, as revealed by the existence of a
compensating matrix $K$.

\begin{lemma}[\cite{KSh}]\label{K} Assuming (A1'), condition (A2') is equivalent to the following:

(K1) There exists a smooth skew-symmetric matrix $K(W)$ such that
\begin{equation}\label{K1}\R(K(A^0)^{-1}A + B)(W)\ge
\theta_2>0.\end{equation}
\end{lemma}

Define $\alpha$ by the ODE \begin{equation}\label{alphaeq} \alpha_x
= -\mbox{sign}(A^{11})c_*|\bW_x|\alpha, \quad
\alpha(0)=1\end{equation} where $c_*>0$ is a large constant to be
chosen later. Note that we have $\alpha_x/\alpha = \beta_x/\beta +
\gamma_x/\gamma$ and thus \begin{align} \label{alpha-est}
(\alpha_x/\alpha)A^{11} \le
-c_*\theta_1|\bW_x|=:-\omega(x)\end{align}
 and  \begin{align} \label{alpha-est1}
|\alpha_x/\alpha |\le c_*|\bW_x|=\theta_1^{-1}\omega(x).\end{align}

In what follows, we shall use $\wprod{,}$ as the $\alpha$-weighted
$L^2$ inner product defined as $$\wprod{f,g} = \iprod{\alpha
f,g}_{L^2}$$ and $\|f\|_s =
\sum_{i=0}^s\wprod{\frac{d^{(i)}}{dx^i}f,\frac{d^{(i)}}{dx^i}f}^{1/2}$
as the norm in weighted $H^s$ space. Note that for any symmetric
operator $S$,
\begin{align*}\wprod{S f_x,f} &= -\frac 12\wprod{(S_x+(\alpha_x/\alpha)S)f,f}-\frac 12 S_0f_0.f_0.\end{align*}

Note that in what follows, we shall pay attention to keeping track of $c_*$. For constants independent of $c_*$, we simply write them as $C$.

\subsubsection{Zeroth order ``Friedrichs-type'' estimate} First
employing integration by parts yields, and using estimates
\eqref{bound-A1}, \eqref{bound-A2}, and then \eqref{alpha-est}, we
obtain
\begin{align*}-\langle A&W_x,W\rangle \\&= \frac 12
\wprod{(A_x+(\alpha_x/\alpha)A)W,W}+ \frac 12 A_0W(0)\cdot W(0)\\&\le
\frac 12 \wprod{(\alpha_x/\alpha)A^{11}w^{I},w^{I}} +
C\wprod{(\zeta+|\bW_x|)|W| + \omega(x)|w^{II}|,|W|} + I^0_b
\\&\le
-\frac 12\wprod{\omega (x)w^{I},w^{I}}+ C(\zeta\|w^{I}\|_0^2
+\wprod{|\bW_x|w^{I},w^{I}})+ C(c_*)\|w^{II}\|^2_0+ I^0_b,
\end{align*} where $I^0_b$ denotes the boundary term $\frac 12 A_0W(0)\cdot W(0)$, and the term $\wprod{|\bW_x|w^{I},w^{I}}$ may be easily absorbed into the first right-hand term, as for $c_*$ sufficiently large,
\begin{equation}\label{est-v}\wprod{|\bW_x|w^{I},w^{I}} \le (c_*\theta_1)^{-1}\wprod{\omega(x)w^{I},w^{I}}\le \frac 1{4C}\wprod{\omega (x)w^{I},w^{I}}.\end{equation}
Whereas, integration by parts also yields
\begin{align*}\wprod{(BW_{x})_x,W}&=-\wprod{BW_{x},W_x}-\wprod{(\alpha_x/\alpha)BW_{x},W} - B_0W_x(0)\cdot W(0)\\&\le
-\theta
\|w^{II}_x\|_0^2+C\wprod{\omega(x)w^{II}_x,u}-b_0w^{II}_x(0)w^{II}(0)\\&\le
-\theta \|w^{II}_x\|_0^2+C(c_*)\|w^{II}\|_0^2 -
b_0w^{II}_x(0)w^{II}(0).
\end{align*}
where we used the fact that $B_0W_x(0)\cdot W(0) =
b_0w^{II}_x(0)w^{II}(0)$, and similarly, together with the
block-diagonal form of $B$ and thus of $M_2 = B(W+\bW) - B(\bW)$,
\begin{align*}\langle (M_2&\bW_x)_x,W\rangle\\&=-\wprod{M_2\bW_x,W_x}-\wprod{(\alpha_x/\alpha)M_2\bW_x,W
}- M_2(0)\bW_x(0)\cdot W(0)\\&\le C\wprod{|\bW_x||W|,|w^{II}_x|} +
C\wprod{\omega(x)|W|,w^{II}}- M_2(0)\bW_x(0)\cdot W(0)
\\&\le \xi\|w^{II}_x\|_0^2 + C\Big(\epsilon\wprod{\omega(x)w^{I},w^{I}}+ C(c_*)\|w^{II}\|_0^2\Big)- M_2(0)\bW_x(0)\cdot W(0)
\end{align*}for any small $\xi$. Note that $C$ is independent of $c_*$. Therefore, for $\xi=\theta/2$ and $c_*$ sufficiently large, we can compute
\begin{align}
\frac 12 &\dt\wprod{A^0W,W}\\
 &= \wprod{A^0W_t,W} +
\frac 12 \wprod{A^0_tW,W}\notag\\
&=\wprod{-AW_x +
(BW_{x})_x+M_1\bW_x + (M_2\bW_x)_x,W}+\frac 12
\wprod{A^0_tW,W}\notag\\
&\le -\frac 14[\wprod{\omega(x)w^{I},w^{I}}
+\theta \|w^{II}_x\|_0^2]+ C\zeta\|w^{I}\|_0^2 +
C(c_*)\|w^{II}\|^2_0 + I_b^0\label{Fzeroth-est0}
\end{align} where the boundary term
\begin{equation} I_b^0:=\frac 12 A_0W(0)\cdot W(0)- b_0w^{II}_x(0)w^{II}(0) -
M_2(0)\bW_x(0)\cdot W(0)\end{equation}which, thanks to the negative
definite of $A_0^1 = A^{11}(0)$ and the special form of $M_2$, is
estimated as
\begin{equation} I_b^0\le -\frac {\theta_1}{2}|w^{I}(0)|^2 + C(|w^{II}(0)|^2 +
|w^{II}_x(0)||w^{II}(0)|)\end{equation} for the outflow case, and
similarly
\begin{equation} I_b^0\le C(|W(0)|^2
+ |w^{II}_x(0)||w^{II}(0)|)\end{equation}for the inflow case.

Therefore together with this, \eqref{Fzeroth-est0} yields
\begin{align} \frac 12 \dt&\wprod{A^0W,W}\notag\\&\le -\frac 14[\wprod{\omega(x)w^{I},w^{I}} +\theta
\|w^{II}_x\|_0^2]+ C\zeta\|w^{I}\|_0^2 +
C(c_*)\|w^{II}\|^2_0+I_b^0.\label{Fzeroth-est}
\end{align}

\subsubsection{First order ``Friedrichs-type'' estimate} Similarly as
above, we shall need the following key estimate. We compute by the
use of integration by parts, \eqref{est-v}, plus $c_*$ being
sufficiently large,
\begin{align}-\wprod{W_x,AW_{xx} } &=\frac12
\wprod{W_x,(A_x+(\alpha_x/\alpha)A)W_x} +\frac 12
A_0W_x(0)\cdot W_x(0)\notag\\&\le -\frac
14\wprod{\omega(x)w^{I}_x,w^{I}_x}+C\zeta\|w^{I}_x\|_0^2
+Cc_*^2\|w^{II}_x\|_0^2\\
&\quad +\frac 12 A_0W_x(0)\cdot W_x(0).\label{ineq-key1}
\end{align}
We shall deal with the boundary term later. Now let us compute
\begin{align}\label{F1} \frac 12 \dt &\wprod{A^0W_x,W_x}
=\wprod{W_x,(A^0W_t)_x}-\wprod{W_x,A^0_xW_t} +\frac12
\wprod{A^0_tW_x,W_x}.
\end{align}
We shall control each term in turn. First we have
\begin{align*}\wprod{A^0_tW_x,W_x} \le C\zeta\|W_x\|_0^2
\end{align*} and by multiplying $(A^0)^{-1}$ into
\eqref{perturb-eqs}, \begin{align*} |\wprod{W_x,A^0_xW_t}| \le&
C\wprod{(\zeta+|\bW_x|)|W_x|, (|W_x|+|w^{II}_{xx}|+|W|)}\\\le&
\xi\|w^{II}_{xx}\|_0^2 + C\wprod{(\zeta+|\bW_x|)w^{I}_x,w^{I}_x}\\&+
C\wprod{(\zeta+|\bW_x|)w^{I},w^{I}} + C\|w^{II}\|_1^2,
\end{align*}where, once again, the term $\wprod{|\bW_x|w^{I}_x,w^{I}_x}$ may be treated the same way as \eqref{est-v}. Using \eqref{perturb-eqs}, we write the first right-hand term in \eqref{F1} as
\begin{align*}\wprod{W_x,(A^0W_t)_x}=&\wprod{W_x,[-AW_x +
(BW_x)_x+M_1\bW_x+(M_2\bW_x)_x]_x}\\=&-\wprod{W_x,AW_{xx}}
+\wprod{W_x,-A_xW_x+(M_1\bW_x)_x}\\&-\wprod{W_{xx}+(\alpha_x/\alpha)W_x,[(BW_x)_x+(M_2\bW_x)_x]}\\&-W_x(0).[(BW_x)_x+(M_2\bW_x)_x](0)
\\\le&-\frac 14\Big[\wprod{\omega(x)w^{I}_x,w^{I}_x}+\theta\|w^{II}_{xx}\|_0^2\Big]\\&+C\Big[\zeta\|w^{I}\|_1^2
+C(c_*)\|w^{II}_x\|_0^2+\wprod{|\bW_x|w^{I},w^{I}}\Big]+I_b^1\end{align*}
where $I_b^1$ denotes the boundary terms
\begin{equation}\label{Ib}I_b^1:=\frac 12 A_0W_x(0)\cdot W_x(0)
-W_x(0)\cdot[(BW_x)_x+(M_2\bW_x)_x](0),\end{equation} and we have
used estimate \eqref{ineq-key1},\eqref{est-v} for $w^{I},w^{I}_x$,
and applied Young's inequality to obtain:
\begin{align*}
\wprod{W_x,-A_xW_x+(M_1\bW_x)_x}&\le C\wprod{(\zeta+|\bW_x|)|W_x|,|W_x|+|W|}.
 \\
-\wprod{W_{xx}+(\alpha_x/\alpha)W_x,(BW_x)_x}&\le
 -\theta\|w^{II}_{xx}\|_0^2 \\&\quad+ C\wprod{|w^{II}_{xx}|+\omega(x)|w^{II}_x|,(\zeta+|\bW_x|)|w^{II}_x|}
\\
-\wprod{W_{xx}+(\alpha_x/\alpha)W_x,(M_2\bW_x)_x}&\le\\
C&\wprod{|w^{II}_{xx}|+\omega(x)|w^{II}_x|,(\zeta+|\bW_x|)(|W_x|+|W|)}.
\end{align*}

Putting these estimates together into \eqref{F1}, we have obtained
\begin{align} \frac 12 \dt \wprod{A^0W_x,W_x}  &+\frac 14\theta\|w^{II}_{xx}\|^2_{0}+\frac 14
\wprod{\omega(x)w^{II}_x,w^{II}_x}\notag\\&\le
C\Big[\zeta\|w^{II}\|_1^2
+\wprod{|\bW_x|w^{I},w^{I}}+C(c_*)\|w^{II}\|_1^2\Big]+I_b^1.\label{ineq-key2}
\end{align}
Let us now treat the boundary term. First observe that using the
parabolic equations with noting that $A^0$ is the diagonal-block
matrix diag$(A^0_1,A^0_2)$, we can write
\begin{align*}\Big|W_x(0)\cdot[(BW_x)_x&+(M_2\bW_x)_x](0) \Big|\\&=\Big|w^{II}_x(0)\cdot[(bw^{II}_x)_x+(M_2^{22}\bW_x)_x](0)
\Big|\\&= \Big|w^{II}_x(0)\cdot[A^0_2w^{II}_t +
A_{21}w^{I}_x+A_{22}w^{II}_x-M_1\bW_x](0)\Big|\\&\le\epsilon|w^{II}_x(0)|^2+C(|W(0)|^2+|w^{II}_x(0)|^2+|w^{II}_t(0)|^2).\end{align*}

Whereas, for the first term in $I_b$, we shall consider two cases
separately. When $A^{11}\le -\theta_1<0$, we get
$$A_0W_x(0)\cdot W_x(0)\le -\frac{
\theta_1}2|w^{I}_x(0)|^2 + C|w^{II}_x(0)|^2.$$ Therefore
\begin{align}\label{Ibd1}I_b^1 \le -\frac{\theta_1}2|w^{I}_x(0)|^2+ C(|W(0)|^2+|w^{II}_x(0)|^2+|w^{II}_t(0)|^2).
\end{align}

Meanwhile, for the case $A^{11}\ge \theta_1>0$, we have
$$|A_0W_x(0)\cdot W_x(0)|\le C|W_x(0)|^2.$$ Now the invertibility of $A^{11}$ allows us to use
the hyperbolic equation to derive
\begin{align*}|w^{I}_x(0)|&\le C(|w^{I}_t(0)|
+ |w^{II}_x(0)|).\end{align*}Therefore in the case of $A^{11}\ge
\theta_1>0$, we get \begin{align} I_b^1 \le
C(|W(0)|^2+|W_t(0)|^2+|w^{II}_x(0)|^2).\end{align}

Now applying the standard Sobolev inequality (applies for
$\alpha-$weighted norms as long as $|\alpha_x/\alpha|$ is uniformly
bounded): \begin{equation}|w(0)|^2 \le
C\|w\|_{L^2}(\|w_{x}\|_{L^2}+\|w\|_{L^2})\end{equation} to control
the term $|w^{II}_x(0)|^2$ in $I_b^1$ in both cases. We get
\begin{equation}\label{Sob-ineq}|w^{II}_x(0)|^2\le \epsilon'\|w^{II}_{xx}\|_0^2 +
C\|w^{II}_x\|^2_0.\end{equation}

Using this with $\epsilon'=\theta/8$, \eqref{Ib}, and \eqref{Ibd1},
the estimate \eqref{ineq-key2} reads
\begin{align} \frac 12 \dt \wprod{A^0W_x,W_x} &+\frac\theta8\|w^{II}_{xx}\|^2_{0}+\frac 14
\wprod{\omega(x)w^{I}_x,w^{I}_x }\notag\\&\le
C\Big(\zeta\|w^{I}\|_1^2 +\wprod{|\bW_x|w^{I},w^{I}}+
C(c_*)\|w^{II}\|_{1}^2\Big)+ I_b^1\label{ineq-key3}
\end{align}where the boundary term $I_b^1$ is estimated as
\begin{equation} I_b^1\le -\frac {\theta_1}{2}|w^{I}_x(0)|^2 + C(|W(0)|^2+|w^{II}_t(0)|^2)\end{equation} for the outflow case, and similarly
\begin{equation} I_b^1\le C(|W(0)|^2+|W_t(0)|^2)\end{equation}for the inflow case.

\subsubsection{Higher order ``Friedrichs-type'' estimate} Similarly
as above, we shall now derive an estimate for
$\wprod{A^0\partial^k_xW,\partial^k_xW}$, $k=2,3,4$. We need the
following key estimate. Integration by parts and \eqref{alpha-est}
give\begin{equation}
\begin{aligned}-\langle\partial^k_xW,A\partial^{k+1}_xW\rangle  =&\frac12
\wprod{\partial^k_xW,(A_x+(\alpha_x/\alpha)A)\partial^k_xW} +\frac
12 A_0\partial^k_xW(0)\cdot\partial^k_xW(0)\notag\\\le& -\frac
14\wprod{\omega(x)\partial^k_xw^{I},\partial^k_xw^{I}}+C\zeta\|\partial^k_xw^{I}\|_0^2
\\&+Cc_*^2\|\partial^k_xw^{II}\|_0^2+\frac 12
A_0\partial^k_xW(0)\cdot\partial^k_xW(0).\label{ineq-key11}
\end{aligned}\end{equation}

We compute
\begin{align} \frac 12 \dt \wprod{A^0\partial^k_xW,\partial^k_xW}=&\frac12 \wprod{A^0_t\partial^k_xW,\partial^k_xW}+\wprod{A^0\partial^k_xW,\partial^k_xW_t}
\notag\\=&\frac12 \wprod{A^0_t\partial^k_xW,\partial^k_xW}+\langle
A^0\partial^k_xW,\partial^k_x[(A^0)^{-1}\notag\\&(-AW_x +(BW_x)_x) +
M_1 \bW_x + (M_2 \bW_x)_x]\rangle.\label{F-higher}
\end{align}
We shall estimate each term in turn. First,
$|\wprod{A^0_t\partial^k_xW,\partial^k_xW}|\le C\zeta \|\partial_x^k
W\|_0^2$, and
\begin{align*}&\wprod{A^0\partial^k_xW,\partial^k_x[-(A^0)^{-1}AW_x
]}
\\&=\wprod{A^0\partial^k_xW,\sum_{i=0}^k\partial^i_x[-(A^0)^{-1}A]\partial^{k-i+1}_xW
}\\&= -\wprod{\partial^k_xW,A\partial^{k+1}_xW } +
\sum_{i=1}^k\wprod{A^0\partial^k_xW,\partial^i_x[-(A^0)^{-1}A]\partial^{k-i+1}_xW
}
\end{align*}
where we have
\begin{equation}
\Big|\partial_x^i[-(A^0)^{-1}A]\Big| \le C
\sum_{\sum{\alpha_j =i}} \prod_{1\le j \le
i}|\partial_x^{\alpha_j}W|.
\end{equation} 
Using the hypothesis on the boundedness of solutions in
$W^{2,\infty}$ and weak Moser inequality \cite[Lemma 1.5]{Z4}, we
get
$$
\begin{aligned}
&|\wprod{A^0\partial _x^kW,\partial _x^i[-(A^0)^{-1}A]\partial_x^{k-i+1}W}|
\le \\
&
\qquad \qquad
C\Big(\|w^{II}\|_k^2 + \zeta\|w^{I}\|_k^2 +
\sum_{i=1}^k\wprod{|\bW_x|\partial_x^iw^{I},\partial_x^iw^{I}}\Big).
\end{aligned}
$$

This, \eqref{ineq-key11}, similar treatment \eqref{est-v} for
$\wprod{|\bW_x|\partial_x^kw^{I},\partial_x^kw^{I}}$ with $c_*$
being sufficiently large give
\begin{align}\wprod{A^0\partial^k_xW,\partial^k_x[-(A^0)^{-1}AW_x
]}&\le -\frac 14\wprod{\omega\partial_x^kw^{I},\partial_x^kw^{I}} +
\frac
12A_0\partial_x^kW(0)\cdot\partial_x^kW(0)\notag\\&+C\Big(\|w^{II}\|_k^2
+ \zeta\|w^{I}\|_k^2 +
\sum_{i=1}^{k-1}\wprod{|\bW_x|\partial_x^iw^{I},\partial_x^iw^{I}}\Big)\end{align}

Next, similarly, we obtain $$|\wprod{A^0\partial_x^k W,
\partial_x^k[(A^0)^{-1}M_1\bW_x]}| \le C\Big(\|w^{II}\|_k^2 +
\zeta\|w^{I}\|_k^2 +
\sum_{i=1}^k\wprod{|\bW_x|\partial_x^iw^{I},\partial_x^iw^{I}}\Big).$$

Finally, we compute and \begin{align*} \langle
A^0\partial^k_xW,&\partial^k_x[(A^0)^{-1}(BW_x+M_2\bW_x)_x]\rangle
\\=& \sum_{i=0}^k\wprod{A^0\partial^k_xW,\partial^{i}_x[(A^0)^{-1}]\partial^{k-i+1}_x(BW_x + M_2\bW_x)}
\\=&\wprod{\partial^k_xW,\partial^{k+1}_x(BW_x + M_2\bW_x)}\\&+
\sum_{i=1}^k\wprod{A^0\partial^k_xW,\partial^{i}_x[(A^0)^{-1}]\partial^{k-i+1}_x(BW_x
+ M_2\bW_x)}
\\\le&-\wprod{\partial^{k+1}_xW + (\alpha_x/\alpha)\partial^k_xW,\partial^k_x(BW_x + M_2\bW_x)} \\&- \partial_x^k[b\partial _xw^{II} + M_2^{22}\bW_x](0)\partial_x^kw^{II}(0)
\\& +\xi \|\partial_x^{k+1}w^{II}\|_0^2 +
C\Big(c_*^2\|w^{II}\|_k^2 + \zeta\|w^{I}\|_k^2 +
\sum_{i=1}^k\wprod{|\bW_x|\partial_x^iw^{I},\partial_x^iw^{I}}\Big)\\\le&-\frac\theta2\|\partial_x^{k+1}w^{II}\|_0^2-\partial_x^k[b\partial
_xw^{II} + M_2^{22}\bW_x](0)\partial_x^kw^{II}(0)\\&\quad +
C\Big(c_*^2\|w^{II}\|_k^2 + \zeta\|w^{I}\|_k^2 +
\sum_{i=1}^k\wprod{|\bW_x|\partial_x^iw^{I},\partial_x^iw^{I}}\Big)
\end{align*}
where in the last inequality we used the special form of $B$ and
$M_2$ to get \begin{align*}\langle&\partial^{k+1}_xW +
(\alpha_x/\alpha)\partial^k_xW,\partial^k_x(BW_x +
M_2\bW_x)\rangle\\&\le\langle|\partial^{k+1}_xw^{II}| +
\omega(x)|\partial^k_xw^{II}|,|\partial^k_x(bw^{II}_x +
\Pi_2M_2\bW_x)|\rangle\\&\le -\theta\|\partial_x^{k+1}w^{II}\|_0^2+
C\Big(C(c_*)\|w^{II}\|_k^2 + \zeta\|w^{I}\|_k^2 +
\sum_{i=1}^k\wprod{|\bW_x|\partial_x^iw^{I},\partial_x^iw^{I}}\Big).\end{align*}
Note that in the last inequality, there is no term of
$\wprod{\omega(x)\partial_x^iw^{I},\partial_x^iw^{I}}$ because of
the presence of $|\bW_x|$ in term of $\Pi_2 M_2$.

Put all these estimates into \eqref{F-higher} together, we have obtained
\begin{align}
\frac 12\dt &\wprod{A^0\partial_x^kW,\partial_x^kW}
+ \frac 14\theta\|\partial^{k+1}_xw^{II}\|_0^2+\frac 14
\wprod{\omega(x)\partial^k_xw^{I},\partial^{k}_xw^{I} }\notag\\&\le
C\Big(C(c_*)\|w^{II}\|_k^2 + \zeta\|w^{I}\|_k^2 +
\sum_{i=1}^{k-1}\wprod{|\bW_x|\partial_x^iw^{I},\partial_x^iw^{I}}\Big)+I_b
\label{F-higher1}\end{align} where the boundary term
\begin{equation}\label{bdry-terms}I_b:=\frac
12A_0\partial_x^kW(0)\cdot\partial_x^kW(0)- \partial_x^k[b\partial
_xw^{II} + M_2^{22}\bW_x](0)\partial_x^kw^{II}(0).\end{equation}

For this boundary term, we shall treat the same as we did before. First using the parabolic equations with noting that $A^0$ is the diagonal-block matrix diag$(A^0_1,A^0_2)$, we can write
\begin{align}\label{bdry-est}
\partial_x^k &[b\partial _xw^{II} + M_2^{22}\bW_x](0)\notag \\
&= \partial^{k-1}_x[A^0_2(0)w^{II}_t(0,t) + A_{21}w^{I}_x(0)+A_{22}w^{II}_x(0) - \Pi_2 M_1(0)\bW_x(0)]
.\end{align}

Therefore we get
\begin{align}\label{bdry-treat}|\partial_x^k[b\partial _xw^{II} &+ M_2^{22}\bW_x](0)\partial_x^kw^{II}(0)|
\notag\\&\le
C|\partial^{k}_xw^{II}(0)|\Big[|\partial^{k-1}_xw^{II}_t(0)|+\sum_{i=0}^k(|\partial_x^iw^{II}(0)|
+ |\partial_x^iw^{I}(0)|)\Big]\notag\\&\le
\epsilon\sum_{i=0}^k|\partial_x^iw^{I}(0)|^2 +
C\sum_{i=1}^k|\partial_x^iw^{II}(0)|^2\\
&\quad +C|\partial^{k}_xw^{II}(0)||\partial^{k-1}_xw^{II}_t(0)|\end{align}
for any $\epsilon$ small. To deal with the term of $w^{II}_t$, for
simplicity, assume $k=3$. By solving the parabolic-part equations
and using the invertibility of $b$, we obtain
\begin{equation}\begin{aligned} |\partial_x^2 w^{II}_t| &= |\partial_t
w^{II}_{xx}| \le C(|w^{II}_{tt}| + |W_t| + |W_x|+|W_{xt}|)\\|W_{xt}|
&\le C(|W|+|W_x|+|W_{xx}| + |w^{II}_{xxx}|).
\end{aligned}\end{equation}

Since for case $k=3$ we have a good term
$\|\partial^{4}_xw^{II}\|_0$ (see \eqref{F-higher1}), the term
$|w^{II}_{xxx}(0)|$ resulting from the boundary treatment is easily
treated via Sobolev embedding inequality. Hence all terms in a form
$\partial_x^r w^{II}(0)$ are easily estimated. Meanwhile, using the
hyperbolic-part equations, we have
\begin{equation} |w^{I}_t|\le C(|W|+|W_x|).\end{equation}

Employing Young's inequality to the last term in \eqref{bdry-treat},
we obtain
\begin{align}\label{bdry-treat1}|\partial_x^k[b\partial _xw^{II} &+ M_2^{22}\bW_x](0)\partial_x^kw^{II}(0)|
\notag\\&\le \epsilon\sum_{i=0}^k|\partial_x^iw^{I}(0)|^2 +
C(\sum_{i=0}^{k}|\partial_x^iw^{II}(0)|^2+|w^{II}_t(0)|^2+|w^{II}_{tt}(0)|^2)\end{align}
 To deal with the term of $w^{I}$, we need to consider two cases
separately. When $A^{11}\le -\theta_1<0$, we get
$$A_0\partial_x^kW(0)\cdot\partial_x^kW(0)\le
-\frac{\theta_1}2|\partial_x^kw^{I}(0)|^2 +
C|\partial_x^kw^{II}(0)|^2.$$ Therefore
\begin{align}\label{Ibk}I^k_b \le& -\frac{\theta_1}2|\partial_x^kw^{I}(0)|^2+ C(\sum_{i=0}^{k-1}|\partial_x^iw^{I}(0)|^2 \notag\\&+
 \sum_{i=0}^{k}|\partial_x^iw^{II}(0)|^2+|w^{II}_t(0)|^2+|w^{II}_{tt}(0)|^2).
\end{align}

Meanwhile, for the case $A^{11}\ge \theta_1>0$, we have
$$|A_0\partial_x^kW(0)\cdot\partial_x^kW(0)|\le C(|\partial_x^kw^{I}(0)|^2 +
|\partial_x^kw^{II}(0)|^2).$$ The invertibility of $A^{11}$ allows
us to use the hyperbolic equation to derive
\begin{align*}|\partial_x^kw^{I}(0)|&\le C(\sum_{i=0}^{k}(|\partial_x^iw^{II}(0)|^2+|\partial_t^iw^{I}(0)|^2)+|w^{II}_t(0)|^2+|w^{II}_{tt}(0)|^2)
.\end{align*}Therefore in the case of $A^{11}\ge \theta_1>0$, we get
\begin{align} I^k_b \le C(\sum_{i=0}^{k}(|\partial_x^iw^{II}(0)|^2+|\partial_t^iw^{I}(0)|^2)+|w^{II}_t(0)|^2+|w^{II}_{tt}(0)|^2).\end{align}

Employing the boundary estimates into \eqref{F-higher1}, we have
obtained
\begin{align} \dt& \wprod{A^0\partial^k_xW,\partial^k_xW} +\theta\|\partial^{k+1}_xw^{II}\|^2_{0}+c_*\theta_1
\wprod{|\bW_x|\partial^k_xw^{I},\partial^{k}_xw^{I} }\notag\\&\le
C\Big(\zeta\|w^{I}\|_k^2 +
c_*^2\|w^{II}\|_{k}^2+\sum_{j=0}^{k-1}\wprod{|\bW_x|\partial^j_xw^{I},\partial^j_xw^{I}}
\Big)+I_b^k\label{ineq-key30}
\end{align}where, after absorbing the terms of $|\partial^r _xw^{II}(0)|$ via Sobolev embedding, the boundary term $I_b^k$ satisfies
\begin{align}\label{Ibk-out}I^k_b \le -\frac{\theta_1}2|\partial_x^kw^{I}(0)|^2+ C(\sum_{i=0}^{k-1}|\partial_x^iw^{I}(0)|^2 +|w^{II}_t(0)|^2+|w^{II}_{tt}(0)|^2)
\end{align}for outflow case, and
\begin{align}\label{Ibk-in} I^k_b \le
C(\sum_{i=0}^{k}|\partial_t^iw^{I}(0)|^2+|w^{II}_t(0)|^2+|w^{II}_{tt}(0)|^2)\end{align}for
the inflow case.

We shall establish an Kawashima-type estimate to bound the term
$\|w^{I}\|_k^2$ appearing on the left hand side of the above.

\subsubsection{``Kawashima-type'' estimate} Let $K$ be the
skew-symmetry in \eqref{K1}. Integration by parts and skew-symmetry
property of $K$ yield \begin{align*}\wprod{KW_{xt},W} &=
-\wprod{KW_t,W_x}-\wprod{(K_x+(\alpha_x/\alpha)K)W_t,W}
-K_0W_0\cdot(W_0)_t
\\&=\wprod{KW_x,W_t}+\wprod{(K_x+(\alpha_x/\alpha)K)W,W_t} -K_0W_0\cdot(W_0)_t
.\end{align*} Using this, we compute
\begin{align*}
\dt&\wprod{KW_x,W}=\\
&\wprod{K_tW_x+KW_{xt},W} + \wprod{KW_x,W_{t}}
\\=&\wprod{K_tW_x,W}+\wprod{2KW_x+(K_x+(\alpha_x/\alpha)K)W,W_t}\\&-K_0W_0\cdot(W_0)_t
\\=&\wprod{K_tW_x,W}+\wprod{2KW_x+(K_x+(\alpha_x/\alpha)K)W,-(A^0)^{-1}AW_x}
\\&+\langle
2KW_x+(K_x+(\alpha_x/\alpha)K)W,(A^0)^{-1}(BW_x)_x\\&+M_1\bW_x+(M_2\bW_x)_x\rangle
-K_0W_0\cdot(W_0)_t
\\
\le&-2\wprod{K(A^0)^{-1}AW_x,W_x} + \xi\|w^{I}_x\|_0^2
-K_0W_0\cdot(W_0)_t\\&+
C\Big(C(c_*)\|w^{II}\|_{2}^2+\zeta\|w^{I}\|_0^2
+\wprod{\omega(x)w^{I},w^{I}}+\wprod{\omega(x)w^{I}_x,w^{I}_x}\Big).\end{align*}
Using \eqref{K1}, we get
$$\wprod{K(A^0)^{-1}AW_x,W_x} \ge \theta_2 \|w^{I}_x\|_0^2 -
C(c_0)\|w^{II}_{x}\|_0^2,$$ and thus obtain from the above estimate
with $\xi = \theta_2/2$
\begin{align}\label{K-1st-est}\dt\wprod{KW_x,W} \le & -\frac{\theta_2}{2}\|w^{I}_x\|_0^2
+C\Big(C(c_*)\|w^{II}\|_{2}^2+\zeta\|w^{I}\|_0^2 \notag
\\&+\wprod{\omega(x)w^{I},w^{I}}+\wprod{\omega(x)w^{I}_x,w^{I}_x}\Big)-K_0W_0\cdot(W_0)_t.\end{align}

\subsubsection{Higher order ``Kawashima-type'' estimate} With similar
calculations, we shall obtain an estimate for
$\dt\wprod{K\partial_x^{k}W,\partial_x^{k-1}W}, k\ge 1$. We compute
\begin{align*}\langle
K\partial_x^{k}W_{t}&,\partial_x^{k-1}W\rangle=\wprod{K\partial_x^{k}W,\partial_x^{k-1}W_t}\\
&+\wprod{(K_x+(\alpha_x/\alpha)K)\partial_x^{k-1}W,\partial_x^{k-1}W_t}
- K\partial_x^{k-1}W_t\cdot\partial_x^{k-1}W (0) .\end{align*} and
hence
\begin{align*}\dt&\wprod{K\partial_x^{k}W,\partial_x^{k-1}W}=\wprod{K_t\partial_x^{k}W,\partial_x^{k-1}W}+\wprod{2K\partial_x^{k}W,\partial_x^{k-1}W_t}
\\&\quad+\wprod{(K_x+(\alpha_x/\alpha)K)\partial_x^{k-1}W,\partial_x^{k-1}W_t}- K\partial_x^{k-1}W_t\cdot\partial_x^{k-1}W (0)
\\&=\wprod{2K\partial_x^{k}W,\partial_x^{k-1}[(-A^0)^{-1}(AW_x+(BW_x)_x+M_1\bW_x+(M_2\bW_x)_x)]}
\\&\quad+\langle (K_x+(\alpha_x/\alpha)K)\partial_x^{k-1}W,\\&\quad\quad\quad\partial_x^{k-1}[(-A^0)^{-1}(AW_x+(BW_x)_x+M_1\bW_x+(M_2\bW_x)_x)]\rangle\\&\quad- K\partial_x^{k-1}W_t\cdot\partial_x^{k-1}W (0)\\&
\le-2\wprod{K(A^0)^{-1}A\partial_x^{k}W,\partial_x^{k}W} +
\epsilon\|w^{I}\|_k^2 +
Cc_*^2\|w^{II}\|_{k+1}^2\\&\quad+C\zeta\|w^{I}\|_0^2
+C\sum_{l=1}^k\wprod{\omega(x)\partial_x^{l}w^{I},\partial_x^{l}w^{I}}-
K\partial_x^{k-1}W_t\cdot\partial_x^{k-1}W (0)\end{align*} for
$\epsilon$ small.

Using \eqref{K1}, we obtain from the above
\begin{align}\label{K2}
\dt\wprod{K\partial_x^{k}W,\partial_x^{k-1}W}
\le&-\frac{\theta_2}{3}\|\partial_x^k w^{I}\|_0^2+
Cc_*^2\|w^{II}\|_{k+1}^2+ \epsilon \|w^{I}\|_{k-1}\\
&\quad + C \zeta\|w^{I}\|_0^2\notag
+C\sum_{l=1}^k\wprod{\omega(x)\partial_x^{l}w^{I},\partial_x^{l}w^{I}}\\
&\quad - K\partial_x^{k-1}W_t\cdot\partial_x^{k-1}W (0).
\end{align}

\subsubsection{Final estimates} We are ready to conclude our result.
First combining the estimate \eqref{ineq-key3} with
\eqref{Fzeroth-est}, we easily obtain
\begin{align*} \frac 12 \dt& \Big(\wprod{A^0W_x,W_x} + M\wprod{A^0 W,W}\Big)
 \\\le& -\Big(\frac\theta8\|w^{II}_{xx}\|^2_{0}+\frac 14\wprod{\omega(x)w^{I}_x,w^{I}_x }\Big)\\&+
C\Big(\zeta\|w^{I}\|_1^2  +\wprod{|\bW_x|w^{I},w^{I}}+
C(c_*)\|w^{II}\|_{1}^2\Big) + I_b^1
\\&-\frac M4\Big(\wprod{\omega(x) w^{I},w^{I}} +\theta
\|w^{II}_x\|_0^2\Big)+ CM\zeta\|w^{I}\|_0^2 \\
&\quad + MC(c_*)\|w^{II}\|^2_0 + MI_b^0
\end{align*}

By choosing $M$ sufficiently large such that $M\theta \gg C(c_*)$,
and noting that $c_*\theta_1 |\bW_x|\le \omega(x)$, we get
\begin{equation}\label{F-combine01}\begin{aligned} \frac 12 \dt&
\Big(\wprod{A^0W_x,W_x} + M\wprod{A^0 W,W}\Big)\\\le&
-\Big(\theta\|w^{II}\|^2_{2}+\wprod{\omega(x) w^{I},w^{I}}+
\wprod{\omega(x)w^{I}_x,w^{I}_x }\Big)\\&+ C\Big(\zeta\|w^{I}\|_1^2+
C(c_*)\|w^{II}\|^2_0\Big) + I_b^1 + MI_b^0.
\end{aligned}\end{equation}
We shall treat the boundary terms later. Now we employ the estimate
\eqref{K-1st-est} to absorb the term $\|w^{I}\|_1$ into the left
hand side. Indeed, fixing $c_*$ large as above, adding
\eqref{F-combine01} with \eqref{K-1st-est} times $\epsilon$, and
choosing $\epsilon ,\zeta$ sufficiently small such that $\epsilon
C(c_*)\ll \theta, \epsilon \ll 1$ and $\zeta \ll \epsilon \theta_2$,
we obtain

\begin{align*} \frac 12 \dt& \Big(\wprod{A^0W_x,W_x} +
M\wprod{A^0 W,W} + \epsilon \wprod{KW_x,W} \Big) \notag\\\le&
-\Big(\theta\|w^{II}\|^2_{2}+\wprod{\omega(x) w^{I},w^{I}}+
\wprod{\omega(x)w^{I}_x,w^{I}_x }\Big)\\&+ C\Big(\zeta\|w^{I}\|_1^2+
C(c_*)\|w^{II}\|^2_0\Big)
-\frac{\theta_2\epsilon}{2}\|w^{I}_x\|_0^2\\&+
C\epsilon\Big(C(c_*)\|w^{II}\|_{2}^2+\zeta\|w^{I}\|_0^2+\wprod{\omega(x)
w^{I},w^{I}}+ \wprod{\omega(x)w^{I}_x,w^{I}_x }\Big)\\&+ I_b^1 +
MI_b^0-\epsilon K_0W_0\cdot(W_0)_t
\\\le &-\frac 12\Big(\theta\|w^{II}\|^2_{2} + \theta_2\epsilon\|w^{I}_x\|_0^2\Big)+
C(c_*)\Big(\zeta\|w^{I}\|_0^2+ \|w^{II}\|^2_0\Big)+ I_b
\end{align*} where $I_b:=I_b^1 +
MI_b^0-\epsilon K_0W_0\cdot(W_0)_t$.

By a view of boundary terms $I_b^0,I_b^1$, we treat the term $I_b$
in each inflow/outflow case separately. Recalling the inequality
\eqref{Sob-ineq}, $|w^{II}_x(0)|\le C\|w^{II}\|_2$. Thus, using
this, for the inflow case we have
\begin{align*} I_b&\le M|W(0)|^2 + C|W_t(0)|^2 + M|w^{II}_x(0)||w^{II}(0)|
\\&\le \frac{\theta}{2}\|w^{II}\|_2^2 + M^2|W(0)|^2 + C|W_t(0)|^2.\end{align*}
Meanwhile, for the outflow case, with $M\theta_1 \gg 1$ and
$K_0W_0\cdot(W_0)_t \sim w^{II}_0w^{I}_{0t} + w^{I}_0w^{II}_{0t}$,
we have $I_b$ is bounded by
\begin{align*}
-\frac{\theta_1}{2}&(|w^{I}_x(0)|^2 + |w^{I}(0)|^2)\\
&\quad + C(|w^{II}_t(0)|^2 +|w^{II}(0)|^2 )+
\epsilon (|w^{II}_x(0)|^2+ |w^{I}_{t}(0)|^2)\end{align*}
which,
together with $\epsilon$ being sufficiently small and the facts that
$$
|w^{I}_{t}(0)|\le C(|w^{I}_x(0)| + |w^{II}_x(0)| + |W(0)|)
$$
obtained from solving the hyperbolic equation and the embedding
inequality
$$
|w^{II}_x(0)|\le C\|w^{II}\|_2,
$$
yields
\begin{align*} I_b \le -\frac{\theta_1}{2}(|w^{I}_x(0)|^2 + |w^{I}(0)|^2)+
\frac{\theta}{2}\|w^{II}\|_2^2 + C(|w^{II}(0)|^2 +
|w^{II}_t(0)|^2)\end{align*}for the outflow case. Now by
Cauchy-Schwarz's inequality and by positivity definite of $A^0$, it
is easy to see that
\begin{equation}\cE:=\wprod{A^0W_x,W_x} + M\wprod{A^0 W,W} +
\epsilon \wprod{KW_x,W} \sim \|W\|_{H^1_\alpha}^2 \sim
\|W\|_{H^1}^2.\end{equation}  The last equivalence is due to the
fact that $\beta$ is bounded above and below away from zero. Thus
the above derives

\begin{align*}\dt \cE(W)(t)\le  - \theta_3 \cE(W)(t) + C(c_*)\Big(\|W(t)\|_{L^2}^2 +
\CalB_1(t)\Big),\end{align*} for some positive constant $\theta_3$,
which by the Gronwall inequality yields
\begin{equation}\label{energy-estimate}\|W(t)\|_{H^1}^2 \le Ce^{-\theta t}\|W_0\|_{H^1}^2 + C(c_*)\int_0^t e^{-\theta(t-\tau)}
\Big(\|W(\tau)\|_{L^2}^2 + \CalB_1(\tau)\Big)d\tau,\end{equation}
where $W(x,0)=W_0(x)$ and
\begin{equation} \CalB_1(\tau):= \cO(|W(0,\tau)|^2 +
|W_t(0,\tau)|^2) = \cO(|(h_1,h_2)|^2 +
|(h_1,h_2)_t|^2)\end{equation} for the inflow case, and
\begin{equation} \CalB_1(\tau):= \cO(|w^{II}(0,\tau)|^2 +
|w^{II}_t(0,\tau)|^2)= \cO(|h|^2 + |h_t|^2)\end{equation} for the
outflow case.

Similarly, by induction, we shall derive the same estimates for $W$
in $H^s$. To do that, let us define
\begin{align*} \cE_1(W)&:= \wprod{A^0W_x,W_x} + M\wprod{A^0 W,W} + \epsilon \wprod{KW_x,W}
\\\cE_k(W)&:= \wprod{A^0\partial_x^k W,\partial_x^kW} + M\cE_{k-1}(W) + \epsilon \wprod{K\partial_x^kW,\partial_x^{k-1}W}
.\end{align*}

Then by Cauchy-Schwarz inequality, it is easy to see that $\cE_k(W)
\sim \|W\|_{H^k}^2$, and by induction, we obtain
\begin{align*}\dt \cE_k(W)(t)\le  - \theta_3 \cE_k(W)(t) + C(c_*)(\|W(t)\|_{L^2}^2+\CalB_k(t)),\end{align*} for some positive constant $\theta_3$,
which by the Gronwall inequality yields
\begin{equation}\label{energy-estimate-higher}\|W(t)\|_{H^k}^2 \le Ce^{-\theta t}\|W_0\|_{H^k}^2 + C(c_*)\int_0^t
e^{-\theta(t-\tau)}(\|W(\tau)\|_{L^2}^2+\CalB_k(\tau))d\tau,\end{equation}
where $W(x,0)=W_0(x), k=2,3,4$, and $\CalB_k$ are defined as in
\eqref{Bdry-out} and \eqref{Bdry-in}.

\subsubsection{The general case} Following \cite{MaZ4}, the general
case that hypotheses (A1)-(A3) hold can easily be covered via
following simple observations. First, we may express matrix $A$ in
\eqref{perturb-eqs} as
\begin{align}\label{formA}A(W+\bW) = \hat A +
(\zeta + |\bW_x|)\begin{pmatrix}0 & \cO(1) \\
\cO(1) & \cO(1)\end{pmatrix} \end{align}
where $\hat A$ is a symmetric matrix obeying the same derivative
bounds as described for $A$, identical to $A$ in the $11$ block and
obtained in other blocks $jk$ by
\begin{align}\label{formAjk}A^{jk}(W+\bW) &=
A^{jk}(\bW)+A^{jk}(W+\bW) - A^{jk}(\bW)\notag\\&= A^{jk}(W_+) +
\cO(|W_x|+|\bW_x|)= A^{jk}(W_+) + \cO(\zeta+|\bW_x|).
\end{align}

Replacing $A$ by $\hat A$ in the $k^{th}$ order
Friedrichs-type bounds above, we find that the resulting error terms
may be expressed as
$$\wprod{\partial_x^k\cO(\zeta +
|\bW_x|)|W|,|\partial_x^{k+1}w^{II}|},$$ plus lower order terms,
easily absorbed using Young's inequality, and boundary terms
$$\cO(\sum_{i=0}^k|\partial_x^iw^{II}(0)||\partial^{k}_xw^{I}(0)|)$$ resulting from the use of integration by
parts as we deal with the $12-$block. However these boundary terms
were already treated somewhere as before (see \eqref{bdry-treat}).
Hence we can recover the same Friedrichs-type estimates obtained
above. Thus we may relax $(A1')$ to $(A1)$.

The second observation is that, because of the favorable terms
$$c_*\theta_1 \wprod{|\bW_x|\partial^k_xw^{I},\partial^{k}_xw^{I} }$$
occurring in the lefthand sides of the Friedrichs-type estimates
\eqref{ineq-key30}, we need the Kawashima-type bound only to control
the contribution to $|\partial_x^kw^{I}|^2$ coming from $x$ near
$+\infty$; more precisely, we require from this estimate only a
favorable term
$$-\theta_2 \wprod{(1-\cO(\zeta+|\bW_x|))\partial^k_xw^{I},\partial^{k}_xw^{I} }$$
rather than $\theta_2\|\partial_x^kw^{I}\|_0^2$ as in \eqref{K2}.
But, this may easily be obtained by substituting for $K$ a
skew-symmetric matrix-valued function $\hat K: = K(W_+)$, and using
the fact that
$$\R(K(A^0)^{-1}A + B)(W_+)\ge
\theta_2>0,$$ and same as \eqref{formAjk}, $K = \hat K +
\cO(\zeta+|\bW_x|)$, we have $$\R(K(A^0)^{-1}A + B)(W)\ge
\theta_2(1-\cO(\zeta+|\bW_x|))>0.$$ Thus we may relax $(A2')$ to
$(A2)$.

Finally, notice that the term $g(\tilde W_x)-g(\bW_x)$ in the
perturbation equation may be Taylor expanded as
$$\begin{pmatrix}0 \\
g_1(\tilde W_x,\bW_x)+g_1(\bW_x,\tilde W_x)\end{pmatrix}+\begin{pmatrix}0 \\
\cO(|W_x|^2)\end{pmatrix}$$ The first, since linear term on the
righthand side decays at plus spatial infinity and vanishes in the
1-1 block, it may be treated as follows by Young's inequality
$$\begin{pmatrix}0 \\
g_1(\tilde W_x,\bW_x)+g_1(\bW_x,\tilde W_x)\end{pmatrix}\begin{pmatrix}w^{I}_x \\
w^{II}_x\end{pmatrix}\le
C\Big(\wprod{(\zeta+|\bW_x|)w^{I}_x,w^{I}_x}
+\|w^{II}_x\|_{0}^2\Big)$$which can be treated in the
Friedrichs-type estimates. The $(0,O(|W_x|^2)$ nonlinear term may be
treated as other source terms in the energy estimates. Specifically,
the
worst-case term $$\wprod{\partial_x^k W,\partial_x^k\begin{pmatrix}0 \\
\cO(|W_x|^2)\end{pmatrix}} = -\wprod{\partial_x^{k+1}
w^{II},\partial_x^{k-1}\cO(|W_x|^2)}-\partial_x^k
w^{II}(0)\partial_x^{k-1}\cO(|W_x|^2)(0)$$ may be bounded by
$$\|\partial_x^{k+1}
w^{II}\|_{L^2}\|W\|_{W^{2,\infty}}\|W\|_{H^k} -\partial_x^k
w^{II}(0)\partial_x^{k-1}\cO(|W_x|^2)(0).$$ The boundary term will
be contributed into the form \eqref{bdry-terms} of $I_b$, and hence
using the parabolic equations to get rid of this term as treating in
\eqref{bdry-est}. Thus, we may relax $(A3')$ to $(A3)$, completing
the proof of the general case $(A1)-(A3)$ and the
proposition.\end{proof}

\subsection{Energy estimate II}
We require also the following estimate:
\begin{lemma}[\cite{HR}]\label{lem-w-EE} Under the hypotheses of Theorem \ref{theo-nonlin}, let $E_0 :=\|(1+|x|^2)^{3/4}U_0\|_{H^4}$, and suppose that, for $0\le t\le T$, the $W^{2,\infty}$ norm of the solution $U$ of \eqref{perteqs} remains bounded by some constant $C>0$. Then, for all $0\le t\le T$, \begin{equation} \|(1+|x|^2)^{3/4}U(x,t)\|_{H^4}^2 \le M E_0e^{Mt}. \end{equation}
\end{lemma}

\begin{proof} This follows by standard Friedrichs symmetrizer estimates carried out
in the weighted $H^4$ norm. \end{proof}

\begin{remark}\label{rem-wpoint-bounds} An immediate consequence of Lemma \ref{lem-w-EE}, by Sobolev embedding: $W^{3,\infty}\subset H^4$
and equation \eqref{perteqs}, is that, if $E_0$ and $ \|U\|_{H^4}$
are uniformly bounded on $[0,T]$, then
 \begin{equation} (1+|x|)^{3/2}\Big[|U|+|U_t|+|U_x|+|U_{xt}|\Big](x,t)
 \end{equation}
is uniformly bounded on $[0,T]$ as well.
\end{remark}

\section{Stability analysis}
In this section, we shall prove Theorems \ref{theo-lin} and
\ref{theo-nonlin}. Following \cite{HZ,MaZ3}, define the nonlinear
perturbation $U = (u, v)$ by
\begin{equation}
U(x,t):=\tilde U(x,t)-\bU(x),
\end{equation}
we obtain
\begin{equation}
\label{perteqs} U_t-LU=Q(U,U_x)_x,
\end{equation}
where linearized operator
\begin{equation}\label{linearized-operator}L U := - (A U)_x + ( B U_x)_x\end{equation} where
$$ A U:= d F(\bU) U - (dB(\bU)U) \bU_x, \quad  B =  B(\bU)$$
and the second-order Taylor remainder:
\begin{align*}
Q(U,U_x) &=F(\bU + U) - F(\bU) + A(\bU)U + (B(\bU + U) - B(\bU))U_x
\end{align*}
satisfying
\begin{equation}\label{newqbounds}
\begin{aligned}
|Q(U,U_x)|&\le C(|U||U_x|+|U|^2)\\
|\Pi_1Q(U,U_x)_x| &\le C(|U||U_x|+|U|^2)\\
|Q(U,U_x)_x|&\le C(|U||U_{xx}|+|U_x|^2+|U||U_x|)\\
|Q(U,U_x)_{xx}|&\le C(|U||U_{xx}|+|U||U_{xxx}|+|U_x||U_{xx}|+|U_x|^2)
\end{aligned}
\end{equation}
so long as $|U|$ remains bounded.

For boundary conditions written in $U-$coordinates, (B) gives
\begin{equation}\label{BCs-in}
\begin{aligned}
h(t)=\tilde h(t) -\bar h&=
(\tilde W(U+\bar U)-\tilde W(\bar U))(0,t)\\
&=(\partial \tilde W/\partial \tilde U)(\bar U_0) U (0,t)+
\cO(|U(0,t)|^2).
\end{aligned}
\end{equation}
in inflow case and
\begin{equation}\label{BCs-out}
\begin{aligned}
h(t)=\tilde h(t) -\bar h&=
(\tilde w^{II}(U+\bar U)-\tilde w^{II}(\bar U))(0,t)\\
&= (\partial \tilde w^{II}/\partial \tilde U)(\bar U_0)
U(0,t) + \cO(|U(0,t)|^2)\\
&= m\begin{pmatrix} \bar b_1 & \bar b_2 \end{pmatrix}(\bar U_0)
U(0,t) + \cO(|U(0,t)|^2)
\\&= m B(\bar U_0)U(0,t) + \cO(|U(0,t)|^2).
\end{aligned}
\end{equation}



%

\subsection{Integral formulation}
We obtain the following:
\begin{lemma}[Integral formulation] We have
\begin{equation}\label{u}
\begin{aligned}
  U(x,t)=&\int_0^\infty G(x,t;y)U_0(y)\,dy
  \\&+\int_0^t \Big(\tilde G_y(x,t-s;0)BU(0,s)+G(x,t-s;0) AU(0,s)\Big)\,ds\\
  &+ \int^t_0 \int^\infty_{0} H(x,t-s;y)
  \Pi_1Q(U,U_y)_y(y,s)\,dy\,ds\\
  &-\int^t_0 \int^\infty_{0} \tilde G_y(x,t-s;y)
  \Pi_2Q(U,U_y)(y,s)\,dy\,ds
\end{aligned}
\end{equation} where $U(y,0) = U_0(y)$. 
\end{lemma}

\begin{proof}
From the duality (see \cite[Lemma 4.3]{ZH}), we find that
$G(x,t-s;y)$ considered as a function of $y,s$ satisfies the adjoint
equation
\begin{equation}\label{timeadj}
(\partial_s- L_y)^*G^*(x, t- s; y)= 0,
\end{equation} or \begin{equation}\label{explicit}
-G_s -(GA)_y + GA_y= (G_yB)_y.
\end{equation}
 in the distributional sense, for all $x,y,t>s>0$,
where the adjoint operator of $L_y$ is defined by \begin{equation}
L_y^* V: = V^*_y A + (V^*_yB)_y,
\end{equation} with $V^* = V^{tr}$.

Likewise, for boundary conditions, we have,
by duality
\\
\\
(iii') for all $x,t>0$, $G(x,t;0)\equiv 0$ in the
outflow case $\bar A_*<0$; and
$G(x,t;0)B=0$
in the inflow case $\bar A_*>0$, noting that no
boundary condition
need be applied on the
hyperbolic part for the adjoint equations in
the inflow case.

Thus, integrating $G$ against \eqref{perteqs}, we obtain for any
classical solution that
\begin{equation}
\begin{aligned}
\int_0^t\int_0^\infty & G(x,t-s;y)Q(U,U_y)_y(y,s)\, dy\, ds=\\
& \int_0^t\int_0^\infty G(x,t-s;y)(\partial_s-L_y)U(y,s) \, dy\, ds
\\=&: I_1 + I_2.
\end{aligned}
\end{equation}

Integrating by parts and using the boundary conditions (iii')
on the boundary $y=0$, we get
\begin{align*} I_1 =& \int_0^t\int_0^\infty G(x,t-s;y)\partial_s U(y,s) \, dy\,
ds \\=&\int_0^t\int_0^\infty \partial_s G(x,t-s;y)U(y,s) \, dy\, ds
\\&+ \int_0^\infty G(x,0;y)U(y,t)\,dy - \int_0^\infty G(x,t;y)U(y,0)\,dy
\end{align*} where note that $$U(x,t) = \int_0^\infty
G(x,0;y)U(y,t)\,dy$$ and also
\begin{align*} I_2 =& \int_0^t\int_0^\infty G(x,t-s;y) (-L_y)U(y,s) \, dy\,
ds \\=&\int_0^t\int_0^\infty G(x,t-s;y)((AU)_y - (BU_y)_y)(y,s)\,
dy\, ds \\=&\int_0^t\int_0^\infty
(-G_yA - (G_yB)_y)
U(y,s)\, dy\, ds \\& - \int_0^t G_y(x,t-s;0)BU(0,s)ds - \int_0^t
G(x,t-s;0) AU(0,s) ds
\end{align*}

Combining these estimates, and noting that $G_yB=\tilde G_yB$ since
$HB\equiv 0$,  we obtain \eqref{u} by rearranging and
integrating by parts the last term of
\begin{equation}\begin{aligned}\int_0^t\int_0^\infty
&G(x,t-s;y)Q(U,U_y)_y(y,s) \, dy\, ds \\&= \int_0^t\int_0^\infty
(H+\tilde G)(x,t-s;y)Q(U,U_y)_y(y,s) \, dy\,
ds\end{aligned}\end{equation} \end{proof}

As an expression for $U_x$, we obtain the following.
\begin{lemma}[Integral formulation for $U_x$] We have
\begin{equation}\label{ux}
\begin{aligned}
  U_x(x,t)=&\int_0^\infty G_x(x,t;y)U_0(y)\,dy
  -\int_0^tH(x,t-s;0)\Pi_1Q(U,U_y)_y(0,s)\,ds\\
   &+\int_0^t \Big[\tilde G_{xy}(x,t-s;0)BU(0,s)+G_x(x,t-s;0)AU(0,s)\Big]\,ds\\
  &+ \int^t_0 \int^\infty_{0} (H_x-H_y)(x,t-s;y) \Pi_1Q(U,U_y)_y(y,s)\,dy\,ds\\
  &-\int^t_0 \int^\infty_{0} H(x,t-s;y) \Pi_1Q(U,U_y)_{yy}(y,s)\,dy\,ds\\
  &-\int^{t-1}_0 \int^\infty_{0} \tilde G_{xy}(x,t-s;y) \Pi_2Q(U,U_y)(y,s)\,dy\,ds\\
  &+\int^t_{t-1} \int^\infty_{0} \tilde G_{x}(x,t-s;y)\Pi_2Q(U,U_y)_y(y,s)\,dy\,ds
\end{aligned}
\end{equation}
where $U(y,0) = U_0(y)$.
\end{lemma}
\begin{proof} Differentiating the formulation \eqref{u} for $U(x,t)$ with respect to $x$ and noting that \begin{align*} \int^t_0 \int^\infty_{0} &H_x \phi\,dy\,ds =
\int^t_0 \int^\infty_{0} (H_x -
  H_y) \phi\,dy\,ds \\&- \int^t_0 \int^\infty_{0} H(x,t-s;y)
 \phi_y(y,s)\,dy\,ds - \int_0^t H(x,t-s;0)\phi(0,s)ds\end{align*}
and\begin{align*} \int^t_0 \int^\infty_{0} &\tilde
G_{xy}\psi\,dy\,ds = \int^{t-1}_0 \int^\infty_{0} \tilde G_{xy}
\psi\,dy\,ds \\&- \int_{t-1}^t \int^\infty_{0} \tilde G_{x}
\psi_y\,dy\,ds  - \int_{t-1}^t \tilde
G_x(x,t-s;0)\psi(0,s)ds\end{align*} are valid for any smooth
functions $\phi,\psi$, we obtain the lemma. \end{proof}

\subsection{Convolution estimates}
To establish stability, we use the following lemmas proved in
\cite{HZ,HR,RZ}.

\begin{lem}[Linear estimates I]\label{iniconvolutions}
Under the assumptions of Theorem \ref{theo-nonlin},
\begin{equation}\label{iniconeq}
\begin{aligned}
\int_{0}^{+\infty}|\tilde G(x,t;y)|(1+|y|)^{-3/2}\, dy &\le
C(\theta+\psi_1+\psi_2)(x,t), \\\int_{0}^{+\infty}|\tilde
G_x(x,t;y)|(1+|y|)^{-3/2}\, dy &\le C(\theta+\psi_1+\psi_2)(x,t),
\end{aligned}
\end{equation}
and so the latter is dominated by $\psi_1  + \psi_2$, for $0\le t\le
+\infty$, some $C>0$.
\end{lem}

\begin{lem}[Linear estimates II]\label{iniconvolutionsh}
Under the assumptions of Theorem \ref{theo-nonlin}, if
$|U_0(x)|+|\partial_xU_0(x)|\le E_0 (1+|x|)^{-3/2}, E_0>0$, then,
for some $\theta>0$,
\begin{equation}\label{iniconeqh}
\begin{aligned}
\int_{0}^{+\infty}H(x,t;y)U_0(y)\, dy & \le CE_0 e^{-\theta
t}(1+|x|)^{-3/2},\\\int_{0}^{+\infty}H_x(x,t;y)U_0(y)\, dy& \le CE_0
e^{-\theta t}(1+|x|)^{-3/2},
\end{aligned}
\end{equation}
and so both are dominated by $CE_0(\psi_1  + \psi_2)$, for $0\le
t\le +\infty$, some $C>0$.
\end{lem}

\begin{lem}[Nonlinear estimates I]\label{convolutionsI}
Under the assumptions of Theorem \ref{theo-nonlin},
\begin{equation}\label{coneq}
\begin{aligned}
\int_0^t\int_{0}^{+\infty}|\tilde G_y(x,t-s;y)|\Psi(y,s)\, dy ds
&\le C(\theta+\psi_1+\psi_2)(x,t),\\
\int_0^{t-1}\int_{0}^{+\infty}|\tilde G_{xy}(x,t-s;y)|\Psi(y,s)\, dy
ds
&\le C(\theta+\psi_1+\psi_2)(x,t),\\
\end{aligned}
\end{equation}
for $0\le t\le +\infty$, some $C>0$, where
\begin{equation}\label{source}
\begin{aligned}
\Psi(y,s)&:=(\theta + \psi_1+\psi_2)^2(y,s).\\
\end{aligned}
\end{equation}
\end{lem}

\begin{lem}[Nonlinear estimates II]\label{convolutionsII}
Under the assumptions of Theorem \ref{theo-nonlin},
\begin{equation}\label{coneqh}
\begin{aligned}
\int_0^t\int_{0}^{+\infty}H(x,t-s;y)\Upsilon(y,s)\, dy ds &\le C
(\psi_1+\psi_2)(x,t)\\
\int_0^t\int_{0}^{+\infty}(H_x-H_y)(x,t-s;y)\Upsilon(y,s)\, dy ds
&\le C (\psi_1+\psi_2)(x,t)\\
\int_{t-1}^t\int_{0}^{+\infty}|\tilde G_x(x,t-s;y)|\Upsilon(y,s)\,
dy ds &\le C (\psi_1+\psi_2)(x,t)
\end{aligned}
\end{equation}
for all $0<t< +\infty$, some $C>0$, where
\begin{equation}\label{sourceE}
\begin{aligned}
\Upsilon(y,s)&:= s^{-1/4}(\theta +
\psi_1+\psi_2)(y,s)
\end{aligned}
\end{equation}
\end{lem}

We require also the following estimate accounting boundary effects.

\begin{lem}[Boundary estimates I]\label{boundaryconvolutionI}
Under the assumptions of Theorem \ref{theo-nonlin}, if
$|h(t)|+|h'(t)|\le E_0(1+t)^{-1}$,
\begin{equation}\label{bconeqh}
\begin{aligned}
\int_{0}^{t}H(x,t-s;0)h(s)\, ds&\le
CE_0(\psi_1+\psi_2)(x,t)\\\int_{0}^{t}H_x(x,t-s;0)h(s)\, ds &\le
CE_0(\psi_1+\psi_2)(x,t),
\end{aligned}
\end{equation}
 for $0\le t\le +\infty$,
some $C>0$.
\end{lem}
\begin{proof} 
Note that
$H(x,t;0)\equiv 0$ for the outflow case $A_*<0$. Consider the inflow
case $A_*>0$ (and thus $\bar a_*>0$). We have
\begin{align*}
\Big|\int_{0}^{t}H(x,t-s;0)&h(s)\, ds\Big| \\&= e^{-\eta_0x/\bar
a_*}|h(-\frac 1{\bar a_*}(x-\bar a_*t))|\\&\le
e^{-\eta_0|x|}(1+|x-\bar a_*t|)^{-1}\le CE_0(\psi_1+\psi_2)(x,t),
\\\Big|\int_{0}^{t}H_x(x,t-s;0)&h(s)\, ds\Big| \\&\le
e^{-\eta_0x/\bar a_*}\Big[|h|+|h'|\Big](-\frac 1{\bar a_*}(x-\bar
a_*t))|\\&\le e^{-\eta_0|x|}(1+|x-\bar a_*t|)^{-1}\le
CE_0(\psi_1+\psi_2)(x,t),\end{align*} which completes the proof of the
lemma.
\end{proof}

\begin{lem}[Boundary estimates II]\label{boundaryconvolutionII}
Under the assumptions of Theorem \ref{theo-nonlin}, if
$|h(t)|+|h_t(t)|\le E_0(1+t)^{-1}$,
\begin{equation}\label{bconeq}
\begin{aligned}
\Big|\int_0^t \Big(\tilde G_y(x,t-s;0)Bh(s)+&G(x,t-s;0)
Ah(s)\Big)\,ds \Big|\\&\le CE_0(\theta+\psi_1+\psi_2)(x,t)
\\
\Big|\int_0^t \Big(\tilde G_{xy}(x,t-s;0)Bh(s)+&G_x(x,t-s;0)
Ah(s)\Big)\,ds \Big|\\&\le CE_0(\theta+\psi_1+\psi_2)(x,t)
\end{aligned}
\end{equation}
for $0\le t\le +\infty$, some $C>0$.
\end{lem}

\begin{proof}
The estimate on $\int_0^{t-1}$, where $G_y(x,t-s;0),\tilde
G_{xy}(x,t-s;0)$ is nonsingular, follows readily by estimates
similar to but somewhat simpler than those of Lemma
\eqref{convolutionsI}, which we therefore omit.

To bound the singular part $\int_{t-1}^t$, we integrate
\eqref{explicit} in $y$ from $0$ to $+\infty$
to obtain
\begin{equation}\label{keyrel}
\tilde G_yB + GA =
-\int_0^{+\infty}G(x,t-s;y)A_y\, dy + \int_0^{+\infty}G_s(x,t-s;y)\,
dy.
\end{equation}
Substituting in the lefthand side of \eqref{bconeq}, and integrating
by parts in $s$, we obtain
\begin{equation}\label{partsest}
\begin{aligned}
\int_{t-1}^t (\tilde G_yB + GA)h(s)\, ds
&= \int_{0}^1 \Big(\int_0^{+\infty}A_y(y)G(x,\tau ;y)\, dy \Big) h(t-\tau)\, d\tau  \\
&\quad - \int_0^1 \Big( \int_0^{+\infty}G(x,\tau;y)\, dy \Big)
h'(t-\tau)\, d\tau\\
&\quad + \Big(\int_0^{+\infty}G(x,1;y)\, dy\Big) h(t-1),
\end{aligned}
\end{equation}
which by $\int |G|  dy\le C$ is bounded by $\max_{0\le \tau\le
1}(|h|+|h'|)(t-\tau)$.

Combining this with the following more straightforward estimate (for
large $x$, $|x|>a_n^+ t$)
\begin{equation}\label{basicest}
\begin{aligned}
\Big|\int_{t-1}^{t}\tilde G_y(x,t-s;0)B h(s)\, ds\Big|&\le
\int_{0}^{1}|\tilde G_y(x,\tau ;0)|B h(t-\tau)\, d\tau \\
&\le C\max_{0\le \tau\le 1}|h(t-\tau)|
\int_{0}^{1}\tau^{-1} e^{-|x|^2/C\tau}\, d\tau\\
&=C|x|^{-2}\max_{0\le \tau\le 1}|h(t-\tau)|\\
&\quad \times
\int_{0}^{1}(|x|^2/\tau) e^{-|x|^2/C\tau}\, d\tau\\
&\le C\max_{0\le \tau\le 1}|h(t-\tau)| |x|^{-2},
\end{aligned}
\end{equation}
\begin{equation}\label{basicestGA}
\begin{aligned}
\Big|\int_{t-1}^{t}\tilde G(x,t-s;0)A h(s)\, ds\Big|&\le
\int_{0}^{1}|\tilde G(x,\tau ;0)|A h(t-\tau)\, d\tau \\
&\le C\max_{0\le \tau\le 1}|h(t-\tau)|
\int_{0}^{1}\tau^{-1/2} e^{-|x|^2/C\tau}\, d\tau\\
&\le C\max_{0\le \tau\le 1}|h(t-\tau)|
\int_{0}^{1}\tau^{-1} e^{-|x|^2/C\tau}\, d\tau\\
&\le C\max_{0\le \tau\le 1}|h(t-\tau)| |x|^{-2},
\end{aligned}
\end{equation}
and the estimate \eqref{bconeqh} for $H$ term
(thus together with \eqref{basicestGA} for $G = \tilde G + H$),
we find that the contribution from $\int_{t-1}^t$ has norm bounded
by $$ \max_{0\le \tau\le 1}(|h|+|h'|)(t-\tau)(1+|x|)^{-2} \le C
E_0(\psi_1 + \psi_2)(x,t).
$$
Combining this estimate with the one for $\int_0^{t-1}$, we obtain
the first inequality in \eqref{bconeq}. For second inequality, we
first differentiate \eqref{partsest} with respect to $x$ to get
\begin{equation}\label{partsestx}
\begin{aligned}
\int_{t-1}^t (\tilde G_{xy}B + G_xA)h(s)\, ds
&= \int_{0}^1 \Big(\int_0^{+\infty}A_y(y)G_x(x,\tau ;y)\, dy \Big) h(t-\tau)\, d\tau  \\
&\quad - \int_0^1 \Big( \int_0^{+\infty}G_x(x,\tau;y)\, dy \Big)
h'(t-\tau)\, d\tau\\
&\quad + \Big(\int_0^{+\infty}G_x(x,1;y)\, dy\Big) h(t-1),
\end{aligned}
\end{equation}
which, by $\int_{0}^1\int |G_x|  dyd\tau\le
C\int_0^1\tau^{-1/2}d\tau \le C$, is bounded by $\max_{0\le \tau\le
1}(|h|+|h'|)(t-\tau)$, similarly as above.

For the large $x$, clearly we still have
similar estimates as
\eqref{basicest} 
and \eqref{basicestGA} for $\tilde G_{xy}$ and $\tilde G_x$.
These, estimate \eqref{bconeqh} for $H_x$, and \eqref{partsestx}
yield the contribution from $\int_{t-1}^t$ as above, which together
with the estimate for $\int_0^{t-1}$ completes the proof of
\eqref{bconeq}.
\end{proof}

\subsection{Linearized stability} In this subsection, we shall give
the proof of Theorem \ref{theo-lin}. We first need the following
estimates:
\begin{lem}[\cite{MaZ4}]\label{iniconvolutions-lin}
Under the assumptions of Theorem \ref{theo-lin},
\begin{equation}\label{iniconeq-lin}
\begin{aligned}\Big|
\int_{0}^{+\infty}\tilde G(\cdot,t;y)f(y)\, dy \Big| _{L^p} &\le
C(1+t)^{-\frac 12(1-1/r)}|f|_{L^q},\\ \Big|
\int_{0}^{+\infty}H(\cdot,t;y)f(y)\, dy \Big| _{L^p} &\le Ce^{-\eta
t}|f|_{L^p},
\end{aligned}
\end{equation}
for all $t\ge 0$, some $C$, $\eta>0$, for any $1\le q\le p$ and
$f\in L^q\cap L^p$, where $1/r+1/q = 1+1/p$.
\end{lem}

\begin{lem}\label{bdryconvolutions-lin}
Under the assumptions of Theorem \ref{theo-lin}, if $|h(t)|\le
E_0(1+t)^{-1}$,
\begin{equation}
\begin{aligned}
\Big|\int_0^t \Big(\tilde G_y(x,t-s;0)Bh(s)+&G(x,t-s;0)
Ah(s)\Big)\,ds \Big|_{L^p}\\&\le CE_0(1+t)^{-\frac 12(1-1/p)}
\end{aligned}
\end{equation}
for $0\le t\le +\infty$, some $C>0$.
\end{lem}
\begin{proof} This follows at once by the boundary estimate \eqref{bconeq} and the fact that
$|(\theta + \psi_1+\psi_2)(\cdot,t)|_{L^p}\le C(1+t)^{-\frac
12(1-1/p)}$.
\end{proof}

\begin{proof}[Proof of Theorem \ref{theo-lin}]
Sufficiency of (D) for linearized stability (the main point here)
follows easily by applying the above lemmas to the following
representation for solution $U(x,t)$ of the linearized equations
\eqref{linearized-eqs}
\begin{align*}
  U(x,t)=&\int_0^\infty G(x,t;y)U_0(y)\,dy
  \\&+\int_0^t \Big(\tilde G_y(x,t-s;0)BU(0,s)+G(x,t-s;0) AU(0,s)\Big)\,ds
  \end{align*}
where $U(y,0) = U_0(y)$
and $|U(0,s)|\le C|h(s)| \le C(1+s)^{-1}$ by \eqref{inBC-U} in the
inflow case, and $|BU(0,s)|\le C|h(s)| \le C(1+s)^{-1}$ by
\eqref{outBC-U} in the outflow case, noting that
$G(x,t;0)\equiv 0$
in this case. 
Necessity follows by a much simpler argument, restricting $x$, $y$
to a bounded set and letting $t\to \infty$, noting that $G$ is given
by the ODE evolution of the spectral projection onto the finite set
of zeros of $D$ in $\Re \lambda\ge 0$, necessarily nondecaying, plus
an $O(e^{-\eta t})$ error, $\eta>0$, from which we find that
asymptotic decay implies nonexistence of any such zeros; see
Proposition 7.7 and Corollary 7.8, \cite{MaZ3} for details.
\end{proof}

\subsection{Nonlinear argument} In this subsection, we shall give
the proof of Theorem \ref{theo-nonlin}. In fact, with the above
preparations, the proof of nonlinear stability is also
straightforward.

\begin{lemma}[$H^4$ local theory] \label{lem-shorttime} Under the hypotheses of Theorem
\ref{theo-nonlin}, then, for $T$ sufficiently small depending on the
$H^4-$norm of $U_0$, there exists a unique solution $U(x,t)\in
L^\infty(0,T;H^4(x))$ of \eqref{perteqs} satisfying
\begin{equation}|U(t)|_{H^4} \le C|U_0|_{H^4}\end{equation}
for all $0\le t\le T$. \end{lemma}

\begin{proof} Short-time existence, uniqueness, and stability are described in \cite{Z2,Z4}, using a
standard (bounded high norm, contractive low norm) contraction
mapping argument. We omit the details.\end{proof}
\begin{lemma} Under the hypotheses of Theorem
\ref{theo-nonlin}, let $U\in L^\infty(0,T;H^4(x))$ satisfy
\eqref{perteqs} on $[0,T]$, and define
\begin{equation}\label{zeta2}
 \zeta(t):= \sup_{x, 0\le s \le t}\Big[
 (|U|+|U_x|)(\theta+\psi_1+\psi_2)^{-1}(x,t)
 \Big].
\end{equation}

If $\zeta(T)$ and $|U_0|_{H^4}$ are bounded by $\zeta_0$
sufficiently small, then, for some $\epsilon >0$, (i) the solution
$U$, and thus $\zeta$ extends to $[0,T+\epsilon]$, and (ii) $\zeta$
is bounded and continuous on $[0,T+\epsilon]$.
\end{lemma}

\begin{proof} Boundedness and smallness of $|U(t)|_{H^4}$ on $[0,T]$ follow by
 Proposition \ref{prop-energy-est}, provided smallness of $\zeta(T)$ and $|U_0|_{H^4}$. By Lemma \ref{lem-shorttime},
 this implies the existence, boundedness of $|U(t)|_{H^4}$ on
 $[0,T+\epsilon]$, for some $\epsilon>0$, and thus, by Lemma \ref{lem-w-EE}, boundedness and continuity of $\zeta$ on
 $[0,T+\epsilon]$.   \end{proof}

\begin{proof}[Proof of Theorem \ref{theo-nonlin}]
We shall establish:

{\it Claim.} For all $t\ge 0$ for which a solution exists with
$\zeta$ uniformly bounded by some fixed, sufficiently small
constant, there holds
\begin{equation}
\label{claim} \zeta(t) \leq C_2(E_0 + \zeta(t)^2).
\end{equation}
\medskip
{}From this result, provided $E_0 < 1/4C_2^2$, we have that by
continuous induction
 \begin{equation}
 \label{bd}
  \zeta(t) < 2C_2E_0
 \end{equation}
for all $t\geq 0$. From \eqref{bd} and the definition of $\zeta$ in
\eqref{zeta2} we then obtain the bounds of \eqref{pointwise}.
Thus, it remains only to establish the claim above.
\medskip


{\it Proof of Claim.} We must show that
$(|U|+|U_x|)(\theta+\psi_1+\psi_2)^{-1}$
is
bounded by $C(E_0 + \zeta(t)^2)$, for some $C>0$, all $0\le s\le t$,
so long as $\zeta$ remains sufficiently small.
First we need an estimate for $U(0,s)$ and $U_s(0,s)$. For the
inflow case, by boundary condition estimate \eqref{BCs-in} and by
the hypotheses on $h(s)$, we have
\begin{equation}\label{inBC-est1}|U(0,s)|\le C(h(s) + |U(0,s)|^2) \le
C(E_0(1+s)^{-1}+|U(0,s)|^2)\end{equation} from which by continuity
of $|U(0,t)|$ (Remark \ref{rem-wpoint-bounds}) and smallness of
$E_0$, we obtain a similar estimate to \eqref{bd}:
\begin{equation}\label{inBC-est2}|U(0,s)|\le CE_0(1+s)^{-1}.\end{equation}

Similarly for an estimate of $U_t(0,t)$, by taking the derivative of
\eqref{BCs-in}, we get
\begin{equation}\label{inBC-dest1}\begin{aligned}
|U_s(0,s)|&\le C(h'(s) +
|U||U_s|(0,s)) \\
&\le C(E_0(1+s)^{-1}+|U(0,s)||U_s(0,s)|)
\end{aligned}\end{equation} which by
the same argument as above 
yields
\begin{equation}\label{inBC-dest2}|U_s(0,s)|\le CE_0(1+s)^{-1}.\end{equation}

Next, for the outflow case with boundary condition \eqref{BCs-out},
we have
\begin{equation}\label{outBC-dest2}\begin{aligned}|BU(0,s)|&\le
CE_0(1+s)^{-1} + \cO(|U(0,s)|^2)\\|(BU)_s(0,s)|&\le CE_0(1+s)^{-1} +
\cO(|U||U_s|(0,s)).\end{aligned}\end{equation}
%

Now by \eqref{zeta2}, we have for all $t\ge 0$ and some $C>0$ that
\begin{equation}\label{ubounds}
\begin{aligned}
|U(x,t)| + |U_x(x,t)| &\le \zeta (t)(\theta +\psi_1+\psi_2)(x,t),\\
\end{aligned}
\end{equation}
and therefore
\begin{equation}\label{Nbounds}
\begin{aligned}
|Q(U,U_y)(y,s)|&\le C\zeta(t)^2 \Psi(y,s) \\
|\Pi_1Q(U,U_y)_y(y,s)|&\le C\zeta(t)^2 \Psi(y,s)
\end{aligned}
\end{equation}
with $\Psi = (\theta +\psi_1+\psi_2)^2$ as defined in
\eqref{source}, for $0\le s\le t$.

As an estimate for $U(x,t)$, we use the representation (\ref{u}) of
$U(x,t)$:
\begin{align*}
  |U&(x,t)|=\Big|\int_0^\infty G(x,t;y)U_0(y)\,dy\Big|
  \\&+\Big|\int_0^t (\tilde G_y(x,t-s;0)BU(0,s)+G(x,t-s;0)AU(0,s))\,ds\Big|\\
  &+ \Big|\int^t_0 \int^\infty_{0} H(x,t-s;y)
  \Pi_1Q(U,U_y)_y(y,s)\,dy\,ds\Big|\\
  &+\Big|\int^t_0 \int^\infty_{0} \tilde G_y(x,t-s;y)
  \Pi_2Q(U,U_y)(y,s)\,dy\,ds\Big|,
\end{align*}
where by applying Lemmas \ref{iniconvolutions}-\ref{convolutionsII}
together with \eqref{Nbounds}, we have
\begin{equation}
\begin{aligned}
\Big|\int_0^\infty &G(x,t;y)g(y)\,dy\Big| \\& \le
  E_0 \int^\infty_{0} (|\tilde G(x,t;y)| +
  |H(x,t;y)|)(1+|y|)^{-3/2}\,dy\\&\le CE_0(\theta + \psi_1+\psi_2)(x,t)
\end{aligned}
\end{equation}
\begin{equation}
\begin{aligned}
\Big|\int^t_0 \int^\infty_{0} &\tilde G_y(x,t-s;y)
  Q(U,U_y)(y,s)\,dy\,ds\Big|\\&\le C\zeta(t)^2\int^t_0 \int^\infty_{0} |\tilde
  G_y(x,t-s;y)|\Psi(y,s)\,dy\,ds\\&\le
C\zeta(t)^2(\theta + \psi_1+\psi_2)(x,t)
\end{aligned}
\end{equation}
\begin{equation}
\begin{aligned}
\Big|\int^t_0 \int^\infty_{0} &H(x,t-s;y)
  \Pi_1Q(U,U_y)_y(y,s)\,dy\,ds\Big|\\&\le C\zeta(t)^2\int^t_0 \int^\infty_{0} H(x,t-s;y)(\theta +\psi_1+\psi_2)^2\,dy\,ds
  \\&\le C\zeta(t)^2\int^t_0 \int^\infty_{0} H(x,t-s;y)\Upsilon(y,s)\,dy\,ds
  \\&\le C\zeta(t)^2(\theta + \psi_1+\psi_2)(x,t)
\end{aligned}
\end{equation}
and, for the boundary term, we apply the estimate
\eqref{inBC-est2}
 and Lemma \ref{boundaryconvolutionII}, yielding
\begin{equation}
\begin{aligned}
\Big|\int_0^t (\tilde
G_y(x,t-s;0)&BU(0,s)+G(x,t-s;0)AU(0,s))\,ds\Big|
  \\&\le C(E_0+\zeta(t)^2)(\theta + \psi_1+\psi_2)(x,t)
\end{aligned}
\end{equation} for the inflow. Whereas, for the outflow case, noting that 
$G(x,t-s;0)\equiv 0$
in the outflow case, we apply the estimate \eqref{outBC-dest2},
\eqref{ubounds} and Lemma \ref{boundaryconvolutionII} to give the
same estimate as
above, yielding 
\begin{align*}
\Big|\int_0^t &\tilde G_y(x,t-s;0)BU(0,s)\,ds\Big|
 \le C(E_0+\zeta(t)^2)(\theta + \psi_1+\psi_2)(x,t) 
\end{align*}
where we used \eqref{ubounds} for $|U(0,s)|\le
\zeta(t)(1+t)^{-1/2}$.

Therefore, combining the above estimates, we obtain
\begin{equation}\label{point-estU}|U(x,t)|(\theta+\psi_1+\psi_2)^{-1}(x,t)\le C(E_0 +
\zeta(t)^2).\end{equation}

To derive the same estimate for $|U_x (x,t)|$, we first obtain by
using Proposition \ref{prop-energy-est},
\begin{align*}|U(t)|_{H^4}^2 &\le Ce^{-\theta t}|U_0|_{H^4}^2 + C
\int_0^t e^{-\theta(t-\tau)}\Big[|U(\tau)|_{L^2}^2 +
\cB_h(\tau)\Big]d\tau
\\&\le C(E_0 + \zeta(t)^2)t^{-1/2},\end{align*} where
$\cB_h$ is the boundary function defined in Proposition
\ref{prop-energy-est},
and thus by the one dimensional Sobolev embedding:
$|U(t)|_{W^{3,\infty}} \le C|U(t)|_{H^4}$,
\begin{equation}\label{Qbounds-der}\begin{aligned}
|Q(U,U_x)_x|&\le C(\zeta^2(t)+4C^2 E_0^2)\Upsilon\\
|Q(U,U_x)_{xx}|&\le C(\zeta^2(t)+4C^2 E_0^2)\Upsilon
\end{aligned} \end{equation}
where $\Upsilon = t^{-1/4}(\theta+\psi_1+\psi_2)$.

Now again applying Lemmas
\ref{iniconvolutions}-\ref{boundaryconvolutionII} together with
\eqref{Qbounds-der},
\eqref{inBC-dest2}, and \eqref{outBC-dest2},
we have obtained the desired estimate, that is, bounded by
$(\zeta^2(t)+C E_0)(\theta+\psi_1+\psi_2)(x,t)$, for most terms in
the formulation \eqref{ux} of $U_x(x,t)$, except one boundary term:
\begin{align*}\int_0^t&H(x,t-s;0)|\Pi_1Q(U,U_y)_y(0,s)|\,ds, \end{align*}
which is bounded by $CE_0(\psi_1+\psi_2)(x,t)$ by using
\eqref{newqbounds}, \eqref{ubounds}, and Lemma
\ref{boundaryconvolutionI}, and noting
that
 \begin{align*}|\Pi_1Q(U,U_y)_y(0,s)| &\le \zeta (t)|h(s)|(\theta +\psi_1+\psi_2)(0,s)
\le C\zeta (t)|h(s)|.
\end{align*}

 Therefore, together with \eqref{point-estU}, we have obtained
\begin{equation}(|U(x,t)|+|U_x(x,t)|)(\theta+\psi_1+\psi_2)^{-1}(x,t)\le C(E_0 +
\zeta(t)^2)\end{equation} as claimed, which completes the proof of
Theorem \ref{theo-nonlin}.
\end{proof}


\end{document}